\documentclass[aop,MSNbibl,seceqn,citesort,dvips]{arximspdf}
\usepackage{mathrsfs}

\doi{10.1214/11-AOP734} 
\volume{41}
\issue{3B}
\pubyear{2013}
\firstpage{2279}
\lastpage{2375}

\makeatletter

\newcommand{\nc}{\wt}
\newcommand{\ol}{\bar}

\newcommand{\bv}{{\mathbf{v}}}
\newcommand{\al}{\alpha}

\newcommand{\f}[1]{{\mathbf{#1}}}
\newcommand{\ff}[1]{\bolds{#1}}

\newcommand{\bb}{\mathbb}
\newcommand{\cal}{\mathcal}
\newcommand{\scr}{\mathscr}
\newcommand{\fra}{\mathfrak}
\newcommand{\wh}{\widehat}
\newcommand{\wt}{\widetilde}

\newcommand{\me}{\mathrm{e}}
\newcommand{\ii}{\mathrm{i}}
\newcommand{\dd}{\mathrm{d}}

\newcommand{\col}{\dvtx}
\newcommand{\st}{\col}
\newcommand{\deq}{:=}
\newcommand{\eqd}{=:}

\renewcommand{\epsilon}{\varepsilon}

\renewcommand{\P}{\mathbb{P}}
\newcommand{\E}{\mathbb{E}}
\newcommand{\R}{\mathbb{R}}
\newcommand{\C}{\mathbb{C}}
\newcommand{\N}{\mathbb{N}}
\newcommand{\IE}{\mathbb{I} \mathbb{E}}

\newcommand{\tr}{\operatorname{Tr}}
\newcommand{\re}{\operatorname{Re}}
\newcommand{\im}{\operatorname{Im}}

\newtheorem{theorem}{Theorem}[section]
\newtheorem{lemma}[theorem]{Lemma}
\newtheorem{corollary}[theorem]{Corollary}
\newtheorem{proposition}[theorem]{Proposition}

\newproclaim{definition}[theorem]{Definition}

\newproclaim{remark}[theorem]{Remark}

\makeatother

\begin{document}
\begin{frontmatter}

\title{Spectral statistics of Erd{\H o}s--R\'enyi graphs I: Local~semicircle law}
\runtitle{Erd{\H o}s--R\'enyi graphs}

\begin{aug}
\author[A]{\fnms{L\'aszl\'o} \snm{Erd\H os}\thanksref{t1}\ead[label=e1]{lerdos@math.lmu.de}},
\author[B]{\fnms{Antti} \snm{Knowles}\corref{}\thanksref{t2}\ead[label=e2]{knowles@math.harvard.edu}},
\author[B]{\fnms{Horng-Tzer} \snm{Yau}\thanksref{t2,t3}\ead[label=e3]{htyau@math.harvard.edu}}\\
\and
\author[C]{\fnms{Jun} \snm{Yin}\thanksref{t4}\ead[label=e4]{jyin@math.wisc.edu}}
\runauthor{Erd\H os, Knowles, Yau and Yin}
\affiliation{University of Munich, Harvard University, Harvard
University and University~of~Wisconsin}
\address[A]{L. Erd\H os\\
Institute of Mathematics\\
University of Munich \\
Theresienstrasse 39\\
D-80333 Munich\\
Germany \\
\printead{e1}} 
\address[B]{A. Knowles\\
H.-T. Yau\\
Department of Mathematics\\
Harvard University\\
Cambridge, Massachusetts 02138\\
USA \\
\printead{e2}\\
\hphantom{E-mail: }\printead*{e3}}
\address[C]{J. Yin\\
Department of Mathematics\\
University of Wisconsin \\
Madison, Wisconsin 53706\\
USA \\
\printead{e4}}
\end{aug}

\thankstext{t1}{Supported in part
by SFB-TR 12 Grant of the German Research Council.}

\thankstext{t2}{Supported in part by NSF Grant DMS-07-57425.}

\thankstext{t3}{Supported in part by NSF Grant DMS-08-04279.}

\thankstext{t4}{Supported in part by NSF Grant DMS-10-01655.}

\received{\smonth{4} \syear{2011}}
\revised{\smonth{11} \syear{2011}}

%
\begin{abstract}
We consider the ensemble of adjacency matrices of Erd{\H o}s--R\'enyi
random graphs, that is, graphs on $N$ vertices where every edge is
chosen independently and with probability $p \equiv p(N)$. We rescale
the matrix so that its bulk eigenvalues are of order one. We prove
that, as long as $p N \to\infty$ (with a speed at least logarithmic in
$N$), the density of eigenvalues of the Erd{\H o}s--R\'enyi ensemble is
given by the Wigner semicircle law for spectral windows of length
larger than $N^{-1}$ (up to logarithmic corrections). As a consequence,
all eigenvectors are proved to be completely delocalized in the sense
that the $\ell^\infty$-norms of the $\ell^2$-normalized eigenvectors
are at most of order $N^{-1/2}$ with a very high probability. The
estimates in this paper will be used in the companion paper [Spectral
statistics of Erd\H{o}s--R\'enyi graphs II: Eigenvalue spacing and the
extreme eigenvalues (2011) Preprint] to prove the universality of
eigenvalue distributions both in the bulk and at the spectral edges
under the further restriction that $p N \gg N^{2/3}$.
\end{abstract}

%

\pdfkeywords{15B52, 82B44, Erdos--Renyi graphs,
local semicircle law, density of states}

\end{frontmatter}

\section{Introduction}

The universality of random matrices has been a central subject since
the pioneering work of Wigner~\cite{W},
Gaudin~\cite{G}, Mehta~\cite{M}
and Dyson~\cite{Dy}. The problem can roughly be divided
into the bulk universality in the interior of the spectrum and the edge
universality near the spectral edge. The bulk
and edge universalities for invariant ensembles have been extensively
studied; see, for example,~\cite{BI,DKMVZ1,DKMVZ2,PS}
and~\cite{AGZ,De1,De2} for a review. A key contributing factor to
the progress in the study of invariant ensembles is
the existence of explicit formulas for the joint density function of
the eigenvalues.
There is no such explicit formula for noninvariant ensembles and,
hence, our understanding of them is much more limited.
The most prominent examples for noninvariant ensembles
are the Wigner matrices with i.i.d. non-Gaussian matrix elements.
The edge universality of Wigner matrices can be proved via the moment
method and its various generalizations; see,
for example,~\cite{SS,Sosh,So1}.
The bulk universality for general classes of Wigner matrices was listed in
Mehta's book~\cite{M} as Conjectures 1.2.1 and 1.2.2 on page 7.
We shall refer to these two conjectures collectively as the
Wigner--Dyson--Gaudin--Mehta conjecture, in recognition of the pioneering
works of Wigner, Dyson, Gaudin and Mehta listed above. It remained
unsolved until very recently. This is mainly due to the fact that all
existing methods on local eigenvalue statistics
depended on explicit formulas, which are not available for Wigner matrices.
In a series of papers \cite
{ESY1,ESY2,ESY3,EPRSY,ESY4,ESYY,EYY,EYY2,EYYrigidity}, a new approach
to understanding the local
eigenvalue statistics was developed.
In particular, it led to the first proof~\cite{EPRSY} of the
Wigner--Dyson--Gaudin--Mehta conjecture
for Hermitian Wigner matrices whose entries have smooth distributions.
This approach is based on three basic ingredients:
(1) a local semicircle law---a precise estimate of the local
eigenvalue density
down to energy scales containing around $(\log N)^C$ eigenvalues; (2)
the eigenvalue distribution of Gaussian
divisible ensembles via an estimate on the rate of decay to local
equilibrium of the Dyson Brownian motion~\cite{Dy};
(3) a density argument which shows that for any probability
distribution of the matrix elements there exists a Gaussian
divisible distribution such that the two associated Wigner ensembles
have identical local eigenvalue statistics down to
the scale $1/N$. Furthermore, the edge universality can also be
obtained by some modifications of (1) and (3)
\cite{EYYrigidity}. The class of ensembles to which this method
applies is extremely general; in particular, it
includes any (generalized) Wigner matrices under the sole assumption
that the distributions of the matrix elements
have a uniform subexponential decay. We remark that the universality of
Wigner matrices, under certain restrictions on
the distribution of the matrix entries, was also established in \cite
{TV,TV2}. We shall discuss these results in the
companion paper~\cite{EKYY2}.

In this paper and its companion~\cite{EKYY2},
we extend the approach (1)--(3) to cover a class of sparse matrices.
This class includes the Erd{\H o}s--R\'enyi
matrices, which we now introduce.
Symmetric $N\times N$ matrices with $0$--$1$ entries arise naturally as
adjacency matrices
of graphs on $N$ vertices. Since every nonoriented graph can be
uniquely characterized by its adjacency matrix, we
shall from now talk about matrix ensembles
(with $0$--$1$ entries) and graph ensembles interchangeably. We shall
always normalize the matrices so that
their spectra typically lie in an interval of length of order one.
One common random graph ensemble is the Erd{\H o}s--R\'enyi graph
\cite{ER1,ER2}. In it each edge is chosen independently and
with\vadjust{\goodbreak}
probability $p \equiv p(N)$. Since each row and
column of the adjacency matrix has typically $pN$ nonzero entries, it
is sparse as long as $p\ll1$. We shall refer to
$pN$ as the sparseness parameter of the matrix.\looseness=-1

Our goal
in this paper, and in its companion~\cite{EKYY2},
is to establish both the bulk and edge universalities for the
Erd{\H o}s--R\'enyi ensemble under the restriction
$p N \gg N^{2/3}$. In other words, we prove that
the eigenvalue gap distributions in the bulk and near the edges are
given by those of the Gaussian Orthogonal Ensemble
(GOE)
provided that $p N \gg N^{2/3}$. We remark that the law of the
Erd{\H o}s--R\'enyi ensemble is even more singular
than that of the Bernoulli Wigner matrices, since the matrix elements
are highly concentrated around $0$. Another way
of expressing the singular nature of the Erd{\H o}s--R\'enyi ensemble is
to say that the moments of the matrix entries
decay much more slowly than in the case of Wigner matrices.

The matrix elements of the Erd{\H o}s--R\'enyi ensemble take on the
values $0$ and~$1$. Hence, they do not satisfy the
mean zero condition which typically appears in the random matrix
literature. Due to the nonzero mean of the entries,
the largest eigenvalue of the Erd{\H o}s--R\'enyi ensemble is very large
and far away from the rest of the spectrum,
which by our normalization lies in the interval $[-2, 2]$. By the edge
universality of the Erd{\H o}s--R\'enyi ensemble,
we therefore mean that the probability distribution of the second
largest eigenvalue is given by the distribution of
the largest eigenvalue of the GOE, which is the well-known Tracy--Widom
distribution.

As the first step of the general strategy to establish universality, we
shall prove a local semicircle law,
Theorem~\ref{LSCTHMA}, stating that the eigenvalue distribution of
the Erd{\H o}s--R\'enyi ensemble in any
spectral window of size $\eta$ containing on average $N \eta\sim
(\log N)^C$ eigenvalues is given by the Wigner semicircle law. Theorem
\ref{LSCTHMA} is valid in the bulk and at the edges as long as the
parameter $p \equiv p(N)$ satisfies $p N \to
\infty$ with a rate at least logarithmic in~$N$.
Similar results but for much larger spectral windows [of lengths at
least $\eta\sim(pN)^{-1/10} $] were recently proved in~\cite{TVW}.

We note that the semicircle law for Wigner matrices in spectral windows
of size $\eta\sim N^{-1/2}$ has been known for some time~\cite{BMT,GZ}.
The semicircle law in the smallest possible spectral window (of
size $\eta\gtrsim N^{-1} $ in the bulk) was established in \cite
{ESY2,ESY3}. This estimate, referred to as the local semicircle law,
has become a fundamental tool in the proofs of the universality of
random matrices in~\cite{EPRSY} as well as in the subsequent works
\cite{TV,ESY4}. The local semicircle law in~\cite{ESY2,ESY3} is
optimal in terms of the range of $\eta$, but
the error estimates, of order $(N \eta)^{-1/2}$ in the bulk and with a
coefficient deteriorating near the spectral edges, were not optimal.
Optimal error estimates, uniform throughout the entire spectrum and
valid for more general classes of Wigner matrices, were obtained in
\cite{EYYrigidity}. The local semicircle law proved in this paper can
also be viewed as a generalization of the results
in~\cite{EYYrigidity} in two unrelated directions: (a) the laws of the
matrix entries are much more singular, and
(b) the matrix entries have nonzero mean.

Besides eigenvalues, eigenvectors also play a fundamental role in the
theory of
random matrices. One important motivation for their study is that
random matrices
can be viewed as mean-field approximations of random Schr\"odinger
operators where delocalization of eigenfunctions
is a key signature for the metallic or conducting phase. Another
question about eigenvectors of random graphs is the
size of their nodal domains, which can studied using delocalization
bounds~\cite{DLL}. It was first proved in
\cite{ESY1} that eigenvectors for Wigner matrices are completely
delocalized, partly motivated
by a conjecture of T. Spencer. The method was refined in \cite
{ESY2,ESY3}, and was also adapted in~\cite{TV,TVW}.
The key observation behind the proof is that the delocalization
estimate for eigenvectors
follows from the local semicircle law
provided that the spectral windows can be made sufficiently small.
Thus, Theorem~\ref{LSCTHMA} also implies, with
$\bv_\alpha$
denoting the $\ell^2$-normalized eigenvectors, that
%
%
\begin{equation}\label{ev}
\max\{{\| \bv_\alpha\|_\infty\st1 \leq\alpha\leq N }
\}
\leq\frac{(\log N)^{C}}{\sqrt{N}}
\end{equation}
holds with exponentially high probability with some constant $C$
(Theorem~\ref{theoremdelocalization}).
This establishes the complete delocalization of all eigenvectors as
long as the sparseness parameter $p N$ increases at
least logarithmically in $N$. In particular, this result gives the
optimal answer to a question posed in Section 3.3 of
\cite{DLL}, asking whether $\|\f v_\alpha\|_\infty\leq
N^{-1/2 +
o(1)}$ holds for all eigenvectors $\f v_\alpha$. In
fact, this question was originally posed for fixed $p$, but our result
shows that the bound conjectured in~\cite{DLL}
holds even for $p \geq(\log N)^C N^{-1}$.
It was recently proved in~\cite{TVW} that $\| \bv_\al\|_\infty\leq
(p N)^{-1/2}$ away from the spectral edge; some
earlier results were obtained in~\cite{DP}. These results established
only the lower bound $pN$ on the localization
length; the complete
delocalization (\ref{ev}) corresponds to the optimal localization
length of order $N$.

Our main result on the bulk and edge universalities will require a
further condition
%
%
\begin{equation}\label{res}
p N \gg N^{2/3}.
\end{equation}
This and related issues will be discussed in the second paper~\cite{EKYY2}.

\section{Definitions and results}\label{secdef}

We begin this section by introducing a class of $N \times N$ sparse
random matrices $A \equiv A_N$.
Here $N$ is a large
parameter. (Throughout the following we shall often refrain from
explicitly indicating $N$-dependence.)

The motivating example is the \textit{Erd\H{o}s--R\'enyi matrix}, or the
adjacency matrix of the \textit{Erd\H{o}s--R\'enyi random graph}. Its
entries are independent (up to the constraint that
the matrix be symmetric), and equal to $1$ with probability $p$ and $0$
with probability $1 - p$. For our purposes it
is convenient to replace $p$ with the new parameter $q \equiv q(N)$,
defined through $p = q^2 / N$. Moreover, we
rescale the matrix in such a way that its bulk eigenvalues typically
lie in an interval of size of order one.\vadjust{\goodbreak}

Thus, we are led to the following definition. Let $A = (a_{ij})$ be the
symmetric $N\times N$ matrix whose entries
$a_{ij}$ are independent
(up to the symmetry constraint $a_{ij}=a_{ji}$) and each element is
distributed according to
%
%
\begin{equation}\label{sparsedef}
a_{ij} =
\frac{\gamma}{q}
\cases{
1, &\quad with probability $\dfrac{q^2}{N}$,
\vspace*{2pt}\cr
0, &\quad with probability $1 - \dfrac{q^2}{N}$.}
\end{equation}
Here $\gamma\deq(1 - q^2 / N)^{-1/2}$ is a scaling introduced for
convenience. The parameter $q \leq N^{1/2}$
expresses the sparseness of the matrix; it may depend on $N$.
Since $A$ typically has $q^2 N$ nonvanishing entries, we find that if
$q \ll N^{1/2}$, then the matrix is sparse.

We extract the mean of each matrix entry and write
\[
A = H + \gamma q | \f e \rangle\langle\f e |,
\]
where the entries of $H$ (given by $h_{ij} = a_{ij} - \gamma q/N$) have
mean zero, and we defined the vector
%
%
\begin{equation} \label{defofbe}
\f e \equiv\f e_N \deq\frac{1}{\sqrt{N}} (1,\ldots,
1)^T.
\end{equation}
(As above, we often neglect the subscript $N$ of $\f e$; the precise
value of this subscript will always be clear from
the context.) Here we use the notation $| \f e \rangle\langle\f e |$ to
denote the orthogonal projection onto $\f e$, that is,
$(| \f e \rangle\langle\f e |)_{ij} \deq N^{-1}$.

One readily finds that the matrix elements of $H$ satisfy the moment bounds
%
%
\begin{equation}
\E h_{ij}^2 = \frac{1}{N},\qquad
\E\vert h_{ij} \vert^p \leq\frac{1}{N q^{p - 2}},
\end{equation}
where $p \geq2$.

More generally, we consider the following class of random matrices
with noncentered entries characterized by two parameters $q$ and $f$,
which may be $N$-dependent. The parameter $q$ expresses how singular
the distribution of $h_{ij}$ is; in particular, it expresses
the sparseness of $A$ for the special case (\ref{sparsedef}). The
parameter $f$ determines the nonzero expectation value
of the matrix elements.

Throughout the following we shall make use of a (possibly
$N$-dependent) quantity $\xi\equiv\xi_N$ satisfying
%
%
\begin{equation} \label{boundsonxi}
1 + a_0 \leq\xi\leq A_0 \log\log N
\end{equation}
for some fixed positive constants $a_0 > 0$ and $A_0 \geq10$. The
parameter $\xi$ will be used as an exponent in
logarithmic corrections as well as probability estimates.
%
%
\begin{definition}[($H$)] \label{definitionofH}
Fix a parameter $\xi\equiv\xi_N$ satisfying (\ref{boundsonxi}).
We consider $N \times N$ random matrices $H = (h_{ij})$ whose entries
are real and independent up to the symmetry
constraint\vadjust{\goodbreak} $h_{ij} = h_{ji}$. We assume that the elements of $H$
satisfy the moment conditions
%
%
\begin{equation} \label{momentconditions}
\E h_{ij} = 0,\qquad \E\vert h_{ij} \vert^2 = \frac{1}{N}
,\qquad
\E\vert h_{ij} \vert^p \leq\frac{C^p}{N q^{p - 2}}
\end{equation}
for $1 \leq i,j \leq N$ and $3 \leq p \leq(\log N)^{A_0 \log\log N}$,
where $C$ is a positive constant. Here $q \equiv
q(N)$ satisfies
%
%
\begin{equation} \label{lowerboundond}
(\log N)^{3 \xi} \leq q \leq C N^{1/2}
\end{equation}
for some positive constant $C$.
\end{definition}

Note that the entries of $H$ exhibit a slow decay of moments. The
variance is of order $N^{-1}$, but higher moments decay at a rate
proportional to inverse powers of $q$ and not $N^{1/2}$. Thus, unlike
for Wigner matrices, the entries of sparse matrices satisfying
Definition~\ref{definitionofH} do not have a natural scale. (The
entries of a Wigner matrix live on the scale $N^{-1/2}$, which means
that the high moments decay at a rate proportional to inverse powers of
$N^{1/2}$. See Remark~\ref{remarkWigner} below for a more precise statement.)
%
%
\begin{definition}[($A$)] \label{definitionofA}
Let $H$ satisfy Definition~\ref{definitionofH}. Define
%
%
\begin{equation} \label{A=H+fee}
A \deq H + f | \f e \rangle\langle\f e |,
\end{equation}
where $f \equiv f(N)$ is a deterministic number that satisfies
%
%
\begin{equation} \label{upperboundonf}
0 \leq f \leq N^{C}
\end{equation}
for some constant $C > 0$.
\end{definition}
%
%
\begin{remark}
For definiteness, and bearing the Erd\H{o}s--R\'enyi matrix in mind, we
restrict ourselves to real symmetric matrices
satisfying Definition~\ref{definitionofA}. However, our proof
applies equally to complex Hermitian sparse matrices.
\end{remark}
%
%
\begin{remark}
To simplify the presentation, we assume that all matrix elements of $H$
have identical variance $1/N$. As in~\cite{EYY}, Section 5, one may,
however, easily generalize this condition and require that the
variances be bounded by $C/N$ and their column sums (hence, also the
row sums) be equal to $1$. Thus, one may, for instance, consider
Erd\H{o}s--R\'enyi graphs in which a vertex cannot link to itself
(i.e., the diagonal elements of $A$ vanish).
\end{remark}
%
%
\begin{remark} \label{remarkWigner}
In particular, we may take $H$ to be a Wigner matrix whose entries have
subexponential decay,
\[
\P({N^{1/2} \vert h_{ij} \vert\geq x}) \leq C \exp
(-x^{1/\theta})
\]
for some positive constants $\theta$ and $C$. Indeed, in this case we get
\[
\E h_{ij} = 0,\qquad \E\vert h_{ij} \vert^2 = \frac
{1}{N},\qquad
\E\vert h_{ij} \vert^p \leq C \frac{(\theta p)^{\theta
p}}{N^{p/2}}\qquad (p \geq3).\vadjust{\goodbreak}
\]
Now we choose
\[
q \deq N^{1/2} ({\theta(\log N)^{ A_0 \log\log N}}
)^{-\theta}.
\]
Since $q^{-1} \leq(\log N)^{C \log\log N} N^{-1/2}$, we find that all
factors $q^{-1}$ in error estimates such as (\ref{scm}) and (\ref
{Gijestimate}) below may be replaced with $N^{-1/2}$ at the expense of
a larger exponent in the preceding logarithmic factors. In fact, using
Lemma~\ref{lemmamsc} below, it is easy to see that in this case all
terms depending on $q$ in estimates such as (\ref{scm}) and (\ref
{Gijestimate}) may dropped, as they are bounded by the other error
terms. In particular, Theorem~\ref{LSCTHM} generalizes Theorem 2.1 of
\cite{EYYrigidity}.\vspace*{-2pt}
\end{remark}

We shall frequently have to deal with events of very high probability,
for which the following definition is useful. It
is characterized by two positive parameters, $\xi$~and $\nu$, where
$\xi$ is subject to (\ref{boundsonxi}).\vspace*{-2pt}
%
%
\begin{definition}[(High probability events)]
We say that an $N$-dependent event $\Omega$ holds with \textit{$(\xi,
\nu)$-high probability} if
%
%
\begin{equation} \label{highprob}
\P(\Omega^c) \leq\me^{-\fra\nu(\log N)^\xi}
\end{equation}
for $N \geq N_0(\nu, a_0, A_0)$.

Similarly, for a given event $\Omega_0$, we say that $\Omega$
\textit{holds with $(\xi, \nu)$-high probability on
$\Omega_0$} if
\[
\P(\Omega_0 \cap\Omega^c) \leq\me^{- \nu(\log N)^\xi}
\]
for $N \geq N_0(\nu, a_0, A_0)$.\vspace*{-2pt}
\end{definition}
%
%
\begin{remark}
In the following we shall not keep track of the explicit value
of~$\nu$; in fact, we allow $\nu$ to decrease from one line to another
without introducing a new notation. It will be clear from the proof
that such reductions of $\nu$ occur only at a few, finitely many steps.
Hence, all of our results will hold for $\nu\leq \nu_0$, where $\nu_0$
depends only on the constants $C$ in Definition~\ref{definitionofH} and
the parameter $\Sigma$ in (\ref{definitionD}) below. (In particular,
$\nu$ is independent of $\xi$.)\vspace*{-2pt}
\end{remark}

We shall use $C$ and $c$ to denote generic positive constants which may
only depend on the constants in assumptions such
as (\ref{boundsonxi}) and (\ref{momentconditions}).
Typically, $C$ denotes a large constant and $c$ a small constant.
Note that the fundamental large parameter of our model is $N$, and the
notation $\gg, \ll, O(\cdot), o(\cdot)$ always
refers to the limit $N \to\infty$. Here $a \ll b$ means $a = o(b)$.
We write $a \sim b$ for $C^{-1} a \leq b \leq C a$.

We now list our results. We introduce the spectral parameter
\[
z = E + \ii\eta,
\]
where $E \in\R$ and $\eta> 0$. Let $\Sigma\geq3$ be a fixed but
arbitrary constant and define the domain
%
%
\begin{equation} \label{definitionD}
D \deq\{{z \in\C\st\vert E \vert\leq\Sigma, 0 <
\eta
\leq3}\}.\vadjust{\goodbreak}
\end{equation}
We define the density of the semicircle law
%
%
\begin{equation} \label{defrhosc}
\varrho_{\mathrm{sc}}(x) \deq\frac{1}{2 \pi} \sqrt{[4 - x^2]_+},
\end{equation}
and, for $\im z > 0$, its Stieltjes transform
%
%
\begin{equation} \label{defmsc}
m_{\mathrm{sc}}(z) \deq\int_\R\frac{\varrho_{\mathrm{sc}}(x)}{x - z} \,\dd x.
\end{equation}
The Stieltjes transform $m_{\mathrm{sc}}(z) \equiv m_{\mathrm{sc}}$ may also be
characterized as the unique solution of
%
%
\begin{equation} \label{identityofmsc}
m_{\mathrm{sc}} + \frac{1}{z + m_{\mathrm{sc}}} = 0
\end{equation}
satisfying $\im m_{\mathrm{sc}}(z) > 0$ for $\im z > 0$. We define the resolvent
\[
G(z) \deq(H - z)^{-1}
\]
as well as the Stieltjes transform of the empirical eigenvalue density
\[
m(z) \deq\frac{1}{N} \tr G(z).
\]
For $x \in\R$ we define the distance $\kappa_x$ to the spectral edge through
%
%
\begin{equation} \label{defkappa}
\kappa_x \deq\bigl\vert\vert x \vert- 2 \bigr\vert.
\end{equation}

%
\begin{theorem}[(Local semicircle law for $H$)] \label{LSCTHM}
There are universal constants $C_1, C_2 > 0$ such that the following holds.
Suppose that $H$ satisfies Definition~\ref{definitionofH}.
Moreover, assume that
%
%
\begin{equation} \label{assumptionsforSLSC}
\xi= \frac{A_0 (1 + o(1))}{2} \log\log N,\qquad q \geq
(\log N)^{C_1 \xi}.
\end{equation}
Then there is a constant $\nu> 0$, depending on $A_0$, $\Sigma$ and
the constants $C$ in (\ref{momentconditions}) and
(\ref{lowerboundond}), such that the following holds.

We have the local semicircle law: the event
%
%
\begin{equation}\label{scm}
\bigcap_{z \in D} \biggl\{{\vert m(z) - m_{\mathrm{sc}}(z) \vert
\leq(\log N)^{C_2 \xi} \biggl({\min\biggl\{{\frac{1}{q^2 \sqrt
{\kappa_E+\eta}}, \frac{1}{q}}\biggr\} + \frac{1}{N \eta}}
\biggr)}\biggr\}\hspace*{-35pt}
\end{equation}
holds with $(\xi,\nu)$-high probability.
Moreover, we have the following estimate on the individual matrix
elements of $G$. The event
%
%
\begin{eqnarray}\label{Gijestimate}
&&\bigcap_{z \in D} \Biggl\{\max_{1 \leq i,j \leq N} \vert
G_{ij}(z) - \delta_{ij} m_{\mathrm{sc}}(z) \vert\nonumber\\[-8pt]\\[-8pt]
&&\hspace*{0pt}\qquad\leq(\log N)^{C_2 \xi}
\Biggl({\frac{1}{q} + \sqrt{\frac{\im m_{\mathrm{sc}}(z)}{N \eta}} + \frac
{1}{N \eta}}\Biggr)\Biggr\}\nonumber
\end{eqnarray}
holds with $(\xi,\nu)$-high probability.
\end{theorem}

The results of Theorem~\ref{LSCTHM} may be interpreted as follows.
Consider first the bulk, that is, $\kappa_E \geq c > 0$. Then Theorem
\ref{LSCTHM} states roughly that
%
%
\begin{equation} \label{roughineq1}
\vert m(z) - m_{\mathrm{sc}}(z) \vert\lesssim\frac{1}{q^2}
+ \frac{1}{N \eta},\qquad
\vert G_{ij}(z) - \delta_{ij} m_{\mathrm{sc}}(z) \vert\lesssim
\frac{1}{q} + \frac{1}{\sqrt{N \eta}},\hspace*{-28pt}
\end{equation}
up to logarithmic factors.
Since $\vert m_{\mathrm{sc}}(z) \vert\sim1$,
both estimates are stable in a sense that they identify the leading
order terms of $m$ and $G_{ii}$ down to the optimal scale $\eta\gtrsim N^{-1}$.
Note that choosing $\eta\lesssim N^{-1}$ in
\[
\im m(z) = \frac{1}{N\eta}\sum_{\alpha} \frac{\eta^2}{(E -
\lambda_\alpha)^2 + \eta^2}
\]
allows one to resolve individual eigenvalues $\lambda_\al$ of $H$.
Therefore, below the scale $\eta\lesssim N^{-1}$ the quantities $m$
and $G_{ii}$
become strongly fluctuating and these fluctuations are larger than the
main term. In the regime $\eta\geq\break(\log N)^{C \xi} N^{-1}$ in which
the fluctuations are smaller than the main term, a spectral window of
size $\eta$ contains at least $(\log N)^{C \xi}$ eigenvalues, hence,
an averaging takes place.

The factor $1/q$ on the right-hand side of the second inequality of
(\ref{roughineq1}) arises from the strong fluctuations of the matrix
entries $h_{ij}$, which take on values of size $q^{-1}$ with
probability of order $q^2 N^{-1}$. Indeed, it is apparent from the
representations
(\ref{Gijformula}) and (\ref{self-consistent}) that $G_{ij} =
m_{\mathrm{sc}}^2 h_{ij} + \cdots\,$, that is, $G_{ij}$ has a component that
fluctuates on the scale $q^{-1}$. The improvement from $q^{-1}$ to
$q^{-2}$ in the first inequality of (\ref{roughineq1}) arises from
an averaging in the summation
$m= N^{-1}\sum_i G_{ii}$. If the random variables in the average were
independent, one would expect the averaging to yield an improvement of
order $N^{-1/2}$; however, in our case there are strong dependencies,
which result in the more modest gain of order $q^{-1}$.

At the edge ($\kappa_E = 0$), the estimates (\ref{scm}) and (\ref
{Gijestimate}) may be roughly stated as
\[
\vert m(z) - m_{\mathrm{sc}}(z) \vert\lesssim\frac{1}{q} +
\frac{1}{N \eta},\qquad
\vert G_{ij}(z) - \delta_{ij} m_{\mathrm{sc}}(z) \vert\lesssim
\frac{1}{q} + \frac{\eta^{1/4}}{\sqrt{N \eta}} + \frac{1}{N \eta}.
\]

Now we formulate the local semicircle law
for the matrix $A$ given in Definition~\ref{definitionofA}. Define
the quantities
%
%
\begin{equation}\label{tildeg}
\nc G(z) \deq(A - z)^{-1},\qquad \nc m(z) \deq\frac
{1}{N} \tr\nc G(z).
\end{equation}

%
\begin{theorem}[(Local semicircle law for $A$)] \label{LSCTHMA}
There are universal constants $C_1, C_2 > 0$ such that the following holds.
Suppose that $A$ satisfies Definition~\ref{definitionofA}, and that
$\xi$ and $q$ satisfy (\ref{assumptionsforSLSC}).
Then there is a constant $\nu> 0$---depending on $A_0$, $\Sigma$ and
the constants $C$ in (\ref{momentconditions}),
(\ref{lowerboundond}) and (\ref{upperboundonf})---such that
the following holds.

We have the local semicircle law: the event
%
%
\begin{equation}\label{scmA}
\bigcap_{z \in D} \biggl\{{\vert\nc m(z) - m_{\mathrm{sc}}(z)
\vert\leq(\log N)^{C_2 \xi} \biggl({\min\biggl\{{\frac{1}{q^2
\sqrt{\kappa_E+\eta}}, \frac{1}{q}}\biggr\} + \frac{1}{N \eta
}}\biggr)}\biggr\}\hspace*{-35pt}
\end{equation}
holds with $(\xi,\nu)$-high probability.
Moreover, we have the following estimate on the individual matrix
elements of $\nc G$. If the assumption (\ref{upperboundonf}) is
strengthened to
%
%
\begin{equation} \label{fupperbound}
0 \leq f \leq C_0 N^{1/2}
\end{equation}
for some constant $C_0$, then the event
%
%
\begin{eqnarray}\label{GijestimateA}
&&\bigcap_{z \in D} \Biggl\{\max_{1 \leq i,j \leq N} \vert\nc
G_{ij}(z) - \delta_{ij} m_{\mathrm{sc}}(z) \vert\nonumber\\[-8pt]\\[-8pt]
&&\qquad\hspace*{0pt}\leq(\log N)^{C_2 \xi}
\Biggl({\frac{1}{q} + \sqrt{\frac{\im m_{\mathrm{sc}}(z)}{N \eta}} + \frac
{1}{N \eta}}\Biggr)\Biggr\}\nonumber
\end{eqnarray}
holds with $(\xi,\nu)$-high probability, where $\nu$ also depends on $C_0$.
\end{theorem}

Next, let $\lambda_1 \leq\cdots\leq\lambda_N$ be the ordered
family of eigenvalues of $H$, and let $\f u_1,\ldots, \f
u_N$ denote the associated eigenvectors. Similarly, we denote the
ordered eigenvalues of $A$ by $\mu_1 \leq\cdots\leq
\mu_N$ and the associated\vspace*{1pt} eigenvectors by $\f v_1,\ldots, \f v_N$. We
use the notation $\f u_\alpha= (u_\alpha(i))_{i
= 1}^N$ and $\f v_\alpha= (v_\alpha(i))_{i = 1}^N$ for the vector
components. All eigenvectors are $\ell^2$-normalized
and have real components.

We state our main result about the local density of states of $A$.
For $E_1 < E_2$ define the counting functions
%
%
\begin{eqnarray} \label{empiricalintegrateddensity}
\cal N_{\mathrm{sc}}(E_1, E_2) &\deq& N \int_{E_1}^{E_2} \varrho_{\mathrm{sc}}(x)
\,\dd x,\nonumber\\[-8pt]\\[-8pt]
\nc{\cal N}(E_1, E_2) &\deq&\vert\{{\alpha\st E_1 < \mu
_\alpha\leq E_2}\} \vert.\nonumber
\end{eqnarray}

%
\begin{theorem}[(Local density of states)] \label{thmlocaldensityofstates}
Suppose that $A$ satisfies Definition~\ref{definitionofA} and that
$\xi$ and $q$ satisfy (\ref{assumptionsforSLSC}). Then there is a
constant $\nu> 0$---depending on $A_0$, $\Sigma$ and the constants $C$
in (\ref{momentconditions}), (\ref{lowerboundond}) and
(\ref{upperboundonf})---as well as a constant $C > 0$ such that the
following holds.

For any $E_1$ and $E_2$ satisfying $E_2 \geq E_1 + (\log N)^{C \xi}
N^{-1}$ we have
%
%
\begin{eqnarray} \label{localdensityatedge}
&&\nc{\cal N}(E_1, E_2)\nonumber\\
&&\qquad= \cal N_{\mathrm{sc}}(E_1, E_2)\\
&&\qquad\quad{}\times \biggl[{1 + O
\biggl({(\log N)^{C \xi} \biggl({\frac{1}{N (E_2 - E_1)^{3/2}} +
\frac{1}{q^2 (E_2 - E_1)}}\biggr)}\biggr)}\biggr]\nonumber
\end{eqnarray}
with $(\xi,\nu)$-high probability.\vadjust{\goodbreak}

Away from the spectral edge we have a stronger statement. Fix $\kappa
_* > 0$. Then, for any $E_1$ and $E_2$ satisfying
$E_2 \geq E_1 + (\log N)^{C \xi} N^{-1}$ as well as $\kappa_{E_1}
\geq\kappa_*$ and $\kappa_{E_2} \geq\kappa_*$, we
have
%
%
\begin{equation} \label{localdensityinbulk}
\nc{\cal N}(E_1, E_2) = \cal N_{\mathrm{sc}}(E_1, E_2) \biggl[{1 + O
\biggl({(\log N)^{C \xi} \biggl({\frac{1}{N (E_2 - E_1)} + \frac
{1}{q^2}}\biggr)}\biggr)}\biggr]\hspace*{-35pt}
\end{equation}
with $(\xi,\nu)$-high probability, where the constant in $O(\cdot)$
depends on $\kappa_*$.
\end{theorem}
%
%
\begin{remark}
Both results (\ref{localdensityatedge}) and (\ref{localdensityinbulk})
are special cases of a more general,
uniform, estimate; see Proposition~\ref{thmgenerallocaldensityofstates}.
\end{remark}
%

In the recent work~\cite{TVW}, the asymptotics of the local density of
states was also established, but only in much
larger spectral windows, of size at least $(Np)^{-1/10} = q^{-1/5}$,
and with a weaker error estimate.

Our next result concerns the integrated densities of states,
%
%
\begin{equation}
n_{\mathrm{sc}}(E) \deq\frac{1}{N} \cal N_{\mathrm{sc}}(-\infty,
E),\qquad
\nc{\fra n}(E) \deq\frac{1}{N} \nc{\cal
N}(-\infty, E).
\end{equation}

%
\begin{theorem}[(Integrated density of states)]
\label{thmgeneralintegrateddensity} Suppose that $A$ satisfies Definition
\ref{definitionofA} and that $\xi$ and $q$ satisfy
(\ref{assumptionsforSLSC}). Then there is a constant $\nu>
0$---depending on $A_0$, $\Sigma$ and the constants $C$ in (\ref
{momentconditions}), (\ref{lowerboundond}) and
(\ref{upperboundonf})---as well as a constant $C > 0$ such that the
event
%
%
\begin{equation} \label{n-nsc}
\bigcap_{E \in[-\Sigma, \Sigma]} \biggl\{{\vert\nc{\fra
n}(E) - n_{\mathrm{sc}}(E) \vert\leq(\log N)^{C \xi} \biggl({\frac
{1}{N} + \frac{1}{q^3} + \frac{\sqrt{\kappa_E}}{q^2}}
\biggr)}\biggr\}
\end{equation}
holds with $(\xi,\nu)$-high probability.
\end{theorem}

Next, we prove that the $N - 1$ first eigenvalues of $A$ are close to
their classical locations predicted by the
semicircle law. Denote by $\gamma_\alpha$ the classical location of
the $\alpha$th eigenvalue, defined through
%
%
\begin{equation} \label{defofgamma}
n_{\mathrm{sc}}(\gamma_\alpha) = \frac{\alpha}{N} \qquad\mbox{for }
\alpha= 1,\ldots, N.
\end{equation}
The following theorem compares the locations of the eigenvalues $\mu
_1,\ldots, \mu_{N - 1}$ to their classical locations
$\gamma_1,\ldots, \gamma_{N - 1}$. It is well known that the largest
eigenvalue $\mu_N$ of the Erd\H{o}s--R\'enyi matrix
is much larger than $\gamma_N$. This holds for more general sparse
matrices as well; more precisely, if $f \geq1 + c$,
then $\mu_N \approx f + f^{-1}$ is separated from $\mu_{N - 1}
\approx2$ by a gap of order one. The precise behavior
of $\mu_N$ in this regime is established in Theorem \ref
{theoremlargesteigenvalue} below.
%
%
\begin{theorem}[(Eigenvalue locations)] \label{thmeigenvaluelocations}
Suppose that $A$ satisfies Definition~\ref{definitionofA} and that
$\xi$ satisfies (\ref{assumptionsforSLSC}). Let $\phi$ be an exponent
satisfying $0 < \phi\leq1/2$, and set $q = N^{\phi}$. Then there is a
constant $\nu> 0$---depending on\vadjust{\goodbreak} $A_0$, $\Sigma$ and the constants $C$
in (\ref{momentconditions}), (\ref{lowerboundond}) and
(\ref{upperboundonf})---as well as a constant $C > 0$ such that the
following holds.

We have with $(\xi,\nu)$-high probability that
%
%
\begin{equation} \label{mainestimateonQ}
\sum_{\alpha= 1}^{N - 1} \vert\mu_\alpha- \gamma_\alpha\vert^2
\leq(\log N)^{C \xi} ({N^{1 - 4 \phi} + N^{4/3 - 8 \phi
}}).
\end{equation}
Moreover, for all $\alpha= 1,\ldots, N - 1$ we have with $(\xi,\nu
)$-high probability that
%
%
\begin{eqnarray} \label{detailedestimateforlargephi}
\vert\mu_\alpha- \gamma_\alpha\vert&\leq&(\log N)^{C \xi}
\bigl(N^{-2/3} \bigl[{\wh\alpha^{-1/3} + \f1 \bigl({\wh\alpha\leq
(\log N)^{C \xi} (1 + N^{1 - 3 \phi})}\bigr)}\bigr]\hspace*{-32pt}\nonumber\\[-8pt]\\[-8pt]
&&\hspace*{143.6pt}{} + N^{2/3 - 4
\phi} \wh\alpha
^{-2/3} + N^{-2 \phi}\bigr),\hspace*{-35pt}\nonumber
\end{eqnarray}
where we abbreviated $\wh\alpha\deq\min\{{\alpha, N - \alpha}\}$.
\end{theorem}
%
%
\begin{remark}
Under the assumption $\phi\geq1/3$, the estimate (\ref
{detailedestimateforlargephi}) simplifies to
%
%
\begin{equation}
\vert\mu_\alpha- \gamma_\alpha\vert\leq(\log N)^{C \xi}
({N^{-2/3} \wh\alpha^{-1/3} + N^{-2 \phi}}),
\end{equation}
which holds with $(\xi,\nu)$-high probability.
\end{remark}
%
%
\begin{remark}
Theorems~\ref{thmlocaldensityofstates}, \ref
{thmgeneralintegrateddensity} and~\ref{thmeigenvaluelocations}
also hold---with the same proof---for the matrix $H$. More precisely,
Theorem~\ref{thmlocaldensityofstates} holds with $\nc{\cal
N}(E_1, E_2)$ replaced with
\[
\cal N(E_1, E_2) \deq\vert\{{\alpha\st E_1 < \lambda
_\alpha\leq E_2}\} \vert,
\]
Theorem~\ref{thmgeneralintegrateddensity} holds with $\nc{\fra
n}(E)$ replaced with
\[
\fra n(E) \deq\frac{1}{N} \cal
N(-\infty, E),
\]
and Theorem~\ref{thmeigenvaluelocations} holds with $\mu_\alpha$
replaced with $\lambda_\alpha$.
\end{remark}

Our final result shows that the eigenvectors of $A$ are
\textit{completely delocalized}.

%
\begin{theorem}[(Complete delocalization of eigenvectors)]
\label{theoremdelocalization} Suppose that $A$ satisfies Definition
\ref{definitionofA} and (\ref{fupperbound}). Then there is a constant
$\nu> 0$---depending on~$A_0$, $\Sigma$ and the constants $C$ in (\ref
{momentconditions}), (\ref{lowerboundond}) and
(\ref{upperboundonf})---such that the following statements hold for any
$\xi$ satisfying (\ref{boundsonxi}).

We have with $(\xi,\nu)$-high probability
%
%
\begin{equation} \label{delocforbulk}
\max_{\alpha< N} \|\f v_\alpha\|_\infty\leq\frac
{(\log N)^{4 \xi}}{\sqrt{N}}.
\end{equation}
Moreover, we have with $(\xi,\nu)$-high probability
%
%
\begin{equation} \label{l2estimate}
\|\f v_N - \f e \|_2 = \frac{1}{f} + O\Biggl({\sqrt
{\frac{1}{f^3} + \frac{(\log N)^\xi}{f \sqrt{N}}}}\Biggr).
\end{equation}
If additionally $f \leq C$ for some constant $C$, then we have with
$(\xi,\nu)$-high probability
%
%
\begin{equation} \label{delocforconstantf}
\|\f v_N \|_\infty\leq\frac{(\log N)^{4 \xi}}{\sqrt
{N}}.
\end{equation}
Finally, there exists positive constants $C, C_0$ such that if $f \geq
C_0 (\log N)^{\xi}$, then we have with $(\xi,\nu)$-high probability
%
%
\begin{equation} \label{delocalizationofumax}
\|\f v_N - \f e \|_\infty\leq C \frac{(\log N)^\xi
}{\sqrt
{N} f}.
\end{equation}
\end{theorem}
%
%
\begin{remark}
If $f$ does not grow with $N$, then the components $v_N(i)$ of the
largest eigenvector fluctuate, and we do not expect
(\ref{delocalizationofumax}) to hold. However, a delocalization
bound similar to (\ref{delocforconstantf}) holds
for all $f$. In (\ref{delocforconstantf}) this bound was proved for
$f \leq C$. In fact, a slight modification of our
proof yields complete delocalization for the values of $f$ not covered
by Theorem~\ref{theoremdelocalization}, that is,
$1 \ll f \leq C_0 (\log N)^\xi$. We claim that in this case we have
with $(\xi,\nu)$-high probability
%
%
\begin{equation} \label{generaldeloc}
\|\f v_N \|_\infty\leq\frac{(\log N)^{C \xi}}{\sqrt
{N}}.
\end{equation}
The required modifications are sketched at the end of Section \ref
{sectionproofofdeloc} below.
\end{remark}
%
%
\begin{remark} \label{remarkdelocforH}
Similarly, if $H$ satisfies Definition~\ref{definitionofH}, all of
its eigenvectors are delocalized in the sense that
\[
\max_{\alpha} \|\f u_\alpha\|_\infty\leq\frac
{(\log
N)^{4 \xi}}{\sqrt{N}}
\]
with $(\xi,\nu)$-high probability.
The proof is a straightforward application of (\ref{weakGii}) below
and the estimate (\ref{pfofdeloc}) applied to
$G_{jj}$.
\end{remark}

In the recent work~\cite{TVW}, a weaker upper bound of size
$(Np)^{-1/2} = q^{-1}$ was established for the
$\ell^\infty$-norm of the eigenvectors of $A$ associated with
eigenvalues away from the spectral edge.

\section{The weak local semicircle law for $H$} \label{sectionWLSC}

In this section we introduce and prove a weak version of the local
semicircle law for the matrix $H$. This result is
weaker than our final result for $H$, Theorem~\ref{LSCTHM}, but it
will be used as an a priori bound for the proof of
Theorem~\ref{LSCTHM}. Moreover, Theorem~\ref{WLSCTHM} holds under
slightly weaker assumptions on $\xi$ than Theorem
\ref{LSCTHM}, and is for this reason a more suitable tool for proving
eigenvector delocalization, Theorem~\ref{theoremdelocalization}; see
Section~\ref{sectionproofofdeloc} for details.

We shall prove Theorem~\ref{WLSCTHM} (the weak local semicircle law)
for spectral parameters $z$ in the set
%
%
\begin{equation} \label{definitionDL}
D_L \deq\{{z \in\C\st\vert E \vert\leq\Sigma,
(\log
N)^LN^{-1} \leq\eta\leq3}\} \subset D,
\end{equation}
where the parameter $L \equiv L(N)$ will always satisfy
%
%
\begin{equation} \label{lowerboundonL}
L \geq8 \xi.
\end{equation}

%
\begin{theorem}[(Weak local semicircle law for $H$)] \label{WLSCTHM}
Let $H$ satisfy Definition~\ref{definitionofH}.
Then there are constants $\nu> 0$ and $C > 0$ such that the following
statements hold for any $\xi$ satisfying
(\ref{boundsonxi}) and $L$ satisfying (\ref{lowerboundonL}).

The events
%
%
\begin{equation} \label{weakGij}
\bigcap_{z \in D_L} \biggl\{{\max_{i\neq j}\vert G_{ij}(z) \vert
\leq\frac{C}{q} + \frac{C (\log N)^{2 \xi}}{\sqrt{N \eta
}}}\biggr\}
\end{equation}
and
%
%
\begin{equation} \label{weakGii}
\bigcap_{z \in D_L} \biggl\{{\max_{i}\vert G_{ii}(z)-m(z) \vert
\leq\frac{C(\log N)^\xi}{q} + \frac{C (\log N)^{2 \xi}}{\sqrt{N
\eta}}}\biggr\}
\end{equation}
hold with $(\xi,\nu)$-high probability.
Furthermore, we have the weak local semicircle law: the event
%
%
\begin{equation} \label{weakSCL}
\bigcap_{z \in D_L} \biggl\{{\vert m(z) - m_{\mathrm{sc}}(z) \vert\leq
\frac{C (\log N)^{\xi}}{\sqrt{q}} + \frac{C (\log N)^{2 \xi}}{(N
\eta)^{1/3}}}\biggr\}
\end{equation}
holds with $(\xi,\nu)$-high probability.
\end{theorem}

Roughly, Theorem~\ref{WLSCTHM} states that
%
%
\begin{equation}\label{weakrough1}
\vert G_{ij} - \delta_{ij} m(z) \vert\lesssim\frac
{1}{q} + \frac{1}{\sqrt{N \eta}}
\end{equation}
and
%
%
\begin{equation}\label{weakrough2}
\vert m(z) - m_{\mathrm{sc}}(z) \vert\lesssim\frac{1}{\sqrt
{q}} + \frac{1}{(N \eta)^{1/3}}.
\end{equation}
Comparing with the strong local semicircle law, Theorem~\ref{LSCTHM},
we note
that the error bound in (\ref{weakrough1}) for $G_{ij}$ is already optimal
in the bulk. However, unlike Theorem~\ref{LSCTHM},
the quantity $G_{ii}$ is compared to $m$ and not $m_{\mathrm{sc}}$.

On the other hand, the estimate (\ref{weakrough2}) is considerably weaker
than the corresponding bound in (\ref{roughineq1}).
The smaller power $1/3$ in the factor $(N \eta)^{-1/3}$ reflects the
instability
near the edge; it appears because we insist on having uniform bounds up
to the edge.
If we were interested only in the bulk, it would be easy to repeat the
proof of Theorem~\ref{WLSCTHM} to obtain $(N\eta\kappa)^{-1/2}$,
thus replacing the power $1/3$ with $1/2$.
The price would be a coefficient which blows up at the edge.\vadjust{\goodbreak}

As in Theorem~\ref{LSCTHM}, the estimates of Theorem~\ref{WLSCTHM}
are stable down to the optimal scale $\eta\gtrsim N^{-1}$, uniformly
up to the edge. Thus, the difference between Theorems~\ref{LSCTHM} and
\ref{WLSCTHM} lies only in the precision of the estimates.

In order to prove Theorem~\ref{WLSCTHM}, we first collect some basic
tools and notation.

\subsection{Preliminaries}
The following lemma collects some useful properties of $m_{\mathrm{sc}}$ defined
in (\ref{identityofmsc}).
%
%
\begin{lemma} \label{lemmamsc}
For $z = E + \ii\eta\in D_L$ abbreviate $\kappa\equiv\kappa_E$.
Then we have
%
%
\begin{equation} \label{boundsonmsc}
\vert m_{\mathrm{sc}}(z) \vert\sim1,\qquad \vert1 - m_{\mathrm{sc}}(z)^2
\vert\sim
\sqrt{\kappa+ \eta}.
\end{equation}
Moreover,
\[
\im m_{\mathrm{sc}}(z) \sim
\cases{
\sqrt{\kappa+ \eta}, &\quad if $\vert E \vert\leq2$,
\vspace*{2pt}\cr
\dfrac{\eta}{\sqrt{\kappa+ \eta}}, &\quad if $\vert E \vert\geq2$.}
\]
Here the implicit constants in $\sim$ depend on $\Sigma$ in (\ref
{definitionD}).
\end{lemma}
\begin{pf}
The proof is an elementary calculation; see Lemma 4.2 in~\cite{EYY2}.
\end{pf}

In order to streamline notation, we shall often omit the explicit
dependence of quantities on the spectral parameter $z
\in D_L$; thus, we write, for instance, \mbox{$G_{ij}(z) \equiv G_{ij}$}.
Define the $z$-dependent quantities
%
%
\begin{eqnarray} \label{deflambda}
\Lambda_o &\deq&\max_{i \neq j} \vert G_{ij} \vert,\qquad
\Lambda_d \deq\max_i \vert G_{ii} - m_{\mathrm{sc}}
\vert,\nonumber\\[-8pt]\\[-8pt]
\Lambda&\deq&\vert m - m_{\mathrm{sc}} \vert,\qquad v_i \deq
G_{ii} -
m_{\mathrm{sc}}.\nonumber
\end{eqnarray}

%
\begin{definition} \label{definitionofminor}
Let $\bb T \subset\{1,\ldots, N\}$. Then we define $H^{(\bb T)}$ as
the $(N - \vert\bb T \vert) \times(N - \vert\bb T \vert)$
minor of $H$ obtained by removing all rows and columns of $H$ indexed
by $i \in\bb T$. Note that we keep the names of
indices of $H$ when defining $H^{(\bb T)}$.

More formally, for $i \in\{1,\ldots, N\}$ we define the operation
$\pi_i$ on the
probability space by
%
%
\begin{equation} \label{definitionofpi}
(\pi_i(H))_{kl} \deq\f1 (k \neq i) \f1 (l \neq i) h_{kl}.
\end{equation}
For $\bb T \subset\{1,\ldots, N\}$ we also write $\pi_{\bb T} \deq
\prod_{i \in\bb T} \pi_i$. Then we define
\[
H^{(\bb T)} \deq((\pi_{\bb T}(H))_{ij})_{i,j \notin\bb T}.
\]
The quantities $G^{(\bb T)}(z)$, $\lambda_\alpha^{(\bb T)}$, $\f
u_\alpha^{(\bb T)}$, etc. are defined in the obvious
way using $H^{(\bb T)}$. Here $\alpha= 1,\ldots, \alpha_{\max}$,
where $\alpha_{\max} \deq N - \vert\bb T \vert$.

Moreover, we use the notation
\[
\sum_i^{(\bb T)} \deq\mathop{\sum_{i = 1 }}_{ i \notin\bb
T}^N,
\]
and abbreviate $(i) = (\{i\})$ as well as $(\bb T i) = (\bb T \cup\{i\})$.

We also set
%
%
\begin{equation}\label{mTdef}
m^{(\bb T)} \deq\frac{1}{N} \sum_i^{(\bb T)} G^{(\bb T)}_{ii}.
\end{equation}
\end{definition}

Note that we choose the normalization $N^{-1}$ instead of the more
natural $(N - \vert\bb T \vert)^{-1}$ in (\ref{mTdef});
this is simply a convenient choice for later applications.

The next\vspace*{1pt} lemma collects the main identities of the resolvent matrix
elements $G_{ij}^{(\bb T) }$. Its proof is
elementary linear algebra; see, for example,~\cite{EYY}.

%
\begin{lemma} \label{lemmaresolventid}
For $i,j \neq k$ we have
%
%
\begin{equation} \label{GijGijk}
G_{ij} = G_{ij}^{(k)} + \frac{G_{ik} G_{kj}}{G_{kk}}.
\end{equation}
For $i \neq j$ we have
%
%
\begin{equation} \label{Gijformula}
G_{ij} = - G_{ii} G_{jj}^{(i)} ({h_{ij} - Z_{ij}}),\qquad
G_{ii} = (h_{ii}-z - Z_{ii})^{-1},
\end{equation}
where we defined, for arbitrary $i,j \in\{1,\ldots,N\}$,
%
%
\begin{equation} \label{defofZij}
Z_{ij} \deq\f h_i \cdot G^{(ij)} \f h_j = \sum_{k,l}^{(ij)}
h_{ik} G^{(ij)}_{kl} h_{lj}.
\end{equation}
Here $\f h_i$ denotes the vector given by the
$i$th column of $H$. Note that in expressions of the form (\ref
{defofZij}) it is implied that the $i$th and $j$th
entries of $\f h_i$ and $\f h_j$ have been removed; we do not indicate
this explicitly, as it is always clear from the
context.
\end{lemma}
%
%
\begin{remark} \label{remarkidentitiesfoHT}
Lemma~\ref{lemmaresolventid} remains trivially valid for the minors
$H^{(\bb T)}$ of~$H$. For instance,
(\ref{GijGijk}) reads
\[
G_{ij}^{(\bb T)} = G_{ij}^{(\bb T k)} + \frac{G_{ik}^{(\bb T)}
G_{kj}^{(\bb T)}}{G_{kk}^{(\bb T)}}
\]
for $i,j, k \notin\bb T$ and $i,j \neq k$.
\end{remark}
%
%
\begin{definition}\label{defEi}
We denote by $\E_i$ the partial expectation with respect to the
variables $\f h_i = (h_{ij})_{j = 1}^N$, and set $\IE_i
X \deq X - \E_i X$.
\end{definition}

We abbreviate
%
%
\begin{equation} \label{definitionofZi}
Z_i \deq\IE_i Z_{ii} = \IE_i \sum_{k,l}^{(i)} h_{ik}
G^{(i)}_{kl} h_{li}
= \sum_{k,l}^{(i)}\biggl( h_{ik}h_{li} - \frac{1}{N}\delta_{kl}\biggr)
G^{(i)}_{kl}.
\end{equation}

The following trivial large deviation estimate provides a bound on the
matrix elements of $H$.
%
%
\begin{lemma} \label{lemmaOmegah}
For $C$ large enough we have with $(\xi,\nu)$-high probability
\[
\vert h_{ij} \vert\leq\frac{C}{q}.
\]
\end{lemma}
\begin{pf}
The claim follows by choosing $p = \nu(\log N)^\xi$ in (\ref
{momentconditions}) and applying Markov's inequality.
\end{pf}

We collect here the large deviation estimates for random variables
whose moments decay slowly. Their proof is given in the
\hyperref[app]{Appendix}.
%
%
\begin{lemma} \label{lemma_lde}
\textup{(i)}
Let $(a_i)$ be a family of centered and independent random variables satisfying
%
%
\begin{equation} \label{generalizedmomentcondition}
\E\vert a_i \vert^p \leq\frac{C^p}{N^\gamma q^{\alpha p +
\beta}}
\end{equation}
for all $2 \leq p \leq(\log N)^{A_0 \log\log N}$, where $\alpha\geq
0$ and $\beta, \gamma\in\R$. Then there is a
$\nu> 0$, depending only on $C$ in (\ref{generalizedmomentcondition}),
such that for all $\xi$ satisfying
(\ref{boundsonxi}) we have with $(\xi,\nu)$-high probability
%
%
\begin{equation} \label{generalizedLDE}
\biggl\vert\sum_i A_i a_i \biggr\vert\leq(\log N)^\xi
\biggl[{\frac{\sup_i \vert A_i \vert}{q^\alpha} + \biggl({\frac
{1}{N^\gamma q^{\beta+ 2 \alpha}} \sum_i \vert A_i \vert^2}
\biggr)^{1/2}}\biggr].
\end{equation}

\textup{\hphantom{i}(ii)}
Let $a_1,\ldots, a_N$ be centered and independent random variables satisfying
%
%
\begin{equation} \label{boundonmomentsofsparseentries}
\E\vert a_i \vert^p \leq\frac{C^p}{N q^{p - 2}}
\end{equation}
for $2 \leq p \leq(\log N)^{A_0 \log\log N}$.
Then there is a $\nu> 0$, depending only on $C$ in (\ref
{boundonmomentsofsparseentries}), such that for all $\xi$
satisfying (\ref{boundsonxi}), and for any $A_i \in\C$ and $B_{ij}
\in\C$, we have with $(\xi,\nu)$-high probability
%
%
\begin{eqnarray}
\label{aA}\quad
\Biggl\vert\sum_{i = 1}^N A_i a_i \Biggr\vert
&\leq&
(\log
N)^\xi\Biggl[{\frac{\max_i \vert A_i \vert}{q} + \Biggl({\frac
{1}{N} \sum_{i = 1}^N \vert A_i \vert^2}\Biggr)^{1/2}}\Biggr],
\\
\label{aaBd}
\hspace*{28pt}\Biggl\vert\sum_{i = 1}^N \ol{a}_i B_{ii} a_i - \sum_{i = 1}^N
\sigma^2_i B_{ii} \Biggr\vert
& \leq&
(\log N)^\xi\frac{B_d}{q},
\\
\label{aaBo}
\biggl\vert\sum_{1 \leq i \neq j \leq N} \ol{a}_i B_{ij} a_j
\biggr\vert
& \leq&(\log N)^{2 \xi} \biggl[{\frac{B_o}{q} +
\biggl({\frac{1}{N^2} \sum_{i \neq j} \vert B_{ij} \vert^2}
\biggr)^{1/2}}\biggr],
\end{eqnarray}
where $\sigma_i^2$ denotes the variance of $a_i$ and we abbreviated
\[
B_d \deq\max_i \vert B_{ii} \vert,\qquad
B_o \deq\max_{i \neq j} \vert B_{ij} \vert.
\]

\textup{(iii)}
Let $a_1,\ldots, a_N$ and $b_1,\ldots, b_N$ be independent random
variables, each satisfying (\ref{boundonmomentsofsparseentries}).
Then there is a $\nu> 0$, depending only on $C$ in (\ref
{boundonmomentsofsparseentries}), such that for all $\xi$
satisfying (\ref{boundsonxi}) and $B_{ij} \in\C$ we have with
$(\xi,\nu)$-high probability
%
%
\begin{equation}\label{abB}\quad
\Biggl\vert\sum_{i,j = 1}^N a_i B_{ij} b_j \Biggr\vert\leq
(\log N)^{2 \xi} \biggl[{\frac{B_d}{q^2} + \frac{B_o}{q} +
\biggl({\frac{1}{N^2} \sum_{i \neq j} \vert B_{ij} \vert^2}
\biggr)^{1/2}}\biggr].
\end{equation}
\end{lemma}
%
%
\begin{remark}
Note that the estimates (\ref{aA}) and (\ref{aaBd}) are special cases
of (\ref{generalizedLDE}).
The right-hand side of the large deviation bound (\ref{generalizedLDE})
consists of two terms, which can be understood
as follows. The first term gives the large deviation bound for the
special case where $A_i$ vanishes for all but one
$i$; in this case it is immediate that $\vert A_i a_i \vert\leq(\log
N)^\xi\vert A_i \vert q^{-\alpha}$ with $(\xi,\nu)$-high probability.
The second term is
equal to the variance of $\sum_i A_i a_i$. In particular, (\ref
{generalizedLDE}) is optimal (up to factors of $\log
N$). The estimates (\ref{aA})--(\ref{abB}) can be interpreted
similarly. [Note that the powers of $q$ in the estimates
(\ref{aaBo})--(\ref{abB}) are not optimal; this is, however, of no
consequence for later applications.]
\end{remark}

For a family $F_1,\ldots, F_N$ we introduce the notation
\[
[F] \deq\frac{1}{N} \sum_{i = 1}^N F_i.
\]

The following lemma contains the self-consistent resolvent equation on
which our proof relies.
%
%
\begin{lemma} \label{lemmaself-consistentequation}
We have the identity
%
%
\begin{equation} \label{self-consistent}
G_{ii} = \frac{1}{-z - m_{\mathrm{sc}} - ([v] - \Upsilon_i)},
\end{equation}
where
\[
\Upsilon_i \deq h_{ii} - Z_i + \cal A_i
\]
and
%
%
\begin{equation}\label{Adef}
\cal A_i \deq\frac{1}{N} \sum_j \frac{G_{ij} G_{ji}}{G_{ii}}.
\end{equation}
\end{lemma}
\begin{pf}
The proof is a simple calculation using (\ref{Gijformula}) and (\ref{GijGijk}).
\end{pf}

\subsection{\texorpdfstring{Basic estimates on the event $\Omega(z)$}
{Basic estimates on the event Omega(z)}}
%
%
\begin{definition} \label{defOmega}
For $z \in D_L$ introduce the event
%
%
\begin{equation} \label{defOmegaz}
\Omega(z) \deq\{{\Lambda_d(z) + \Lambda_o(z) \leq(\log
N)^{-\xi}}\}
\end{equation}
and the control parameter
%
%
\begin{equation} \label{defPsi}
\Psi(z) \deq\sqrt{\frac{\Lambda(z) + \im m_{\mathrm{sc}}(z)}{N \eta
}}.
\end{equation}
\end{definition}

Note that $\Psi(z)$ is a random variable. Moreover, on $\Omega(z)$ we
have $\Psi(z) \leq C (\log N)^{-4 \xi}$ by
(\ref{lowerboundonL}).

Throughout the following we shall make use of the fundamental identity
%
%
\begin{eqnarray} \label{ward}
\sum_{j} \vert G_{ij} \vert^2 &=& \sum_j \sum_\alpha\frac
{\ol{u}_\alpha(i) u_\alpha(j)}{\lambda_\alpha- z} \sum
_\beta
\frac{u_\beta(i) \ol{u}_\beta(j)}{\lambda_\beta- \bar
z}\nonumber\\[-8pt]\\[-8pt]
&=&
\sum_\alpha
\frac{\vert u_\alpha(i) \vert^2}{\vert\lambda_\alpha- z \vert^2}
= \frac
{1}{\eta} \im G_{ii}.\nonumber
\end{eqnarray}
A similar identity holds for $H^{(\bb T)}$. Using
the lower bound $|m_{\mathrm{sc}}(z)|\geq c$ from (\ref{boundsonmsc})
and the definition (\ref{defOmegaz}), we find
%
%
\begin{equation}\label{lowerboundonGii}
c \leq\vert G_{ii}(z) \vert\leq C
\end{equation}
on $\Omega(z)$. Using (\ref{GijGijk}) repeatedly, we find that on
$\Omega(z)$ we have
%
%
\begin{equation} \label{GiionOmega}
c \leq\bigl\vert G_{ii}^{(\bb T)}(z) \bigr\vert\leq C
\end{equation}
for $\vert\bb T \vert\leq10$ (here 10 can be replaced with any fixed
number). Similarly, we have on $\Omega(z)$ that
%
%
\begin{equation} \label{GijonOmega}
\max_{i\ne j}
\bigl\vert G_{ij}^{(\bb T)}(z) \bigr\vert\leq C \Lambda_o(z) \leq C
(\log
N)^{-\xi}
\end{equation}
for $\vert\bb T \vert\leq10$.
%
%
\begin{lemma} \label{lemmaiTGT}
Fixing $z=E+i\eta\in D_L$,
we have for any $i$ and $\bb T \subset\{1,\ldots, N\}$ satisfying
$i \notin\bb T$ and $\vert\bb T \vert\leq10$ that
%
%
\begin{equation} \label{iTGT}
m^{(i \bb T)}(z) = m^{(\bb T)}(z) + O \biggl({\frac{1}{N \eta
}}\biggr)
\end{equation}
holds in $\Omega(z)$.\vadjust{\goodbreak}
\end{lemma}
\begin{pf}
We use (\ref{GijGijk}) to write
\[
\frac{1}{N} \sum_{j}^{(i \bb T)} G_{jj}^{(i \bb T)} = \frac
{1}{N} \sum_{j}^{(i \bb T)} G_{jj}^{(\bb T)} -
\frac{1}{N} \sum_{j}^{(i \bb T)} \frac{ G_{ji}^{(\bb T)}
G_{ij}^{(\bb T)}}{ G_{ii}^{(\bb T)}} =
\frac{1}{N} \sum_{j}^{(\bb T)} G_{jj}^{(\bb T)} - \frac{1}{N} \sum
_{j}^{(\bb T)} \frac{ G_{ji}^{(\bb T)}
G_{ij}^{(\bb T)}}{ G_{ii}^{(\bb T)}}.
\]
Therefore,
\begin{eqnarray*}
\frac{1}{N} \sum_{j}^{(i \bb T)} G_{jj}^{(i \bb T)} &=&
\frac{1}{N} \sum_{j}^{(\bb T)} G_{jj}^{(\bb T)} + O\Biggl({\frac
{1}{N} \sum_{j}^{(\bb T)} \bigl\vert G_{ij}^{(\bb T)} \bigr\vert
^2}\Biggr)\\
&=& \frac{1}{N} \sum_{j}^{(\bb T)} G_{jj}^{(\bb T)} + O
\biggl({\frac{1}{N \eta} \im G^{(\bb T)}_{ii}}\biggr).
\end{eqnarray*}
The claim now follows from (\ref{GiionOmega}).
\end{pf}
%
%
\begin{lemma} \label{lemmaoff-diagestimate}
For fixed $z \in D_L$ we have on $\Omega(z)$ with $(\xi,\nu)$-high
probability
%
%
\begin{eqnarray}
\label{Gijapriori}
\Lambda_o(z) & \leq & C \biggl({\frac{1}{q} + (\log N)^{2 \xi}
\Psi(z)}\biggr),
\\
\label{Ziapriori}
\max_i \vert Z_i(z) \vert& \leq & C \biggl({\frac{(\log N)^\xi
}{q} +
(\log N)^{2 \xi} \Psi(z)}\biggr).
\end{eqnarray}
\end{lemma}
\begin{pf}
We start with (\ref{Gijapriori}). Let $i \neq j$. Using (\ref
{Gijformula}), (\ref{GiionOmega}), (\ref{GijonOmega})
and (\ref{GijGijk}) we get on $\Omega(z)$ with $(\xi,\nu)$-high probability
%
%
\begin{eqnarray}
\label{estimateofGij}
\vert G_{ij} \vert &\leq& C ({\vert h_{ij} \vert+ \vert
Z_{ij} \vert})
\leq\frac{C}{q} + C \Biggl\vert\sum_{k,l}^{(ij)} h_{ik}
G_{kl}^{(ij)} h_{lj} \Biggr\vert
\nonumber\\[-8pt]\\[-8pt]
&\leq&\frac{C}{q} + C (\log N)^{2 \xi} \frac{\Lambda_o}{q} + C
(\log N)^{2 \xi}
\Biggl({\frac{1}{N^2} \sum_{k,l}^{(ij)} \bigl| G_{kl}^{(ij)} \bigr|
^2}\Biggr)^{1/2},
\nonumber
\end{eqnarray}
where the last step follows using (\ref{abB}) and (\ref{lowerboundond}).
Using (\ref{GijGijk}) repeatedly and recalling (\ref{GiionOmega}),
we find on $\Omega(z)$ that
$G_{kk}^{(ij)} = G_{kk} + O(\Lambda_o^2)$.
Thus, we get on $\Omega(z)$, by (\ref{ward}),
%
%
\begin{equation} \label{estimateofGij2}
\frac{1}{N^2} \sum_{k,l}^{(ij)} \bigl| G_{kl}^{(ij)} \bigr|^2 =
\frac
{1}{N^2 \eta} \sum_k^{(ij)} \im G^{(ij)}_{kk} \leq
\frac{\im m}{N \eta} + \frac{C \Lambda_o^2}{N \eta}.
\end{equation}
Taking the maximum over $i \neq j$ in (\ref{estimateofGij})
therefore yields, on $\Omega(z)$ with $(\xi,\nu)$-high probability,
\[
\Lambda_o \leq\frac{C}{q} + o(1) \Lambda_o + C (\log N)^{2 \xi
} \sqrt{\frac{\im m}{N \eta}},
\]
where we used (\ref{lowerboundond}) and the fact that $N \eta\geq
(\log N)^{8 \xi}$ by (\ref{lowerboundonL}). This
concludes the proof of (\ref{Gijapriori}).

In order to prove (\ref{Ziapriori}), we write, recalling the
definition (\ref{definitionofZi}),
\[
Z_i = \sum_{k}^{(i)} \biggl({| h_{kk} |^2 - \frac
{1}{N}}\biggr) G_{kk}^{(i)} + \sum_{k \neq l}^{(i)} h_{ik} G_{kl}^{(i)}
h_{li}.
\]
Using (\ref{aaBd}), (\ref{aaBo}) and (\ref{GiionOmega}), we
therefore get, on $\Omega(z)$ with $(\xi,\nu)$-high probability,
\begin{eqnarray*}
| Z_i |&\leq&\frac{C (\log N)^\xi}{q} + C (\log
N)^{2\xi}
\Biggl[{\frac{\Lambda_o}{q} + \Biggl({\frac{1}{N^2} \sum
_{k,l}^{(i)} \bigl| G_{kl}^{(i)} \bigr|^2}\Biggr)^{1/2}}\Biggr] \\
&\leq&\frac{C (\log N)^\xi}{q} + \frac{C (\log N)^{2 \xi} \Lambda
_o}{\sqrt{N \eta}} + C (\log N)^{2 \xi}
\sqrt{\frac{\im m}{N \eta}},
\end{eqnarray*}
similarly to above. Invoking (\ref{Gijapriori}) and recalling (\ref
{lowerboundonL}) finishes the proof.
\end{pf}

We may now estimate $\Lambda_d$ in terms of $\Lambda$.
%
%
\begin{lemma} \label{lemmadiagestimate}
Fix $z =E+i\eta\in D_L$. On $\Omega(z)$ we have with $(\xi,\nu
)$-high probability
%
%
\begin{equation}\label{lambdadleqlambda}
{\max_{i}} |G_{ii}(z)-m(z)|
\leq
C \biggl({\frac{(\log N)^\xi}{q} + (\log N)^{2 \xi} \Psi(z)}
\biggr).
\end{equation}
\end{lemma}
\begin{pf} We use the resolvent equation (\ref{self-consistent}). On
$\Omega(z)$ we have $|\cal A_i |\leq C
\Lambda_o^2$ and $| h_{ii} |\leq C / q$ with $(\xi,\nu)$-high
probability by Lem\-ma~\ref{lemmaOmegah}. Thus, Lem\-ma~\ref{lemmaoff-diagestimate} yields on
$\Omega(z)$ with $(\xi,\nu)$-high probability
%
%
\begin{equation} \label{Upsilonbound}
|\Upsilon_i |\leq
C \biggl({\frac{(\log N)^\xi}{q} + (\log N)^{2 \xi} \Psi(z)}
\biggr) \ll1.
\end{equation}
From (\ref{self-consistent}) we therefore get on $\Omega(z)$ with
$(\xi,\nu)$-high probability
%
%
\begin{equation}\hspace*{28pt}
|G_{ii}-G_{jj}| = | G_{ii} || G_{jj}
||\Upsilon_i-\Upsilon_j |\leq C \biggl({\frac
{(\log N)^\xi
}{q} + (\log N)^{2 \xi} \Psi(z)}\biggr).
\end{equation}
Since $m = \frac{1}{N}\sum_{j} G_{jj}$, the proof is complete.
\end{pf}

Note that (\ref{lambdadleqlambda}) implies
%
%
\begin{equation}\label{436}
\Lambda_d(z) \leq\Lambda(z) + C \biggl({\frac{(\log N)^\xi
}{q} + (\log N)^{2 \xi} \Psi(z)}\biggr)
\end{equation}
on $\Omega(z)$ with $(\xi,\nu)$-high probability.

\subsection{\texorpdfstring{Stability of the self-consistent equation of $[v]$ on $\Omega(z)$}
{Stability of the self-consistent equation of [v] on Omega(z)}}

We now expand the self-consistent equation into a form in which the
stability of the averaged quantity $[v]$ may be
analyzed. Recall the definition $v_i \deq G_{ii} - m_{\mathrm{sc}}$.

%
\begin{lemma} \label{lemmaexpandedself-consisteneq}
Fix $z \in D_L$. Then we have on $\Omega(z)$ with $(\xi,\nu)$-high
probability
%
%
\begin{eqnarray} \label{finalself-consistentequation}
(1 - m_{\mathrm{sc}}^2) [v] &=& m_{\mathrm{sc}}^3 [v]^2 + m_{\mathrm{sc}}^2 [Z]
\nonumber\\[-9pt]\\[-9pt]
&&{}+ O
\biggl({\frac{(\log N)^{2 \xi+ 1}}{q^2} + (\log N)^{4 \xi+ 1} \Psi^2 +
\frac{\Lambda^2}{\log N}}\biggr).\nonumber
\end{eqnarray}
\end{lemma}
\begin{pf} Recall that on $\Omega(z)$ we have $v_i = o(1)$. Moreover,
(\ref{identityofmsc}) and (\ref{boundsonmsc}) imply that $|
m_{\mathrm{sc}}(z) + z |= | m_{\mathrm{sc}}(z) |^{-1} \geq c$
for $z \in D_L$. With (\ref{Upsilonbound}) we may
therefore expand (\ref{self-consistent}) on $\Omega(z)$ up to second
order to get, with $(\xi,\nu)$-high probability,
%
%
\begin{equation} \label{expandedself-consistent}
v_i = m^2_{\mathrm{sc}} ([v] - \Upsilon_i) + m_{\mathrm{sc}}^3 ([v] - \Upsilon_i)^2
+ O ([v] - \Upsilon_i)^3.
\end{equation}
Averaging over $i$ in (\ref{expandedself-consistent}) yields with
$(\xi,\nu)$-high probability
\begin{eqnarray*}
(1 - m_{\mathrm{sc}}^2) [v] &=& - m_{\mathrm{sc}}^2[\Upsilon] + m_{\mathrm{sc}}^3 [v]^2 - 2
m_{\mathrm{sc}}^3 [v] [\Upsilon] + m_{\mathrm{sc}}^3 [\Upsilon^2]\\[-2pt]
&&{} + O
\Bigl({[v] + \max_i |\Upsilon_i |}\Bigr)^3.
\end{eqnarray*}
Recall the definition (\ref{Adef}) of $\cal A_i$.
Using (\ref{aA}) and (\ref{GiionOmega}), we find on $\Omega(z)$
with $(\xi,\nu)$-high probability
\begin{eqnarray*}
[\Upsilon] &=& \frac{1}{N} \sum_i h_{ii} - [Z] + [\cal A] =
-[Z] + O \biggl({\frac{(\log N)^\xi}{N} + \frac{1}{N^2} \sum_{i,j}
| G_{ij} |^2}\biggr)\\[-2pt]
&=& -[Z] + O \biggl({\frac{(\log
N)^\xi
}{N} + \Psi^2}\biggr),
\end{eqnarray*}
where in the last step we used (\ref{ward}).
Moreover, recalling that $|[v]|=\Lambda$, we get by Young's inequality
\[
-2 m_{\mathrm{sc}}^3 [v] [\Upsilon] = O \biggl({\frac{\Lambda^2}{\log N}
+ (\log N) |[\Upsilon] |^2}\biggr).
\]
Recalling (\ref{Upsilonbound}), we therefore have
\begin{eqnarray*}
&&
(1 - m_{\mathrm{sc}}^2)[v]\\[-2pt]
&&\qquad = m_{\mathrm{sc}}^3 [v]^2 + m_{\mathrm{sc}}^2 [Z] \\[-2pt]
&&\qquad\quad{} + O
\biggl({\frac{(\log N)^\xi}{N} + \Psi^2 + (\log N) |[\Upsilon]
|^2 + \max_i |\Upsilon_i |^2 + \Lambda^3 + \frac
{\Lambda^2}{\log N}}\biggr)
\\[-2pt]
&&\qquad = m_{\mathrm{sc}}^3 [v]^2 + m_{\mathrm{sc}}^2 [Z] + O \biggl({\frac{(\log N)^{2
\xi+ 1}}{q^2} + (\log N)^{4 \xi+ 1} \Psi^2 + \frac{\Lambda^2}{\log
N}}\biggr),
\end{eqnarray*}
where we used that on $\Omega(z)$ we have $\Lambda\leq\Lambda_d
\leq(\log N)^{- \xi} \leq(\log N)^{-1}$.\vadjust{\goodbreak}
\end{pf}

Note that, together with (\ref{Ziapriori}), Lemma \ref
{lemmaexpandedself-consisteneq} implies a weak
self-consistent equation on $[v]$:
%
%
\begin{equation}\label{442}
(1 - m_{\mathrm{sc}}^2) [v] = m_{\mathrm{sc}}^3 [v]^2 + O\biggl({\frac{\Lambda
^2}{\log N}}\biggr) + O \biggl({\frac{(\log N)^\xi}{q} + (\log
N)^{2 \xi} \Psi}\biggr)\hspace*{-35pt}
\end{equation}
on $\Omega(z)$ with $(\xi,\nu)$-high probability. Here we used (\ref
{lowerboundond}) and (\ref{lowerboundonL}). For the proof of the
weak semicircle law, Theorem~\ref{WLSCTHM}, we shall only use the
weaker form (\ref{442}) of the self-consistent
equation.

\subsection{\texorpdfstring{Initial estimates for large $\eta$}
{Initial estimates for large eta}} \label{sectinitialestimates}
In order to get the continuity argument of Section~\ref{subseccon}
started, we need some initial estimates on
$\Lambda_d + \Lambda_o$ for large~$\eta$. In other words, we need to
prove that $\Omega(E + \ii\eta)$ is an event of
high probability for $\eta\sim1$.
%
%
\begin{lemma}\label{5992}
Let $ \eta\geq2$. Then for $z = E + \ii\eta\in D_L$ we have
\[
\Lambda_d(z) + \Lambda_o(z) \leq\frac{C (\log N)^\xi}{q} +
\frac{C (\log N)^{2 \xi}}{\sqrt{N}} \leq C (\log
N)^{- 2\xi}
\]
with $(\xi,\nu)$-high probability.
\end{lemma}
\begin{pf}
Fix $z = E + \ii\eta\in D_L$ with $\eta\geq2$.
We shall repeatedly make use of the trivial estimates
%
%
\begin{equation} \label{trivialresolventestimates}
\bigl| G_{ij}^{(\bb T)} \bigr|\leq\frac{1}{\eta},\qquad
\bigl| m^{(\bb T)} \bigr|\leq\frac{1}{\eta},\qquad
| m_{\mathrm{sc}} |\leq\frac{1}{\eta},
\end{equation}
where $\bb T \subset\{1,\ldots, N\}$ is arbitrary. These estimates
follow immediately from the definitions of $G^{(\bb
T)}$ and $m_{\mathrm{sc}}$.

We begin by estimating $\Lambda_o$.
For $i \neq j$ we get, following the calculation in (\ref
{estimateofGij}) and recalling (\ref{ward}), with $(\xi,\nu)$-high probability,
\[
| G_{ij} |\leq\frac{C}{q} + o(1) \Lambda_o + C (\log N)^{2
\xi} \sqrt{ \frac{\im m^{(ij)}}{N \eta}} \leq
\frac{C}{q} + o(1) \Lambda_o + \frac{C (\log N)^{2 \xi}}{\sqrt
{N}}.
\]
Taking the maximum over $i \neq j$ yields with $(\xi,\nu)$-high probability
\[
\Lambda_o \leq\frac{C}{q} + \frac{C (\log N)^{2\xi}}{\sqrt
{N}}.
\]

What remains is an estimate on $\Lambda_d$. We begin by estimating
with $(\xi,\nu)$-high probability
\[
|\Upsilon_i |\leq\frac{C}{q} + | Z_i |+
|\cal A_i |.
\]
In order to estimate $|\cal A_i |$, we observe that (\ref
{Gijformula}) implies
\[
\frac{G_{ij}}{G_{ii}} = - G_{jj}^{(i)} (h_{ij} - Z_{ij})\qquad (i
\neq j).\vadjust{\goodbreak}
\]
Therefore, we have with $(\xi,\nu)$-high probability
%
%
\begin{eqnarray} \label{Aestimateforlargeeta}
|\cal A_i |&\leq&\frac{1}{N} | G_{ii} |+
\frac{1}{N}
\sum_{j}^{(i)} \bigl| G^{(i)}_{jj} \bigr|| G_{ji} |
(| h_{ij} |+ | Z_{ij} |) \nonumber\\[-8pt]\\[-8pt]
&\leq&\frac{C}{N} + C
\Lambda_o
\biggl({\frac{1}{q} + \sup_{i \neq j} | Z_{ij} |}\biggr)
\leq\frac{C}{q},\nonumber
\end{eqnarray}
where we used that with $(\xi,\nu)$-high probability
\[
| Z_{ij} |\leq(\log N)^{2 \xi} \biggl[{\frac{C}{q^2} +
\frac{\Lambda_o}{q} + \frac{C}{\sqrt{N}}}\biggr]
\]
as follows from (\ref{abB}) and (\ref{ward}). Similarly, from (\ref
{definitionofZi}) and using (\ref{aaBd}) and
(\ref{aaBo}), we find with $(\xi,\nu)$-high probability
\[
| Z_i |\leq\frac{C (\log N)^\xi}{q} + \frac{C (\log N)^{2
\xi}}{\sqrt{N}}.
\]
Thus we have proved that with $(\xi,\nu)$-high probability $|
\Upsilon_i |\leq C
(\log N)^\xi q^{-1} + C (\log N)^{2 \xi} N^{-1/2}$.

Next, using (\ref{identityofmsc}), we write the self-consistent
equation (\ref{self-consistent}) in the form
%
%
\begin{equation} \label{self-consistentequationlargeeta}
v_i = \frac{[v] - \Upsilon_i}{(z + m_{\mathrm{sc}} + [v] - \Upsilon_i)(z
+ m_{\mathrm{sc}})}.
\end{equation}
The denominator of (\ref{self-consistentequationlargeeta}) is with
$(\xi,\nu)$-high probability larger in absolute value than
\[
\bigl({2 - 1 - O\bigl({(\log N)^{\xi} q^{-1} + (\log N)^{2 \xi}
N^{-1/2}}\bigr)}\bigr) 2 \geq3/2,
\]
since $| z + m_{\mathrm{sc}} |= | m_{\mathrm{sc}} |^{-1} \geq2$ and
$|[v] |
\leq1$ by (\ref{trivialresolventestimates}). Thus,
\[
| v_i |\leq\frac{\Lambda_d + O({(\log N)^{\xi}
q^{-1} + (\log N)^{2 \xi} N^{-1/2}})}{3/2},
\]
which yields, after taking the maximum over $i$,
\[
\Lambda_d \leq\frac{\Lambda_d + O({(\log N)^{\xi}
q^{-1} + (\log N)^{2 \xi} N^{-1/2}})}{3/2}.
\]
This completes the estimate of $\Lambda_d$, and hence the proof.
\end{pf}

\subsection{\texorpdfstring{Dichotomy argument for $\Lambda$}
{Dichotomy argument for Lambda}}\label{subsecdic}
The following dichotomy argument serves as the basis for the continuity
argument of Section~\ref{subseccon}.

We introduce the control parameters
%
%
\begin{equation}\label{58}
\alpha\deq\biggl|\frac{
1-m^2_{\mathrm{sc}}}{m_{\mathrm{sc}}^{3}}
\biggr|,\qquad
\beta\deq\frac{(\log N)^{\xi}}{\sqrt{q}} + \frac
{(\log N)^{4 \xi/3}}{(N \eta)^{1/3}},
\end{equation}
where $\al=\al(z)$ and $\beta=\beta(z)$ depend on the spectral
parameter $z$.
For
any $z\in D_L$ we have the bound $\beta\leq(\log N)^{-\xi}$.

From Lemma~\ref{lemmamsc} it also follows that there is a constant
constant $K\geq1$, depending only on $\Sigma$, such that
%
%
\begin{equation}\label{Kdef}
\frac{1}{K}\sqrt{\kappa+\eta} \leq\alpha(z) \leq K\sqrt
{\kappa+\eta}
\end{equation}
for any $z\in D_L$.

We shall fix $E$ and vary $\eta$ from $2$ down to $(\log N)^L N^{-1}$.
Since $\sqrt{\kappa+ \eta}$ is increasing and $\beta(E + \ii\eta
)$ decreasing in $\eta$, we find that, for any $U >
1$, the equation $\sqrt{\kappa+\eta}= 2 U^2K\beta(E + \ii\eta)$
has a unique solution $\eta$, which we denote by
$\wt\eta=\wt\eta(U, E)$ (recall that $\kappa= || E |
- 2 |$
is independent of $\eta$). Moreover, it is easy to see
that for any fixed $U$ we have
%
%
\begin{equation}\label{teta}
\wt\eta\ll1.
\end{equation}

%
\begin{lemma}[(Dichotomy)]\label{592}
There exists a constant $U_0$ such that, for any fixed $U\geq U_0$,
there exists a constant $C_1(U)$, depending only on
$U$, such that the following estimates hold for any $z = E + \ii\eta
\in D_L$:
%
%
\begin{eqnarray}
\label{81}
\Lambda(z) & \leq & U \beta(z) \quad\mbox{or}\quad \Lambda
(z) \geq\frac{\alpha(z)}{ U }
\qquad\mbox{if } \eta\geq\wt\eta(U,E),\\
\label{82}
\Lambda(z) & \leq & C_1(U) \beta(z) \qquad\mbox{if } \eta<
\wt\eta(U,E)
\end{eqnarray}
on $\Omega(z)$ with $(\xi,\nu)$-high probability and for
sufficiently large $N$.
\end{lemma}
\begin{pf}
Fix $z = E + \ii\eta\in D_L$.
From (\ref{442}) and Lemma~\ref{lemmamsc} we find
\[
\frac{1 - m_{\mathrm{sc}}^2}{m_{\mathrm{sc}}^3} [v] = [v]^2 + O \biggl({\frac
{\Lambda^2 }{\log N}}\biggr) + O \biggl({\frac{(\log N)^\xi}{q} +
\sqrt{\beta^3 \Lambda+ \beta^3 \alpha}}\biggr)
\]
with $(\xi,\nu)$-high probability.
The third term on the right-hand side is bounded by $C^* (\beta\Lambda
+ \alpha\beta+ \beta^2)$ for some constant
$C^* \geq1$. We set $U_0 \deq9 (C^* + 1)$. We conclude that in
$\Omega(z)$ we have with $(\xi,\nu)$-high probability
%
%
\begin{equation} \label{dich}
\biggl|\frac{1 - m_{\mathrm{sc}}^2}{m_{\mathrm{sc}}^3} [v] - [v]^2 \biggr|
\leq O \biggl({\frac{\Lambda^2 }{\log N}}\biggr) + C^* (\beta
\Lambda+
\alpha\beta+ \beta^2).
\end{equation}

Depending on the size of $\beta$ relative to $\al$, which is
determined by $z$, we shall estimate either $ [v] $ or $
[v] ^2$ using (\ref{dich}). This gives rise to the two cases in
Lem\-ma~\ref{592}.

\textit{Case} 1: $\eta\geq\wt\eta$. From the definition of $\wt\eta$
and $C^*$ we find that
%
%
\begin{equation} \label{newalz}
\beta\leq\frac{\alpha}{2 U^2} \leq\frac{\alpha}{2 C^*}
\leq\alpha.
\end{equation}
Recalling that $\Lambda= |[v] |$, we therefore obtain from
(\ref
{dich}) with $(\xi,\nu)$-high probability
\[
\alpha\Lambda\leq2 \Lambda^2 + C^*(\beta\Lambda+ \alpha
\beta+ \beta^2) \leq2 \Lambda^2 + \frac{\alpha
\Lambda}{2} + 2 C^* \alpha\beta,\vadjust{\goodbreak}
\]
which gives
\[
\alpha\Lambda\leq4 \Lambda^2 + 4 C^* \alpha\beta.
\]
Thus, either $\alpha\Lambda/ 2 \leq4 \Lambda^2$ which implies
$\Lambda\geq\alpha/ 8 \geq\alpha/ U$, or $\alpha
\Lambda/ 2 \leq4 C^* \alpha\beta$ which implies $\Lambda\leq8 C^*
\beta\leq U \beta$. This proves (\ref{81}).

\textit{Case} 2: $\eta< \wt\eta$.
In this case the definition of $\wt\eta$ yields $\al\leq
2U^2K^2\beta$.
We express $ |[v] |^2 = \Lambda^2$ from
(\ref{dich}) and we get
%
%
\begin{equation}\label{13-2}
\Lambda^2 \leq2 \alpha\Lambda+ 2 C^* (\beta\Lambda+ \alpha
\beta+ \beta^2) \leq C' \beta\Lambda+ C'
\beta^2
\end{equation}
for some constant $C'$ depending on $U$. Now (\ref{82}) is an
immediate consequence.
\end{pf}

\subsection{\texorpdfstring{Continuity argument: Conclusion of the proof of Theorem \protect\ref{WLSCTHM}}
{Continuity argument: Conclusion of the proof of Theorem 3.1}}\label{subseccon}

We complete the proof of Theorem~\ref{WLSCTHM} using a continuity
argument in $\eta$ to go from $\eta= 2$ down to $\eta
= N^{-1} (\log N)^L$. We focus first on proving (\ref{weakSCL}).
We use Lemma~\ref{5992} for the initial estimate, and the dichotomy in
Lemma~\ref{592} to propagate a strong estimate on
$\Lambda$ to smaller values of $\eta$.

Choose a decreasing finite sequence $\eta_k$, $k=1,2,\ldots, k_0$,
satisfying $k_0\leq CN^{8}$, $|\eta_k-\eta_{k+1} |\leq
N^{-8}$, $\eta_1 = 2$, and $\eta_{k_0}= N^{-1}(\log N)^{L}$. We fix
$E \in[-\Sigma,\Sigma]$ and set $z_k \deq E + \ii\eta_k$.
Throughout this section we fix a $U\geq U_0$ in Lemma~\ref{592}, and
recall the definition of $\wt\eta(U,E)$ from Section
\ref{subsecdic}.

Consider first $z_1$. It is easy to see that, for large enough $N$, we
have $\eta_1\geq\wt\eta(U,E)$, for any $E \in
[-\Sigma,\Sigma]$. Therefore, Lemmas~\ref{5992} and~\ref{592} imply
that both $\Omega(z_1)$ and
\[
\Lambda(z_1) \leq U \beta(z_1)
\]
hold with $(\xi,\nu)$-high probability.
This estimate takes care of the initial point~$\eta_1$. The next lemma
extends this result to all $k \leq k_0$.
%
%
\begin{lemma}\label{lmind}
Define the event
\[
\Omega_k \deq\Omega(z_k) \cap\bigl\{{\Lambda(z_k) \leq
C^{(k)}(U) \beta(z_k)}\bigr\},
\]
where
\[
C^{(k)}(U) \deq
\cases{
U, &\quad if $\eta_k \geq\wt\eta(U,E)$,\cr
C_1(U), &\quad if $\eta_k < \wt\eta(U,E)$.}
\]
Then
%
%
\begin{equation} \label{continuityprobestimate}
\P(\Omega_k^c) \leq2 k \me^{-\nu(\log N)^\xi}.
\end{equation}
\end{lemma}
\begin{pf}
We proceed by induction on $k$. The case $k = 1$ was just proved. Let
us therefore assume that (\ref{continuityprobestimate}) holds for $k$.
We need to estimate
%
%
\begin{eqnarray} \label{Pk+1Pk}\qquad
\P(\Omega_{k+1}^c) &\leq&\P\bigl({\Omega_k \cap\Omega
(z_{k+1}) \cap\Omega_{k + 1}^c}\bigr) + \P\bigl({\Omega_k \cap
(\Omega(z_{k+1}))^c}\bigr)+
\P(\Omega_k^c)\nonumber\\[-8pt]\\[-8pt]
&=& B + A + \P(\Omega_k^c),\nonumber
\end{eqnarray}
where we defined
\begin{eqnarray*}
A & \deq & \P[{\Omega_k \cap\{{\Lambda_d(z_{k+1}) +
\Lambda_o(z_{k+1}) > (\log N)^{-\xi}}\}}],
\\
B & \deq & \P\bigl[{\Omega_k \cap\Omega(z_{k+1}) \cap\bigl\{
{\Lambda(z_{k+1}) > C^{(k+1)}(U) \beta(z_{k+1})}\bigr\}}\bigr].
\end{eqnarray*}
We begin by estimating $A$. For any $i,j$, we have
%
%
\begin{eqnarray} \label{w-z}
| G_{ij}(z_{k+1}) - G_{ij}(z_k) |&\leq&| z_{k+1} -
z_k |
\sup_{z \in D_L} \biggl|\frac{\partial G_{ij}(z)}{\partial z}
\biggr|\nonumber\\[-8pt]\\[-8pt]
&\leq& N^{-8} \sup_{z \in D_L}\frac{1}{(\im z)^2}
\leq N^{-6}.\nonumber
\end{eqnarray}
Therefore, by (\ref{Gijapriori}) and (\ref{436}), we have on
$\Omega_k$ with $(\xi,\nu)$-high probability
\begin{eqnarray*}
\Lambda_d(z_{k+1}) + \Lambda_o(z_{k+1}) &\leq&\Lambda_d(z_k) +
\Lambda_o(z_k) + 2 N^{-6} \\
&\leq& C \biggl({\frac{(\log N)^\xi
}{q} + (\log N)^{2 \xi} \Psi(z_k)}\biggr) + \Lambda(z_k)
\\
&\leq&
C \beta(z_k) \ll(\log N)^{-\xi}.
\end{eqnarray*}
Thus, we find that $A \leq\me^{- \nu(\log N)^\xi}$.

Next, we estimate $B$. Suppose first that $\eta_k \geq\wt\eta(U,
E)$. Then, similarly to (\ref{w-z}), we find
$|\Lambda(z_{k+1}) - \Lambda(z_k) |\leq N^{-6}$. Thus, we find
on $\Omega_k$ with $(\xi,\nu)$-high probability
%
%
\begin{equation} \label{611}
\Lambda(z_{k+1}) \leq\Lambda(z_k) + N^{-6} \leq U \beta
(z_k) + N^{-6} \leq\frac{3 U}{2} \beta(z_{k+1}).
\end{equation}
Suppose now that $\eta_{k+1} \geq\wt\eta(U, E)$. Then from (\ref
{611}) and (\ref{newalz}) we find $\Lambda(z_{k+1}) <
\frac{\alpha(z_{k+1})}{U}$. Now the dichotomy of (\ref{81}) yields
on $\Omega_k \cap\Omega(z_{k+1})$ with $(\xi,\nu)$-high
probability that
$\Lambda(z_{k+1}) \leq U \beta(z_{k+1})$. On the other hand, if $\eta
_{k+1} < \wt\eta(U, E)$, then (\ref{611})
immediately implies $\Lambda(z_{k+1}) \leq C_1(U) \beta(z_{k+1})$.
This concludes the proof of $B \leq\me^{-\nu(\log
N)^\xi}$ if $\eta_k \geq\wt\eta(U,E)$.

Finally, suppose that $\eta_k < \wt\eta(U,E)$. Thus, we also have
$\eta_{k+1} < \wt\eta(U,E)$. In this case we
immediately get from (\ref{82}) on $\Omega(z_{k+1})$ with $(\xi,\nu
)$-high probability
$\Lambda(z_{k+1}) \leq C_1(U) \beta(x_{k+1})$.

We have therefore proved, for all $k$, that $\P(\Omega_{k+1}^c)
\leq2 \me^{-\nu(\log N)^\xi} + \P(\Omega^c_k)$,
and the claim follows.
\end{pf}

In order to complete the proof of Theorem~\ref{WLSCTHM}, we invoke the
following simple lattice argument which
strengthens the result of Lemma~\ref{lmind} to a statement uniform in
$z \in D_L$. The main ingredient is the Lipschitz
continuity of the map $z \mapsto G_{ij}(z)$, with a Lipschitz constant
bounded by $\eta^{-2}\leq N^2$.
%
%
\begin{corollary}\label{cor414}
There is a constant $C$ such that
%
%
\begin{equation} \label{latticeestimate}
\P\biggl[{\bigcup_{z \in D_L} ({\Omega(z)})^c}\biggr]
+ \P\biggl[{\bigcup_{z \in D_L} \{{\Lambda(z) > C \beta
(z)}\}}\biggr] \leq\me
^{-\nu
(\log N)^\xi}.
\end{equation}
\end{corollary}
\begin{pf}
Take a lattice $\cal L \subset D_L$ such that $|\cal L |\leq C
N^6$ and for any $z\in D_L$ there is a $\tilde z \in
\cal L$ satisfying $| z - \tilde z |\leq N^{-3}$. From the
definition of $G$ it is easy to see that for $z, \tilde
z\in D_L$
%
%
\begin{equation} \label{net1}
| G_{ij}(z) - G_{ij}(\tilde z) |\leq\eta^{-2} | z
- \tilde z |\leq\frac{1}{N}.
\end{equation}
The same bound holds for $| m(z) - m(\tilde z) |$. Moreover, Lemma
\ref{lmind} immediately yields
%
%
\begin{equation} \label{net2}
\P\biggl[{\bigcap_{\tilde z \in\cal L} \biggl\{{\Lambda(\tilde z)
\leq\frac{C}{2} \beta(\tilde z)}\biggr\}}\biggr] \geq1 - \me
^{-\nu
(\log N)^\xi}
\end{equation}
for some $C$ large enough and some $\nu> 0$. From (\ref{net1}),
(\ref{net2}) and $N^{-1} \ll\beta(z)$ we get
\[
\P\biggl[{\bigcup_{z \in D_L} \{{\Lambda(z) > C \beta
(z)}\}}\biggr] \leq
\me^{-\nu(\log N)^\xi}.
\]
The first term of (\ref{latticeestimate}) is estimated similarly.
\end{pf}

We have proved (\ref{weakSCL}). In order to prove (\ref{weakGij}),
we note that (\ref{weakSCL}), (\ref{Gijapriori})
and (\ref{latticeestimate}) imply
\[
\Lambda_o(z) \leq\frac{C}{q} + \frac{C (\log N)^{2 \xi
}}{\sqrt{N \eta}}
\]
with $(\xi,\nu)$-high probability.
Now a lattice argument analogous to Corollary~\ref{cor414} yields
(\ref{weakGij}). The diagonal estimate (\ref{weakGii}) follows
similarly using (\ref{lambdadleqlambda}). This
concludes the proof of Theorem~\ref{WLSCTHM}.

\section{\texorpdfstring{Proof of Theorem \protect\ref{LSCTHM}}{Proof of Theorem 2.8}}
In the previous section we proved Theorem~\ref{WLSCTHM}, which is
weaker than the main result Theorem~\ref{LSCTHM}
(strong local semicircle law),
but will be used as an a priori bound in the proof of Theorem \ref
{LSCTHM}. The key ingredient that allows us to
strengthen Theorem~\ref{WLSCTHM} to Theorem~\ref{LSCTHM} is the
following lemma, which shows that $[Z]$, the average of
the $Z_i$'s, is much smaller than that of a typical~$Z_i$.
(Notice that in the proof of Theorem~\ref{WLSCTHM}, to arrive at
(\ref{442}), $[Z]$ was estimated
by the same quantity as each
individual~$Z_i$.)
This lemma is analogous to Lemma 5.2 in~\cite{EYY2} and Corollary 4.2
in~\cite{EYYrigidity}, but we will present a new
proof (in Section~\ref{sectproofof41}), which admits sparse matrix
entries and effectively tracks the dependence of
the exponent~$p$. Our new proof is based on an abstract decoupling
result, Theorem~\ref{abstractZlemma} below, which is
useful in other contexts as well, such as for proving Proposition~\ref
{propaverageforedge} below.
%
%
\begin{lemma}\label{zlemmasparse}
Recall the notation $[Z] = \frac1N \sum_i Z_i$. Suppose that $\xi$
satisfies (\ref{boundsonxi}),
$q\geq(\log N)^{5\xi}$ and that there exists
$\wt D\subset D_L$ with $L\geq14\xi$ such that we have with $(\xi
,\nu)$-high probability
%
%
\begin{equation}\label{111}
\Lambda(z) \leq\gamma(z) \qquad\mbox{for } z \in\wt D,
\end{equation}
where $\gamma$ is a deterministic function satisfying $\gamma(z)\leq
(\log N)^{-\xi}$. Then we have
with $(\xi-2,\nu)$-high probability
%
%
\begin{eqnarray}\label{reszlemmasparse}
|[Z](z) |\leq(\log N)^{14\xi} \biggl(\frac
1{q^2}+\frac
1{(N\eta)^2}
+ (\log N)^{ 4\xi}\frac{\im m_{\mathrm{sc}} (z)
+\gamma(z)}{N\eta}\biggr)\nonumber\\[-8pt]\\[-8pt]
&&\eqntext{\mbox{for } z \in\wt D.}
\end{eqnarray}
In particular, by (\ref{finalself-consistentequation}), we have with
$(\xi-2,\nu)$-high probability
%
%
\begin{eqnarray}\label{reszlemmasparse2}\qquad
&&
\biggl|\frac{1 - m_{\mathrm{sc}}^2}{m_{\mathrm{sc}}^3} [v] -
[v]^2\biggr|\nonumber\\[-8pt]\\[-8pt]
&&\qquad
\leq C \frac{\Lambda^2}{\log N}
+ C (\log N)^{14\xi} \biggl(\frac1{q^2}+\frac1{(N\eta)^2}
+ (\log N)^{ 4\xi}\frac{\im
m_{\mathrm{sc}}+\gamma}{N\eta}\biggr)\nonumber
\end{eqnarray}
for any value of the spectral parameter $z\in\wt D$.
\end{lemma}

The proof of Lemma~\ref{zlemmasparse} is given in Section \ref
{sectionZ-lemma}. In this section we use it to prove
Theorem~\ref{LSCTHM} and to derive an estimate on $\| H \|$ (Lemma
\ref{lemmaboundonHtilde}).

The basic idea behind the proof of Theorem~\ref{LSCTHM} using Lemma
\ref{zlemmasparse} is to iterate
(\ref{reszlemmasparse}) in order to obtain successively better
estimates for $\Lambda$. Each step of the iteration
improves the power $1 - \tau$ of the control parameter $(q^{-1} + (N
\eta)^{-1})^{1 - \tau}$. The iteration is started
with the weak local semicircle law, Theorem~\ref{WLSCTHM}, which
yields $1 - \tau= 1/3$. At each step of the iteration,
$\tau$ is halved at the expense of reducing the parameter $\xi$ to
$\xi- 2$, thus reducing the probability on which the
estimate holds. This iteration procedure is repeated an order $\log
\log N$ times, which allows us effectively to reach
$\tau= 0$.

The iteration step is based on the following lemma, which is entirely
deterministic.
%
%
\begin{lemma}\label{tauhalftau}
Let $1\leq\xi_1\leq\xi_2$ and $q>1$. Let $0 < \tau< 1$ and $L >
1$. Suppose that there is a number $\gamma(z)$
satisfying
%
%
\begin{equation}\label{112}
\gamma(z) \leq(\log N)^{19 \xi_2}\biggl(\frac1q+\frac
{1}{N\eta}\biggr)^{1-\tau} \qquad\mbox{for } z \in D_L
\end{equation}
such that (\ref{111}) holds with $\wt D \deq D_L$.
We also assume that
%
%
\begin{eqnarray}\label{reszlemmasparse3}
&&\biggl|\frac{1 - m_{\mathrm{sc}}^2}{m_{\mathrm{sc}}^3} [v] - [v]^2\biggr| \nonumber\\
&&\qquad\leq C
\frac{\Lambda^2}{\log N}
+ C (\log N)^{14\xi_1} \biggl(\frac1{q^2}+\frac1{(N\eta)^2}
+ (\log N)^{ 4\xi_1}\frac{\al+\gamma}{N\eta}\biggr) \\
&&\eqntext{\mbox
{for } z \in
D_L,}
\end{eqnarray}
where $\alpha$ was defined in (\ref{58}). Finally, we assume that if
$\eta\sim1$, then
%
%
\begin{equation}\label{iLam1}
\Lambda(z) \ll1.
\end{equation}
Then we have
%
%
\begin{equation}\label{yrbj}
\Lambda(z) \leq(\log N)^{19 \xi_2}\biggl(\frac1q+\frac
{1}{N\eta}\biggr)^{1-\tau/2}
\end{equation}
for $z \in D_L$ and large enough $N$.
\end{lemma}
\begin{pf}
The proof is based on a dichotomy argument.
Define
%
%
\begin{equation}\label{al0def}
\al_0(z) \deq(\log N)^{(18+3/4)\xi_2}\biggl(\frac
1q+\frac{1}{N\eta}\biggr)^{1-\tau/2}.
\end{equation}
We consider two cases.

\textit{Case} 1: $\al\leq10 \al_0$.
Using the estimate (\ref{112}), we find
%
%
\begin{equation}\label{113}
\frac1{q^2}+\frac1{(N\eta)^2}
+ (\log N)^{ 4\xi_1}\frac{\gamma}{N\eta}
\leq2(\log N)^{23 \xi_2}\biggl(\frac1q+\frac{1}{N\eta}
\biggr)^{2-\tau}.
\end{equation}
Now in (\ref{reszlemmasparse3}) we may absorb the term $\Lambda^2 /
\log N$ into the term $|[v] |^2$ on the left-hand
side, at the expense of a constant $2$. Then we complete the square on
the left-hand side and take the square root of
the resulting equation; this yields
%
%
\begin{equation}\label{5858}\quad
\Lambda\leq4 \al+ C(\log N)^{{37}\xi_2/{2}}\biggl(\frac
1q+\frac{1}{N\eta}\biggr)^{1-\tau/2}
+C(\log N)^{9\xi_1}\sqrt{\frac{\al}{N\eta}},
\end{equation}
where we used (\ref{113}). Now (\ref{yrbj}) follows from (\ref{5858}).

\textit{Case} 2: $\al\geq10 \al_0$. Let us assume that $\Lambda\leq
\al/2$. Then in (\ref{reszlemmasparse3}) the terms
$[v]^2$ and $\Lambda^2$ can be absorbed into the term $\alpha|
[v] |$, so that we get
%
%
\begin{eqnarray}\label{5959}
\Lambda&\leq& C \frac{(\log N)^{14\xi_1}}{\al}\biggl(\frac
1q+\frac{1}{N\eta}\biggr)^{2}+ C
(\log N)^{18\xi_1}\frac{\gamma}{N\eta\alpha}\nonumber\\[-8pt]\\[-8pt]
&&{}+C(\log N)^{18\xi
_1}\frac1{N\eta}.\nonumber
\end{eqnarray}
By the definitions of $\gamma$ and $\alpha_0$, we have
%
%
\begin{equation}
C (\log N)^{18\xi_1}\frac{\gamma}{N\eta\alpha} \leq\frac{
\alpha_0 } { (\log N)^{1/4} } \leq\frac{10
\al}{(\log N)^{1/4}}
\end{equation}
and the first term in the right-hand side of (\ref{5959}) is bounded
by $\al/\log N$ thanks to (\ref{al0def}). The last
term can be estimated similarly.
Hence, (\ref{5959}) implies that $\Lambda\leq\al/4$ provided that
$\Lambda\leq\al/2$.

In other words, if $\al\geq10\al_0$, then either $\Lambda>\al/2$
or $\Lambda\leq\al/4$. Using the continuity of
$\Lambda(z)$ and $\alpha= \alpha(z)$ in $\eta=\im z$, and the assumption
\[
\Lambda(z) \ll1 = O(\al)
\]
for $\eta\sim1$,
we get $\Lambda\leq\al/4$ on the whole $D_L$. Together with (\ref
{5959}), we obtain~(\ref{yrbj}).
\end{pf}
\begin{pf*}{Proof of Theorem~\ref{LSCTHM}}
The main work is to prove Theorem~\ref{LSCTHM} for spectral parameters
$z \in D_L$, where
%
%
\begin{equation} \label{defofL}
L \deq120 \xi.
\end{equation}
Once this is done, the extension to all $z \in D$ is relatively
straightforward, and is given at the end of the proof.
Recall the definitions (\ref{defPsi}) of $\Psi$ and (\ref{defOmegaz})
of $\Omega(z)$. It is clear that if $D$ is replaced
everywhere by $D_L$, then (\ref{Gijestimate}) follows from (\ref
{scm}), (\ref{Gijapriori}) and (\ref{lambdadleqlambda}).
Therefore, we only need to
prove (\ref{scm}).

We begin by introducing
%
%
\begin{equation}
\wt\xi\deq2(\log\log N/\log2) +\xi.
\end{equation}
By the assumptions (\ref{assumptionsforSLSC}) and (\ref{defofL}),
we have $\wt\xi\leq3 \xi/2 \leq A_0 \log\log
N$, $L\geq60\wt\xi$, and $q\geq(\log N)^{60\wt\xi}$. To prove
(\ref{scm}) with $D$ replaced by $D_L$, it therefore suffices to establish
%
%
\begin{equation}\label{scmfake}
\bigcap_{z \in D_L} \biggl\{{| m(z) - m_{\mathrm{sc}}(z)
|\leq(\log N)^{20\wt\xi} \biggl({\min\biggl\{{\frac{
(\log N)^{20\wt\xi} }{\sqrt{\kappa_E+\eta}}\frac{1}{q^2}, \frac
{1}{q}}\biggr\} + \frac{1}{N \eta}}\biggr)}\biggr\}\hspace*{-35pt}
\end{equation}
with $(\xi,\nu)$-high probability.

The weak local semicircle law, Theorem~\ref{WLSCTHM} with $\wt\xi$
replacing $\xi$, yields
%
%
\begin{eqnarray}\label{s413}
\Lambda&\leq&(\log N)^{2\wt\xi} \biggl(\frac{1}{q}+\frac
{1}{N\eta}\biggr)^{1/3} \nonumber\\[-8pt]\\[-8pt]
&\leq&(\log N)^{19\wt\xi}
\biggl(\frac{1}{q}+\frac{1}{N\eta}\biggr)^{1-2/3} \qquad\mbox{for
} z\in D_L\nonumber
\end{eqnarray}
with $(\wt\xi,\nu)$-high probability. Thus, (\ref{111}) holds
with
%
%
\begin{equation}
\gamma(z) \deq(\log N)^{19\wt\xi} \biggl(\frac{1}{q}+\frac
{1}{N \eta}\biggr)^{1/3}.
\end{equation}
With\vspace*{1pt} $L \geq60\wt\xi$ and $q\geq(\log N)^{60\wt\xi}$, we also
have $ \gamma\leq(\log N)^{-\wt\xi} $. Thus, Lem\-ma~\ref{zlemmasparse} implies that, with $\wt\xi$ replacing $\xi$ and
$\wt D= D_L$, the statement
%
%
\begin{eqnarray}\quad
&&\biggl|\frac{1 - m_{\mathrm{sc}}^2}{m_{\mathrm{sc}}^3} [v] -
[v]^2\biggr|\nonumber\\
&&\qquad\leq C\frac{\Lambda^2}{\log N}
+ C(\log N)^{14\wt\xi} \biggl(\frac1{q^2}+\frac1{(N\eta)^2}
+ (\log N)^{ 4\wt\xi}\frac{\im
m_{\mathrm{sc}}+\gamma}{N\eta}\biggr)\\
&&\eqntext{\mbox{for }
z\in D_L}
\end{eqnarray}
holds with $(\wt\xi-2,\nu)$-high probability. This implies (\ref
{reszlemmasparse3}),
with the choice $\xi_1=\wt\xi$ in
(\ref{reszlemmasparse3}), since $\im m_{\mathrm{sc}} \leq C \alpha$.
Moreover, $\gamma$ satisfies (\ref{112}) with $\xi_2=\wt\xi$ and
$\tau=2/3$. We also find that $\Lambda$ satisfies
(\ref{iLam1}), since $\Lambda\leq\gamma\leq(\log N)^{-\wt\xi}$
[see~(\ref{s413})].
We may therefore apply Lemma~\ref{tauhalftau} with $\xi_1=\xi_2=\wt
\xi$ to get that
%
%
\begin{equation}\label{g417}
\Lambda\leq(\log N)^{19\wt\xi} \biggl(\frac{1}{q}+\frac{1}{N
\eta}\biggr)^{1-1/3} \qquad\mbox{for } z\in D_L
\end{equation}
holds with $(\wt\xi-2,\nu)$-high probability.
We now repeat this process $M$ times, each iteration yielding a
stronger bound on $\Lambda$ which holds with a smaller
probability. After $M$ iterations we get that
%
%
\begin{equation}\label{z418}
\Lambda\leq(\log N)^{19\wt\xi} \biggl(\frac{1}{q}+\frac{1}{N
\eta}\biggr)^{1-2(1/2)^M/3} \qquad\mbox{for } z \in D_L
\end{equation}
holds with $(\wt\xi-2M,\nu)$-high probability.

To clarify the iteration, we spell out the details of the second step.
We start from (\ref{g417}) and define $\gamma$ as
the right-hand side of (\ref{g417}),
%
%
\begin{equation}
\gamma(z) \deq(\log N)^{19\wt\xi} \biggl(\frac{1}{q}+\frac
{1}{N \eta}\biggr)^{1-1/3}.
\end{equation}
Thus, Lemma~\ref{zlemmasparse}, with $\wt\xi-2$ replacing $\xi$,
implies that
%
%
\begin{eqnarray}
&&\biggl|\frac{1 - m_{\mathrm{sc}}^2}{m_{\mathrm{sc}}^3} [v] -
[v]^2\biggr|\nonumber\\
&&\qquad \leq C\frac{\Lambda^2}{\log N}\nonumber\\[-8pt]\\[-8pt]
&&\qquad\quad{}+ C(\log N)^{14(\wt\xi-2)} \biggl(\frac1{q^2}+\frac1{(N\eta)^2}
+ (\log N)^{ 4(\wt\xi-2)}\frac{\im
m_{\mathrm{sc}}+\gamma}{N\eta}\biggr)\nonumber\\
&&\eqntext{\mbox{for } z\in D_L}
\end{eqnarray}
holds with $(\wt\xi-4,\nu)$-high probability. We now apply Lemma
\ref{tauhalftau} with
$\xi_1=\wt\xi-2$, $\xi_2=\wt\xi$ and $\tau=1/3$.
[Similarly, in the $k$th step we set $\xi_1 = \wt\xi-2(k-1)$, $\xi
_2=\wt\xi$, and
$\tau=(2/3)(1/2)^{k-1}$.] This shows
that
%
%
\begin{equation}\label{g417new}
\Lambda\leq(\log N)^{19\wt\xi} \biggl(\frac{1}{q}+\frac{1}{N
\eta}\biggr)^{1-1/6} \qquad\mbox{for } z\in D_L
\end{equation}
holds with $(\wt\xi-4,\nu)$-high probability. This is (\ref{z418})
for $M = 2$.

Now we return to (\ref{z418}) and choose $M \deq[\log\log N / \log
2] - 1$ (where $[\cdot]$ denotes the integer
part). Using $q^{-1}+(N\eta)^{-1} \geq c N^{-1/2}$ [by (\ref
{lowerboundond})], we get
\[
\biggl(\frac{1}{q}+\frac{1}{N \eta}\biggr)^{-2(1/2)^M/3}
\leq C \leq(\log N)^{\wt\xi/2}.
\]
Thus,
%
%
\begin{equation}\label{JWL516}
\Lambda\leq(\log N)^{{39}\wt\xi/{2}} \biggl(\frac
{1}{q}+\frac{1}{N \eta}\biggr) \qquad\mbox{for } z\in D_L
\end{equation}
holds with $(\xi+2,\nu)$-high probability. Recalling (\ref{Kdef}),
we find that (\ref
{JWL516}) implies (\ref{scmfake}), unless
%
%
\begin{equation}\label{419}
(\log N)^{-{39}\wt\xi/{2}}\al\geq\frac{1}{q} \geq
\frac{1}{N \eta}.
\end{equation}

Let us therefore assume that (\ref{419}) holds. Then it remains to
prove that with $(\xi,\nu)$-high probability
%
%
\begin{equation} \label{JWL517}
\Lambda\leq(\log N)^{40\wt\xi} \frac{1}{\al q^2}
+ (\log N)^{20\wt\xi} \frac{1}{N\eta}
\qquad\mbox{for } z\in D_L.
\end{equation}
Defining $\gamma$ as the right-hand side of (\ref{JWL516}), we use
Lemma~\ref{zlemmasparse}, with $ \xi+2$ replacing~$\xi$, to get
%
%
\begin{equation}
\biggl|\frac{1 - m_{\mathrm{sc}}^2}{m_{\mathrm{sc}}^3} [v] - [v]^2\biggr|
\leq
C\frac{\Lambda^2}{\log N}
+C (\log N)^{18\wt\xi} \biggl(\frac1{q^2} + \frac{\al}{N\eta
}\biggr)
\end{equation}
with $(\xi,\nu)$-high probability, where we used (\ref{419}) and
$|{\im m_{\mathrm{sc}} (z)}| \leq
C \alpha(z)$. We can estimate the term $[v]^2$ by
(\ref{JWL516}) and (\ref{419}), so that
%
%
\begin{equation}
\al\Lambda= \al|[v] | \leq C (\log N)^{39\wt\xi} \frac
1{q^2} + C (\log N)^{18\wt\xi} \frac{\al
}{N\eta}.
\end{equation}
This yields (\ref{JWL517}) and hence completes the proof of (\ref
{scm}) with $D$ replaced with~$D_L$. (Recall the simple lattice argument
of Corollary~\ref{cor414}.)

What remains is to extend (\ref{scm}) and (\ref{Gijestimate}) from
$z \in D_L$ to $z \in D$.
Let us therefore assume that $z = E + \ii\eta\in D$ with $0 < \eta
\leq\wt\eta\deq(\log N)^L N^{-1}$. For any $i,j = 1,\ldots, N$ we
get the bound
\[
| G_{ij}(E + \ii\eta) |= \biggl|\sum_\alpha
\frac{\ol{u}_\alpha(i) u_\alpha(j)}{\lambda_\alpha- z}
\biggr|\leq\max_l \sum
_\alpha\frac{| u_\alpha(l) |^2}{|\lambda_\alpha- z
|}.
\]
We define the dyadic decomposition of the eigenvalues
\[
U_0 \deq\{{\alpha\col|\lambda_\alpha- E |<
\eta
}\},\qquad
U_k \deq\{{\alpha\col2^{k - 1} \eta\leq|\lambda
_\alpha- E |< 2^k \eta}\}\qquad (k \geq1).
\]
This yields
\begin{eqnarray*}
\sum_\alpha\frac{| u_\alpha(l) |^2}{|\lambda_\alpha
- z |}
&=& \sum_{k \geq0} \sum_{\alpha\in U_k} \frac{| u_\alpha
(l) |^2}{|\lambda_\alpha- z |} \leq C \sum_{k
\geq0} \sum
_{\alpha\in U_k} \im\frac{| u_\alpha(l) |^2}{\lambda
_\alpha- E
- \ii2^k \eta} \\
&\leq& C \sum_{k \geq0} \im G_{ll}(E + \ii2^k
\eta).
\end{eqnarray*}

Next, we break the summation over $k$ into three pieces delimited by
$k_1 \deq\max\{{k \col2^k \eta< \wt\eta}\}$ and $k_2 \deq\max\{
{k \col2^k \eta< 3}\}$.
By spectral decomposition, it is easy to see that the function $y
\mapsto y \im G_{ll}(E + \ii y)$ is monotone
increasing. Therefore, we get
\begin{eqnarray*}
\sum_{k \geq0} \im G_{ll}(E + \ii2^k \eta) & \leq & \sum_{k =
0}^{k_1} \frac{\wt\eta}{2^k \eta} \im G_{ll}(E + \ii\wt\eta)\\
&&{} +
\sum_{k = k_1 + 1}^{k_2} \im G_{ll}(E + \ii2^k \eta) + \sum_{k =
k_2 + 1}^\infty\frac{1}{\eta2^k}
\\
& \leq & \frac{(\log N)^{C \xi}}{N \eta} + C (k_2 - k_1) + C
\\
& \leq & \frac{(\log N)^{C \xi}}{N \eta}
\end{eqnarray*}
with $(\xi,\nu)$-high probability, where in the second step we used
(\ref{Gijestimate}) for $z \in D_L$. Therefore, we have proved that
\[
\max_{i,j} | G_{ij}(E + \ii\eta) |\leq\frac{(\log N)^{C
\xi}}{N \eta}
\]
with $(\xi,\nu)$-high probability. This concludes the proof of
Theorem~\ref{LSCTHM}.
\end{pf*}

\subsection{\texorpdfstring{Estimate of $\| H \|$}{Estimate of ||H||}}
In this section we derive an upper bound on the norm of $H$.
A standard application of the moment method yields the following weak
bound on $\| H \|$. Its proof is given in the
\hyperref[app]{Appendix}.
%
%
\begin{lemma} \label{lemmaboundon_h_tildeweak}
Suppose that $H$ satisfies Definition~\ref{definitionofH}, that $\xi
$ satisfies (\ref{boundsonxi}) and that $q$
satisfies (\ref{lowerboundond}). Then with $(\xi,\nu)$-high
probability we have
%
%
\begin{equation} \label{boundonHtildeweak}
\| H \|\leq2 + (\log N)^{\xi} q^{-1/2}.
\end{equation}
\end{lemma}

Using the local semicircle law, Theorem~\ref{LSCTHM}, we may prove a
much stronger bound on $\| H \|$. Lemma
\ref{lemmaboundon_h_tildeweak} will be used as an a priori bound
in the proof of Lem\-ma~\ref{lemmaboundonHtilde}.\vadjust{\goodbreak}

%
\begin{lemma}\label{lemmaboundonHtilde}
Suppose that $H$ satisfies Definition~\ref{definitionofH}, and that
$\xi$ and $q$ satisfy (\ref{assumptionsforSLSC}).
Then with $(\xi,\nu)$-high probability we have
%
%
\begin{equation}\label{boundonHtilde}
\| H \|\leq2+(\log N)^{C\xi}(q^{-2}+N^{-2/3}
).
\end{equation}
\end{lemma}
\begin{pf}
We only consider the largest eigenvalue $\lambda_N = \max_\al\lambda
_\al$; the smallest eigenvalue $\lambda_1$ is
handled similarly. Set $L=120\xi$. Using (\ref{scm}) with $\xi+2$
replacing~$\xi$, we get with $(\xi+ 2,\nu)$-high probability
%
%
\begin{equation}\label{430J}
\Lambda(z)\leq(\log N)^{41 \xi} \biggl(\frac{1}{q} +\frac{1}{N
\eta} \biggr).
\end{equation}
Then applying Lemma~\ref{zlemmasparse} with
%
%
\begin{equation}
\gamma(z) \deq(\log N)^{41 \xi} \biggl(\frac{1}{q} + \frac
{1}{N \eta} \biggr)
\end{equation}
and $\xi+2$ replacing $\xi$, we have with $(\xi,\nu)$-high probability
%
%
\begin{eqnarray}\label{reszlemmasparse5}
&&\biggl|\frac{1 - m_{\mathrm{sc}}^2}{m_{\mathrm{sc}}^3} [v] - [v]^2\biggr|\nonumber\\
&&\qquad \leq
C\frac{\Lambda^2}{\log N}
+ C(\log N)^{C_1\xi} \biggl(\frac1{q^2}+\frac1{(N\eta)^2}
+ \frac{\im m_{\mathrm{sc}} }{N\eta}\biggr) \\
&&\eqntext{\mbox{for } z\in D_L,}
\end{eqnarray}
where $C_1$ is a sufficiently large constant.
Now if $E>2$ and $\kappa\geq\eta$, then Lem\-ma~\ref{lemmamsc} and
(\ref{Kdef}) yield
%
%
\begin{equation}\label{426QS}
\im m_{\mathrm{sc}} \sim\frac{\eta}{\sqrt\kappa},\qquad \al\sim
\sqrt{\kappa}.
\end{equation}
Inserting (\ref{426QS}) into (\ref{reszlemmasparse5}), we find with
$(\xi,\nu)$-high probability
%
%
\begin{eqnarray}\label{reszlemmasparse4}
\biggl|\frac{1 - m_{\mathrm{sc}}^2}{m_{\mathrm{sc}}^3} [v] - [v]^2\biggr| \leq C
\frac{\Lambda^2}{\log N}
+ C(\log N)^{C_1\xi} \biggl(\frac1{q^2}+\frac1{(N\eta)^2}
+ \frac{1 }{N\sqrt{\kappa}}\biggr).\hspace*{-34pt}
\end{eqnarray}

Next, for any fixed $C_1> 0$, we can find a large enough constant $C_2
> 2C_1$ such that if $E$ satisfies
%
%
\begin{equation}\label{E38}
2 + (\log N)^{C_2\xi}(q^{-2}+N^{-2/3}) \leq E \leq
3,
\end{equation}
then
%
%
\begin{equation}\label{911}
\min\{{N^{-1/2}\kappa^{1/4}, N^{-1}\kappa^{1/2}q^2,
\kappa}\} \geq(\log N)^{C_1\xi+2} N^{-1}\kappa^{-1/2}.
\end{equation}
(Here $\kappa=\kappa_E= E - 2$.)
From now on we assume that $E$ satisfies (\ref{E38}).
We define
%
%
\begin{equation}\label{E40}
\eta=\eta_E \deq(\log N)^{C_1\xi+1} N^{-1}\kappa^{-1/2}.
\end{equation}
Note that $\eta$ depends on $E$ via $\kappa$.
{F}rom (\ref{911}) we have
%
%
\begin{equation}\label{E412}
\kappa\geq\eta.
\end{equation}
Using (\ref{911}), (\ref{E40}) and (\ref{426QS}), we get
%
%
\begin{equation}\label{E41}
\frac{1}{N\eta} \gg\frac{\eta}{\sqrt\kappa} \sim\im
m_{\mathrm{sc}}(E+i\eta).
\end{equation}
Similarly, using (\ref{911}), we have
%
%
\begin{equation}\label{E411}
\frac{1}{N\eta} \geq\frac{1}{q^2 \sqrt{\kappa}}.
\end{equation}

Next, with the lower bound $\alpha\geq\sqrt\kappa/K$ from (\ref{Kdef})
and (\ref{E412}), we find, using (\ref{430J}), that
%
%
\begin{equation}\label{3999}
\al\geq c (\log N)^{C_1\xi+1}\biggl(\frac{1}{q}+\frac{1}{N
\eta}\biggr) \gg\Lambda
\end{equation}
with $(\xi,\nu)$-high probability,
where we used (\ref{E38}) to obtain the first term $q^{-1}$ on the
right-hand side and we used $N \eta\sqrt\kappa
=(\log N)^{C_1\xi+1}$ [see the definition (\ref{E40})
of $\eta$] for the second term. Now we can assume
%
%
\begin{equation}
q \geq(\log N)^{C_3 \xi}
\end{equation}
for some large $C_3>0$ [otherwise (\ref{boundonHtilde}) holds for
some constant $C$ by
Lem\-ma~\ref{lemmaboundon_h_tildeweak}]. We have $E+ \ii\eta\in D_L$
[recall that $E$ satisfies (\ref{E38})]. Using (\ref{3999}), we can
neglect the terms $\Lambda^2$ and $[v]^2$ in (\ref{reszlemmasparse4})
to get, with $(\xi,\nu)$-high probability,
%
%
\begin{equation} \label{333}
\Lambda\leq C(\log N)^{C_1\xi} \biggl(\frac1{ \alpha q^2}+\frac
1{ \alpha(N\eta)^2}
+ \frac{1 }{ \alpha N\sqrt{\kappa}}\biggr).
\end{equation}
Since $\alpha\geq K \sqrt\kappa$, the last term is bounded by
\[
\frac{1 }{ \alpha N\sqrt{\kappa}} \leq(\log N)^{-C_1\xi- 1}
\frac{\eta}{ \sqrt\kappa} \leq(\log
N)^{-C_1\xi- 1} \frac{1}{N\eta},
\]
where we have used (\ref{E41}). The first term on the right-hand side
of (\ref{333}) can be estimated similarly using
(\ref{E411}) and (\ref{911}).
Finally, the middle term on the right-hand side of (\ref{333}) can be
estimated by using (\ref{3999}).
Putting everything together, we obtain, for any $E$ satisfying (\ref
{E38}), that
%
%
\begin{equation}
\Lambda(z) \ll\frac{1}{N\eta} \qquad\mbox{for } z = E
+ \ii\eta\in D_L
\end{equation}
with $(\xi,\nu)$-high probability. Furthermore, with (\ref{426QS})
and (\ref{E41}), we
obtain that for any $E$ in (\ref{E38})
%
%
\begin{equation} \label{312YSY}
\im m(z) \leq\im m_{\mathrm{sc}}(z) + \Lambda(z) \ll\frac{1}{N\eta
} \qquad\mbox{for } z = E+\ii\eta
\in D_L
\end{equation}
with $(\xi,\nu)$-high probability. Since
%
%
\begin{equation}\label{123}
\im m(z) = \frac1 N \sum_{\al}\frac{ \eta}{(\lambda_\al
-E)^2+\eta^2},
\end{equation}
we have
\[
\im m(z) \geq\frac c { N \eta}
\]
if there is an eigenvalue in $[E-\eta, E+\eta]$.
Then (\ref{123}) and
(\ref{312YSY}) imply that, for any $E$ satisfying (\ref{E38}), there
is no eigenvalue in
$[E-\eta, E+\eta]$
with $(\xi,\nu)$-high probability. The regime $E\geq3$ is covered by
Lemma~\ref{lemmaboundon_h_tildeweak}.
This completes the proof.
\end{pf}

\section{Abstract decoupling lemma and applications} \label{sectionZ-lemma}

In this section we prove an abstract decoupling lemma which is
independent of the random matrix model. We shall apply this
abstract result to random matrices in
Sections~\ref{sectionGdecomposition},~\ref{sectproofof41} and~\ref{sectaverageforedge}.

\subsection{Abstract decoupling lemma}\label{GenZlemma}
Throughout this section we use the letters $A$ and $B$ to denote
abstract random variables. Note that $A$ in this
context has nothing to do with the matrix $A$ from Definition \ref
{definitionofA}.
We work on the probability space generated by the $N\times N$ random
matrices $H$.
Let $(A^{[\bb U]})$ be a family of random variables indexed by subsets
$\bb U \subset\{1,\ldots,N\}$, and denote $A \deq
A^{[\varnothing]}$. For $\bb U \subset\bb S \subset\{1,\ldots, N\}$
we define the random variable
%
%
\begin{equation} \label{56ASU}\quad
A^{\bb S,\bb U} \deq\sum_{\bb T \subset\bb U} (-1)^{|\bb
T |} A^{[(\bb S \setminus\bb U) \cup\bb T]} = (-1)^{|
\bb S \setminus\bb U |} \sum_{\bb V \st\bb S \setminus\bb U
\subset\bb V \subset
\bb S} (-1)^{|\bb V |} A^{[\bb V]}.
\end{equation}

%
\begin{lemma}[(Resolution of dependence)]
For any $\bb S$ we have
%
%
\begin{equation} \label{55ASU}
A = \sum_{\bb U \subset\bb S} A^{\bb S,\bb U}.
\end{equation}
\end{lemma}
\begin{pf}
The proof is a standard inclusion-exclusion argument.
\end{pf}
%
%
\begin{definition} \label{definitionofminor2}
Let $A \deq A(H)$ be a random variable. Then we define the new random
variable $A^{(\bb T)}$ through
%
%
\begin{equation} \label{definitionofupT}
A^{(\bb T)}(H) \deq A(\pi_{\bb T}(H)),
\end{equation}
where $\pi_{\bb T}$ was defined in (\ref{definitionofpi}).
\end{definition}
%
%
\begin{remark}\label{rem73}
Note that the operation $(\cdot)^{(\bb T)}$ is compatible with
algebraic operations in the sense that
%
%
\begin{equation}
(A+B)^{(\bb T)} = A^{(\bb T)}+B^{(\bb T)},\qquad (A B)^{(\bb T)}
= A^{(\bb T)} B^{(\bb T)}.
\end{equation}
Since $\pi_{\bb U} \circ\pi_{\bb V} = \pi_{\bb U \cup\bb V}$, we
also have $(A^{(\bb U)})^{(\bb V)} = A^{(\bb U \cup
\bb V)}$.
\end{remark}
%
%
\begin{remark}
The matrices $H^{(\bb T)}$ and $G^{(\bb T)}$ defined through (\ref
{definitionofupT}) are \mbox{$N \times N$} matrices. We
adopt this convention only in this section. This is in contrast to
Definition~\ref{definitionofminor}, where the same
notation was used for the $(N - |\bb T |) \times(N - |
\bb T |)$ minors of the same matrices. This slight abuse of
notation will not cause ambiguity, however, because we shall only
consider matrix elements $H^{(\bb T)}_{ij}$ and
$G^{(\bb T)}_{ij}$ for $i$, \mbox{$j \notin\bb T$}; for these matrix elements
the two definitions coincide.
\end{remark}
%
%
\begin{definition}
We say that a random variable $A$ is \textit{independent} of the set
$\bb U \subset\{1,\ldots, N\}$ if $A = A^{(\bb U)}$ [or,
equivalently, if $A$ is independent of the family $(h_{ij} \st i \in
\bb U \mbox{ or } j \in\bb U)$].
\end{definition}

We shortly explain the idea behind these definitions. In many
applications we choose $A^{[\bb U]} \deq A^{(\bb U)}$, so
that $A^{[\bb U]}$ is independent of $\bb U$. In this case, the
decomposition (\ref{55ASU}) can be interpreted as
follows. We first fix a reference set $\bb S$. From (\ref{56ASU}) it
is clear that $A^{\bb S,\bb U}$ is independent of
$\bb S \setminus\bb U$, that is, it depends only on the set $\bb U$
(among the variables in $\bb S$). Therefore,
(\ref{55ASU}) can be viewed as a resolution of dependence of $A$ on
subsets of $\bb S$. We shall\vspace*{1pt} see that when we apply
this decomposition to resolvent matrix elements, that is, set $A =
G_{ij}$, then $G^{\bb S,\bb U}_{ij}$ will be comparable
in size with a product of at least $|\bb U |+1$ off-diagonal
resolvent matrix elements, which
are small with high probability. Hence, in this case, the decomposition
(\ref{55ASU}) is effectively a graded
resolution with a trade-off
between dependence and size. A~larger $\bb U$ means that $G^{\bb S,\bb
U}_{ij}$ is smaller, but it depends on more
variables. For smaller $\bb U$'s we will exploit that $G^{\bb S,\bb U}$
is independent of more variables.

The purpose of this graded decoupling is to obtain large deviation
estimates on the average $[\cal Z] \deq\frac{1}{N} \sum_i \cal Z_i$
of $N$ weakly dependent centered
random variables~$\cal Z_i$.
The precise result is given in
Theorem~\ref{abstractZlemma} below. Before stating it, we outline the
main ideas.

In our applications, the covariances between different variables $Z_i$
are too large to be controlled in terms of their variances and, hence,
standard methods
for sums of weakly dependent random variables relying on such ideas do
not apply.
Instead, the weak dependence will be expressed in terms of the
smallness of
$\cal Z_i^{\bb S, \bb U}$ for large~$\bb U$; the size of $\cal Z_i^{\bb
S, \bb U}$ reflects how strongly $\cal Z_i$ depends on the set $\bb U$.
The basic strategy is a high-moment estimate
\[
\E\f1 (\Xi) |[\cal Z] |^p = \frac{1}{N^p}
\sum_{i_1,\ldots, i_p} \E({\f1 (\Xi) \cal Z_{i_1} \cdots
\ol{\cal Z}_{i_p}})
\]
on some high-probability event $\Xi$,
whereby each term $\cal Z_{i_j}$ is expanded according to the graded
expansion of (\ref{55ASU}). The right-hand side is controlled using
the two following facts: (i) $\cal Z_i^{\bb S, \bb U}$ is small for
large $\bb U$ (weak dependence of $\cal Z_i$ on $\bb U$). (ii)~The
expectation vanishes if all factors are independent. Note that this
graded expansion differs from the conventional\vadjust{\goodbreak} martingale-type
arguments used to establish central limit theorems for correlated
random variables.

The basic idea of a graded expansion to control large deviations of
sums of weakly dependent
random variables was introduced in Lemma 5.2 of~\cite{EYY2} in the
context of
Wigner matrices. This result
considers the special case $\cal Z_i = Z_i$ [as defined in (\ref
{definitionofZi})]
and uses expansions in full rows and columns to detect dependencies.
For the applications
in~\cite{EYY2}, only large but $N$-independent powers $p$ were
considered. Hence, in~\cite{EYY2}
it was not necessary to keep track of the $p$-dependence or the
probability of $\Xi$.

A new proof was given in Lemma 4.1 of~\cite{EYYrigidity},
where the $p$-dependence and the probability of $\Xi$ were tracked precisely.
This proof relied on an expansion in terms of individual matrix
elements and not full rows and columns.
Thus, the expansion was more economical, but its combinatorial
structure was considerably more involved.

In this paper we present an abstract generalization of the row and
column expansion method of~\cite{EYY2}. It is formulated for an
arbitrary family of random variables $\cal Z_1,\ldots, \cal Z_N$. As
input, it needs bounds on the terms of the graded expansion of~$\cal
Z_i$. The abstract formulation thus streamlines the argument by
dissociating two unrelated steps of the proof: (i) the moment estimate
using the graded expansion (a~probabilistic estimate given in Theorem
\ref{abstractZlemma}) and (ii) controlling the size of the graded terms
for a concrete application (in the case of resolvent matrix elements, a
deterministic, almost entirely algebraic, argument given in Section
\ref{sectionGdecomposition}).

For our purposes, this increased generality is needed for two reasons.
First, it allows
for an efficient control of the strong fluctuations associated with
sparse matrix entries.
Second, we use it to control the average of not only $Z_i$ (Lemma \ref
{sp74}) but also
quantities like (\ref{resexperror}) with a different algebraic structure.
In the special case $\cal Z_i = Z_i$ and $q = N^{1/2}$ (Wigner matrix),
our result reduces to
that of Lem\-ma~4.1 in~\cite{EYYrigidity}.
%
%
\begin{theorem}[(Abstract decoupling lemma)]
\label{abstractZlemma} Let $\cal Z_1,\ldots, \cal Z_N$ be random
variables and recall the notation
\[
[\cal Z] = \frac{1}{N} \sum_{i = 1}^N \cal Z_i.
\]
Let $\Xi$ be an event and $p$ an even integer. Suppose that there
exists a family of random variables $(\cal Z_i^{[\bb U]})_{i,\bb U}$
indexed by $i \in\{1,\ldots,
N\}$ and $\bb U \subset\{1,\ldots, N\}$ satisfying $i \notin\bb U$,
such that $\cal Z_i^{[\varnothing]} = \cal Z_i$ and the
following assumptions hold with some constant~$C$:
\begin{longlist}
\item Recall the partial expectation $\E_i$ from
Definition~\ref{defEi}. For $i \notin\bb U$ we have that $\cal
Z_i^{[\bb U]}$ is independent of $\bb U$ and
%
%
\begin{equation} \label{Zmean0}
\E_i \cal Z_i^{[\bb U]} = 0.\vadjust{\goodbreak}
\end{equation}

\item($L^r$-norm in $\Xi$). For any $\bb U$, $\bb S$ with $\bb
U\subset\bb S$ and $i \notin\bb S$ we consider
$\cal Z_i^{\bb S, \bb U}$ defined by (\ref{56ASU}) from the family
$\cal Z_i^{[\bb U]}$.
Then for any numbers $r\leq p$ with $|\bb S |\leq p$
we have
%
%
\begin{equation}\label{511}
\E({\mathbf1}(\Xi) |\cal Z_i^{\bb S, \bb U}|^{r}) \leq
( Y (CXu )^{ u})^{r}
\qquad\mbox{with } u \deq|\bb U|+1,
\end{equation}
where
$X$ and $Y$ are deterministic and $X$ satisfies
%
%
\begin{equation}\label{512}
X \leq\frac1{p^5\log N}.
\end{equation}
\item(rough bound on the $L^2$-norm in $[\Xi]_i$).
Define
%
%
\begin{equation}\label{defXii}
[\Xi]_i \deq(\pi_i^{-1} \circ\pi_{i}) (\Xi).
\end{equation}
For any $\bb U$, $\bb S$ satisfying $\bb U\subset\bb S$, $i \notin\bb
S$, and $|\bb S |\leq p$
we have
%
%
\begin{equation}\label{5112}
\E({\mathbf1}([\Xi]_i)| \cal Z_i^{\bb S, \bb U}|^{2})
\leq N^{Cp}.
\end{equation}

\item(rough bound on $\cal Z_i$). For any $\bb U$ we have
%
%
\begin{equation}\label{ZUNC}
{\mathbf1}(\Xi)\bigl| \cal Z_i^{[\bb U]}\bigr| \leq Y N^C.
\end{equation}

\item($\Xi$ has high probability). We require that
%
%
\begin{equation}\label{513}
\P[\Xi^c] \leq\me^{-c(\log N)^{3/2} p}.
\end{equation}
\end{longlist}

Then, under the assumptions \textup{(i)--(v)}, we have
%
%
\begin{equation}\label{resZZ}
\P\bigl( {\mathbf1}(\Xi) |[\cal Z] |\geq p^{ 12} Y
(X^2+N^{-1})
\bigr) \leq\frac{C^{ p}}{p^{ p}}
\end{equation}
for some $C>0$ and sufficiently large $N$. The constant in (\ref
{resZZ}) depends on the constants in (\ref{511}),
(\ref{ZUNC}) and (\ref{513}).
\end{theorem}

The key assumptions in Theorem~\ref{abstractZlemma} are (i) and (ii);
the key (small) parameter is $X$. Assumption (i) simply ensures that
all terms of the graded expansion of $\cal Z_i$ have zero expectation.
Assumption (ii) defines the decay of $\cal Z_i^{\bb S, \bb U}$ in
the size of $\bb U$; roughly, it states that
\[
|\cal Z_i^{\bb S, \bb U} |\lesssim X^{|
\bb U |+1}
\]
in the sense of high moments.
This is in accordance with the principle outlined above
that terms of the graded expansion which depend on many variables have
a small size,
while those which are independent of many variables may be larger.
The parameter $Y$ is trivial in our applications, where we shall take
it to be a logarithmic factor.
In Lemma 4.1 of~\cite{EYYrigidity}, the role of $X$ was played by the
parameter $\Psi$ defined in (\ref{defPsi}).
\begin{pf*}{Proof of Theorem~\ref{abstractZlemma}}
We find
%
%
\begin{equation}
\E({\mathbf1}(\Xi) |\cal Z|^p) = N^{-p}\sum_{\al
_1,\al_2, \ldots, \al_p=1}^N \E\Biggl({\mathbf1}(\Xi)
\prod_{j=1}^p \cal Z^{\#}_{\al_j}\Biggr),
\end{equation}
where $\#$ stands for either nothing or complex conjugation. Let
$\ff{\alpha}= (\alpha_1,\ldots,\break \alpha_p)$ and define
$\bb S \equiv\bb S(\ff{\alpha}) \deq\{\al_1,\al_2, \ldots, \al
_p\}
$. Then we have
%
%
\begin{equation}\label{FCC}
\E({\mathbf1}(\Xi) |\cal Z|^p) \leq N^{-p}p^p \sum
_{s=1}^p N^{s} \max_{\ff{\alpha}\st|\bb S(\ff{\alpha})
|=s} \Biggl|\E\Biggl({\mathbf1}(\Xi) \prod_{j=1}^p \cal Z^{\#
}_{\al_j}\Biggr)\Biggr|.
\end{equation}
Abbreviating $\bb S_j \deq\bb S \setminus\{\alpha_j\}$, we find from
(\ref{55ASU}) that
%
%
\begin{equation}
\cal Z_{\al_j} = \sum_{\bb U'_j \subset\bb S_j} \cal Z^{\bb S_j,
\bb U'_j}_{\al_j}.
\end{equation}
Thus, (\ref{FCC}) implies
%
%
\begin{eqnarray}
&&\E({\mathbf1}(\Xi) |\cal Z|^p) \nonumber\hspace*{-35pt}\\[-8pt]\\[-8pt]
&&\qquad\leq N^{-p} p^p \sum
_{s=1}^p N^{s} \max_{\ff{\alpha}\st|\bb S(\ff{\alpha})
|=s} \Biggl|\E\Biggl({\mathbf1}(\Xi) \sum_{\bb U_1' \subset
\bb
S_1} \cdots\sum_{\bb U_p' \subset\bb S_p} \prod_{j=1}^p [\cal
Z^{\#}_{\al_j}]^{\bb S_j, \bb U_j'}\Biggr)\Biggr|\nonumber\hspace*{-35pt}
\end{eqnarray}
(abbreviating $\overline{A^{\bb S,\bb U}} = \ol{A}^{\bb S,\bb U}$).
Writing $\bb U_j \deq\bb U_j' \cup\{\alpha_j\}$, we have
%
%
\begin{eqnarray}\label{518}\quad
&&\E( {\mathbf1}(\Xi) |\cal Z|^p) \nonumber\\
&&\qquad\leq\biggl(\frac
{p}{N}\biggr)^p \sum_{s=1}^p \sum_{n=1}^{sp} N^{s} s^n n^p
\max\Biggl\{\Biggl|\E\Biggl({{\f1}(\Xi) \prod_{j=1}^p
[\cal Z^{\#}_{\al_j}]^{\bb S_j, \bb U_j'}}\Biggr) \Biggr|\st
|\bb S(\ff{\alpha}) |= s,\\
&&\qquad\quad\hspace*{198.2pt} \bb U_j' \subset\bb S_j
,
\sum
_{j = 1}^p |\bb U_j |= n \Biggr\}.\nonumber
\end{eqnarray}

Now we claim that
%
%
\begin{eqnarray}\label{jbd}\quad
&&N^{s} s^n n^p
\max\Biggl\{{\Biggl|\E\Biggl({{\f1}(\Xi) \prod_{j=1}^p
[\cal Z^{\#}_{\al_j}]^{\bb S_j, \bb U_j'}}\Biggr) \Biggr|\st
|\bb S(\ff{\alpha}) |= s, \bb U_j' \subset\bb S_j
,
\sum
_{j = 1}^p |\bb U_j |= n }\Biggr\}\hspace*{-20pt}\nonumber\\[-8pt]\\[-8pt]
&&\qquad\leq\bigl(CNp^{ 10}Y(X^2+N^{-1})\bigr)^p\nonumber
\end{eqnarray}
for some $C>0$. Then inserting (\ref{jbd}) into (\ref{518}), we find
%
%
\begin{equation}
\E({\f1}(\Xi) |\cal Z|^p) \leq\bigl(Cp^{
11}Y(X^2+N^{-1})\bigr)^p,
\end{equation}
which implies (\ref{resZZ}) by Markov's inequality.

It only remains to prove (\ref{jbd}). We consider two cases: $n\geq
2s$ and $n\leq2s-1$.\vadjust{\goodbreak}

We begin by proving (\ref{jbd}) for the case $n\geq2s$. Using
H\"older's inequality, we find
%
%
\begin{equation}\label{520}
\Biggl|
\E\Biggl({\f1}(\Xi) \prod_{j=1}^p [\cal Z^{\#}_{\al_j}]^{\bb S_j,
\bb U_j'}\Biggr)
\Biggr|
\leq\Biggl( \prod_{j=1}^p \E( {\f1}(\Xi) |[\cal
Z^{\#}_{\al_j}]^{\bb S_j,
\bb U_j'}|^p)\Biggr)^{1/p}.
\end{equation}
Applying (\ref{511}) to the right-hand side, we obtain that
%
%
\begin{equation}\label{521}
\Biggl( \prod_{j=1}^p \E( {\f1}(\Xi) |[\cal Z^{\#
}_{\al_j}]^{\bb S_j, \bb U_j'}|^p)\Biggr)^{1/p}
\leq Y^p (Cn X)^{ n}
\end{equation}
since $\sum_j(|\bb U'_j|+1)=\sum_j |\bb U'_j|=n$.
Combining (\ref{520}), (\ref{521}) and the factor $n\geq
p$, we have bounded the left-hand side of (\ref{jbd}) as follows:
\begin{eqnarray*}
&&N^{s} s^n n^p
\max\Biggl\{\Biggl|\E\Biggl({{\f1}(\Xi) \prod_{j=1}^p
[\cal Z^{\#}_{\al_j}]^{\bb S_j, \bb U_j'}}\Biggr) \Biggr|\st
|\bb S(\ff{\alpha}) |= s, \bb U_j' \subset\bb S_j
,
\sum
_{j = 1}^p |\bb U'_j |= n \Biggr\}\\
&&\qquad \leq N^s Y ^p (Cn^2Xs)^{ n}
\\
&&\qquad \leq N^s Y ^p (Cn^2Xs)^{2s}
\\
&&\qquad \leq \bigl(CN p^{ 10}Y (X^2+N^{-1}) \bigr)^p,
\end{eqnarray*}
where in the second inequality we used
\[
Cn^2Xs\leq CXs^3p^2\leq CXp^5\ll1
\]
[see (\ref{512})] and $n\geq2s$, and in the third inequality $s\leq
p$ and $n\leq sp$. This completes the proof of (\ref{jbd}) for
the case $n\geq2s$.

Now we prove (\ref{jbd}) for the case $n\leq2s-1$. Fix sets $\bb U_j'$
with $\sum_j |\bb U_j |= n$, where we recall that $\bb U_j \deq\bb U_j'
\cup\{\alpha_j\}$ and $|\bb U_j |= |\bb U_j' |+ 1$. By
definition\vspace*{1pt} of $\bb U_j$, we have $\alpha _j \in\bb U_j$
for all $j$. Since $n\leq2s-1$, we therefore find that there exists a
$k$ such that $\al_k \in\bb U_k$ and $ \al_k \notin\bb U_{j}$ for $j
\neq k$. By the definitions (\ref {56ASU}) and (\ref{Zmean0}), $[\cal
Z^{\#}_{\al_j}]^{\bb S_j, \bb U_j'}$ is independent of $\bb S_j
\setminus\bb U_j'$, that is, of $\bb S \setminus\bb U_j$. We conclude
that
%
%
\begin{equation}\label{527}
\prod_{j\neq k}^p [\cal Z^{\#}_{\al_j}]^{\bb S_j, \bb U_j'}
\end{equation}
is independent of $\{\al_k\}$. Therefore,
%
%
\begin{eqnarray}
&&
\E\Biggl({\f1}([ \Xi]_{\al_k}) \prod_{j=1}^p [\cal Z^{\#}_{\al
_j}]^{\bb S_j, \bb U_j'}\Biggr)\nonumber\\[-8pt]\\[-8pt]
&&\qquad= \E\Biggl(
\Biggl[ \prod_{j\neq k}^p
[\cal Z^{\#}_{\al_j}]^{\bb S, \bb U_j} \Biggr] {\f1}([ \Xi]_{\al
_k}) \E_{\al_k} [\cal Z^{\#}_{\al_k}]^{\bb S_k, \bb U_k'}
\Biggr)
= 0.\nonumber
\end{eqnarray}
Thus,
%
%
\begin{equation}
\E\Biggl({\f1}(\Xi) \prod_{j=1}^p [\cal Z^{\#}_{\al_j}]^{\bb S_j,
\bb U_j'}\Biggr) = -\E\Biggl( {\f1}([ \Xi]_{\al_k}
\setminus\Xi) \prod_{j=1}^p [\cal Z^{\#}_{\al_j}]^{\bb S_j, \bb
U_j'}\Biggr),
\end{equation}
which yields
%
%
\begin{eqnarray}\label{530}
&&\Biggl|\E\Biggl({\f1}(\Xi) \prod_{j=1}^p [\cal Z^{\#}_{\al
_j}]^{\bb S_j, \bb U_j'}\Biggr)\Biggr|\nonumber\\
&&\qquad \leq \E\Biggl( {\f1}([ \Xi]_{\al_k} \setminus\Xi) \prod
_{j=1}^p |[\cal Z^{\#}_{\al_j}]^{\bb S_j,
\bb U_j'}|\Biggr)
\\
&&\qquad \leq \Biggl\|
{\f1}([ \Xi]_{\al_k} \setminus\Xi) \prod_{j\neq k}^p
|[\cal Z^{\#}_{\al_j}]^{\bb S_j, \bb U_j'}| \Biggr\|_\infty
\E\bigl( {\f1}([ \Xi]_{\al_k} \setminus\Xi) |[\cal Z^{\#
}_{\al_k}]^{\bb S_k, \bb U_k'}|\bigr)
.\nonumber
\end{eqnarray}
Since (\ref{527}) is independent of $\al_k$, we get
%
%
\begin{eqnarray}\label{531}
\Biggl\|{\f1}([ \Xi]_{\al_k} \setminus\Xi) \prod_{j\neq k}^p
|[\cal Z^{\#}_{\al_j}]^{\bb S_j, \bb U_j'}| \Biggr\|_\infty
&\leq&
\Biggl\|{\f1}([\Xi]_{\alpha_k}) \prod_{j\neq k}^p
|[\cal Z^{\#}_{\al_j}]^{\bb S_j, \bb U_j'}| \Biggr\|
_{\infty}\nonumber\\[-8pt]\\[-8pt]
&=&\Biggl\|{\f1}(\Xi) \prod_{j\neq k}^p |[\cal Z^{\#}_{\al
_j}]^{\bb S_j, \bb U_j'}| \Biggr\|_{\infty}.\nonumber
\end{eqnarray}
Using the definition of $\cal Z_i^{\bb S, \bb U}$ in (\ref{56ASU}) and
(\ref{ZUNC}), we have
%
%
\begin{equation}
|{\f1}(\Xi)[\cal Z^{\#}_{\al_j}]^{\bb S_j, \bb U_j'}
|\leq
YN^C 2^{|\bb U_j|}
\end{equation}
and
%
%
\begin{equation}\label{533}
\Biggl|{\f1}(\Xi)\prod_{j\neq k}^p[\cal Z^{\#}_{\al_j}]^{\bb
S_j, \bb U_j'} \Biggr|\leq
(YN^{C})^{p-1} 2^{n} \leq(YN^{C})^{p-1} 2^{2p},
\end{equation}
where we used $s\leq p$ and $n\leq2s$ in the last inequality.
Combining (\ref{530}), (\ref{531}) and (\ref{533}), we
get
%
%
\begin{equation}\quad
\Biggl|\E\Biggl({\f1}(\Xi) \prod_{j=1}^p [\cal Z^{\#}_{\al
_j}]^{\bb S_j, \bb U_j'}\Biggr)\Biggr|
\leq(YN^{C})^{p-1} 2^{2p}
\bigl(\E{\f1}([ \Xi]_{\al_k}\setminus\Xi) |[\cal Z^{\#
}_{\al_k}]^{\bb S_k, \bb U_k'}|\bigr).\hspace*{-28pt}
\end{equation}
Applying Schwarz's inequality on the right-hand side, we find
%
%
\begin{eqnarray}\quad
&&\Biggl|\E\Biggl({\f1}(\Xi) \prod_{j=1}^p [\cal Z^{\#}_{\al
_j}]^{\bb S_j, \bb U_j'}\Biggr)\Biggr|\nonumber\\[-8pt]\\[-8pt]
&&\qquad\leq(YN^{C})^{p-1} 2^{2p}
\bigl(\P([ \Xi]_{\al_k}\setminus\Xi)\bigr)^{1/2} (\E{\f
1}([ \Xi]_{\al_k})|[\cal Z^{\#}_{\al_k}]^{\bb S_k, \bb
U_k'} |^2)^{1/2}.\nonumber
\end{eqnarray}
Using (\ref{512}), (\ref{513}), (\ref{5112}) and that $n\leq
2s-1\leq2p$, we get for any $\wt C>0$
%
%
\begin{equation}
\Biggl|\E\Biggl({\f1}(\Xi) \prod_{j=1}^p [\cal Z^{\#}_{\al
_j}]^{\bb S_j, \bb U_j'}\Biggr)\Biggr|
\leq(YN^{C})^{p } 2^{2p}
\P(\Xi^c)^{1/2} \leq Y^pN^{- \wt Cp}.
\end{equation}
Since $s\leq p$, the proof of (\ref{jbd}) in the case $n\leq2s-1$ is complete.
\end{pf*}

\subsection{Decomposition of $G_{ij}$} \label{sectionGdecomposition}

In order to apply Theorem~\ref{abstractZlemma} to estimate $[Z]$, we
need to derive bounds, and hence formulas, for the
decomposition $G_{ij}^{\bb S, \bb U}$ of resolvent matrix elements $G_{ij}$.
As usual, $G$ refers to the resolvent of $H$ at a fixed spectral
parameter $z$ (which is suppressed
in the notation), that is, $G= G(H)$ is viewed as a function of $H$.
The main result of this section is the bound
(\ref{616}) below.

Note that the results in this subsection are entirely deterministic.
%
%
\begin{lemma}\label{YJM}
Let $z=E+i\eta\in D$, where $D\subset\C$ is some compact domain. Let
$\bb U\subset\{1,2,\ldots, N\}$ and
%
%
\begin{equation}\label{631i}
|\bb U| \leq\frac{1}{(\Lambda_o+\Lambda_d)\log N}.
\end{equation}
Then for any $i,j\notin\bb U$, we have
%
%
\begin{equation} \label{Gijdeltaij}
\bigl|G^{(\bb U)}_{ij}-m_{\mathrm{sc}} \delta_{ij}\bigr| \leq C\bigl( {\f1
}(i=j)\Lambda_d+\Lambda_o\bigr).
\end{equation}
In particular, if $\Lambda_d + \Lambda_o \leq(\log N)^{-1}$, then
%
%
\begin{equation} \label{Giilowerbound}
\inf_{i\notin\bb U}\bigl|G^{(\bb U)}_{ii}\bigr| \geq c.
\end{equation}
Here the constants $c$ and $C$ depend only on $D$.
\end{lemma}
\begin{pf}
Define
%
%
\begin{equation}
B_m \deq\max\bigl\{ \bigl| G_{ij}^{(\bb V)}-\delta_{ij}G_{ii}
\bigr| \st i,j\notin\bb V, |\bb V|=m\bigr\}.
\end{equation}
In the case $m=0$, (\ref{Gijdeltaij}) follows from the definitions of
$\Lambda_o$ and $\Lambda_d$. The estimate
(\ref{Giilowerbound}) follows from (\ref{Gijdeltaij}), noting that
$|m_{\mathrm{sc}}(z)|\geq c$ on a compact domain $z\in D$
with $c$ depending on $D$. Next, from (\ref{GijGijk}) we get
%
%
\begin{equation}
G_{ij}^{(k \bb T)} = G^{(\bb T)}_{ij}-\frac{G^{(\bb T)}_{ik}G^{(\bb
T)}_{kj}}{G^{(\bb T)}_{kk}}\qquad \mbox{where } i,j
\notin\{k\}\cup\bb T \mbox{ and } k\notin\bb T.
\end{equation}
Assuming (\ref{Giilowerbound}) for $|\bb U|=m$, we therefore obtain
%
%
\begin{equation}
B_{m+1} \leq B_{m}+C_0B_{m}^2
\end{equation}
for some constant $C_0>0$ independent of $m$. This implies that
%
%
\begin{equation}
B_{m+1} \leq C_0\sum_{k=0}^mB_{k}^2+B_0.
\end{equation}
By induction on $m$ one obtains $B_m\leq2B_0$ as long as $C_0mB_0 \leq1/2$.
\end{pf}

In order to state the next result, we introduce a class of rational
functions in resolvent matrix elements. Fix two sets
$\bb U \subset\bb S$ satisfying $\bb U \neq\varnothing$. For fixed
$n \in\N$ let the following be given:
\begin{longlist}
\item
a sequence of integers $(i_r)_{r = 1}^{n+1}$ satisfying $i_k \neq i_{k
+ 1}$ for $1 \leq k \leq n$;
\item
a collection of sets $(\bb U_\alpha)_{\alpha= 1}^n$ satisfying
$i_\alpha, i_{\alpha+ 1} \notin\bb U_\alpha$ as well as $\bb S
\setminus\bb U \subset\bb U_\alpha\subset\bb S$ for $1 \leq\alpha
\leq n$;
\item
a collection of sets $(\bb T_\beta)_{\beta= 2}^n$ satisfying $i_\beta
\notin\bb T_\beta$ as well as $\bb S \setminus\bb U \subset
\bb T_\beta\subset\bb S$ for $2 \leq\beta\leq n$.
\end{longlist}
Then we define the random variable, parametrized by $(i_r)_{r =
1}^{n+1}, (\bb U_\al)_{\al= 1}^n, (\bb T_\beta)_{\beta=
2}^{n}$,
%
%
\begin{equation}\label{6488new}
F((i_r)_{r = 1}^{n+1}, (\bb U_\al)_{\al= 1}^n, (\bb T_\beta
)_{\beta= 2}^{n} ) \deq\frac PQ,
\end{equation}
where
\[
P = \prod_{\al=1}^n G^{(\bb U_{\al})}_{i_\al,i_{\al+1}},\qquad
Q = \prod_{\beta=2}^{n } G^{(\bb T_{\beta})}_{i_\beta,i_\beta}.
\]
Note that $F$ depends on the randomness via the resolvent matrix elements.
All matrix elements are off-diagonal in the numerator and diagonal in
the denominator; $n$ counts the number of
off-diagonal elements in the numerator.
The sequence of indices of these matrix elements is consecutive if
$P/Q$ is written as an alternating product of
off-diagonal elements from the numerator $P$ and reciprocals of
diagonal elements from the denominator $Q$, that is, in the
form
%
%
\begin{equation}\label{product}
\frac{P}{Q} = G_{i_1 i_2}^{(\bb
U_1)} \bigl[ G_{i_2 i_2}^{(\bb T_2)}\bigr]^{-1}G_{i_2 i_3}^{(\bb U_2)}
\bigl[G_{i_3 i_3}^{(\bb T_3)}\bigr]^{-1} \cdots G_{i_n i_{n+1}}^{(\bb U_n)}.
\end{equation}

%
\begin{definition}\label{58G}
For $\bb U \subset\{1,\ldots, N\}$ and $i,j \notin\bb U$ define
$G_{ij}^{[\bb U]} \deq G_{ij}^{(\bb U)}$. For $i,j \notin\bb S$ and
$\bb U
\subset\bb S$ define $G_{ij}^{\bb S,\bb U}$ through (\ref{56ASU}).
\end{definition}
%
%
\begin{lemma}\label{STG}
Let $\bb S \subset\{1,\ldots, N\}$ and $i,j \notin\bb S$.
Then
%
%
\begin{equation} \label{656}
G^{\bb S, \varnothing}_{ij} = G^{(\bb S)}_{ij}.
\end{equation}
If $\varnothing\neq\bb U \subset\bb S$, then $G_{ij}^{\bb S,\bb U}$
can be written as
%
%
\begin{equation} \label{657}
G_{ij}^{\bb S,\bb U} = \sum_{n = |\bb U |+1}^{2 |
\bb U |}
F_n,\qquad F_n = \sum_{k = 1}^{K_n} F_{n,k},\vadjust{\goodbreak}
\end{equation}
where $\sum_{n=|\bb U|+1}^{2|\bb U|}K_n\leq4^{|\bb U| }|\bb U|!$ and
each $F_{n,k}$ is of the form (\ref{6488new})
(with a possible minus sign), with $i_2,\ldots,i_n \in\bb U$, $i_1 =
i$, $i_{n + 1} = j$, and with some
appropriately chosen sets $(\bb U_\al)_{\al= 1}^n$ and $(\bb T_\beta
)_{\beta= 2}^{n}$ which may be different for each
$F_{n,k}$.
\end{lemma}

Note: the index $n$ in $F_n$ and $F_{n,k}$ refers to the number of
off-diagonal elements appearing in the rational
functions (\ref{6488new}), while $k$ is just a counting index.
\begin{pf*}{Proof of Lemma~\ref{STG}}
First, (\ref{656}) follows from (\ref{56ASU}).

It remains to prove (\ref{657}) in the case $\bb U\neq\varnothing$.
Using Definition~\ref{58G} and Remark~\ref{rem73},
one readily sees that, for a set $\bb T$ satisfying $\bb T \cap\bb U =
\varnothing$ and $i,j \notin\bb S \cup\bb T$, we
have
%
%
\begin{equation}\label{ATSU}
(G_{ij}^{\bb S, \bb U})^{(\bb T)} = G_{ij}^{\bb S\cup\bb T, \bb
U}.
\end{equation}
Thus, if ${a}\in\bb U\subset\bb S$, we get from (\ref{56ASU}) and
(\ref{ATSU}) that
%
%
\begin{eqnarray}\label{619}
G_{ij}^{\bb S, \bb U} &=& G_{ij}^{\bb S\setminus\{{a}\}, \bb
U\setminus\{{a}\}} - G_{ij}^{\bb S, \bb
U\setminus\{{a}\}}\nonumber\\[-8pt]\\[-8pt]
&=& G_{ij}^{\bb S\setminus\{{a}\}, \bb U\setminus\{{a}\}} -
\bigl(G_{ij}^{\bb S\setminus\{{a}\}, \bb
U\setminus\{{a}\}}\bigr)^{({a})} \qquad\mbox{for } i,j \notin
\bb S.\nonumber
\end{eqnarray}
In the special case $|\bb U|=1$, writing $\bb U=\{{a}\}$, we have
%
%
\begin{equation}\label{620}
G_{ij}^{\bb S, \bb U} = G_{ij}^{\bb S, \{{a}\}} = G_{ij}^{(\bb
S\setminus\{{a}\})}-G_{ij}^{(\bb S)}.
\end{equation}
Using (\ref{GijGijk}), we obtain (\ref{657}) for the case $|\bb
U|=1$, that is,
%
%
\begin{equation}
G_{ij}^{\bb S,\{{a}\}} = G_{ij}^{(\bb S\setminus\{{a}\})} -
G_{ij}^{(\bb S)} = \frac{G_{i{a}}^{(\bb S\setminus\{{a}\})}G_{{a}
j}^{(\bb S\setminus\{{a}\})}}{G_{{a}{a}}^{(\bb S\setminus\{{a}\})}}.
\end{equation}

For a general set $\bb U$ with $|\bb U|\geq2$, using (\ref{619}), we
can write $G^{\bb S,\bb U}$ iteratively as $F -
F^{({a})}$, where $F$ itself is of the form $E - E^{({b})}$ for some
appropriate $E$. For example, for ${a}, {b}\in
\bb U$ we have
\begin{eqnarray*}
G_{ij}^{\bb S, \bb U} & = & G_{ij}^{\bb S\setminus\{{a}\}, \bb
U\setminus\{{a}\}} - \bigl(G_{ij}^{\bb
S\setminus\{{a}\}, \bb U\setminus\{{a}\}}\bigr)^{({a})} \\
& = & G_{ij}^{\bb S\setminus\{{a}{b}\}, \bb U\setminus\{{a}{b}\}}
-\bigl(G_{ij}^{\bb S\setminus\{{a}{b}\}, \bb
U\setminus\{{a}{b}\}} \bigr)^{({b})}\\
&&{} - \bigl(G_{ij}^{\bb S\setminus\{{a}{b}\}, \bb U\setminus\{{a}{b}\}
} -\bigl(G_{ij}^{\bb S\setminus\{{a}{b}\},
\bb U\setminus\{{a}{b}\}} \bigr)^{({b})}\bigr)^{({a})}.
\end{eqnarray*}
Recall $F^{({a})} = F\circ\pi_{a}$ from Definition \ref
{definitionofminor2}. Then to prove (\ref{657}) in the case
$\bb U$ with $|\bb U|\geq2$, we use induction on $|\bb U|$. The key
step is Lemma~\ref{BQXnew} below, which contains
the required properties of $F-F^{({a})}$. Its proof will be given later.
%
%
\begin{lemma}\label{BQXnew}
Let $F$ be of the form (\ref{6488new}).
We assume that
%
%
\begin{equation}
\Biggl|\Biggl({\bigcup_{\al= 1}^n \bb U_\al}\Biggr)\cup
\Biggl({\bigcup_{\beta= 2}^n \bb T_{\beta}}\Biggr)\Biggr| \leq
\frac{1}{(\Lambda_o+\Lambda_d)\log N}-1.
\end{equation}
If
%
%
\begin{equation}
s \notin\{i_1, i_2,i_3, \ldots, i_{n+1}\} \cup\Biggl({\bigcup
_{\al=1}^n \bb U_{\al}}\Biggr)
\cup
\Biggl({\bigcup_{\beta=1}^{n - 1} \bb T_{\beta}}\Biggr),
\end{equation}
then $F-F^{(s)}$ is equal to the sum (with signs $\pm$) of $2n - 1$
terms of the form~(\ref{6488new}),
%
%
\begin{eqnarray}\label{747}
F-F^{(s)} &=& \sum_{l=1}^{n}F_l^A((\tilde i^A_{l,r})_{r =
1}^{\tilde n^A + 1}, (\wt{\bb U}^A_{l,\al})_{\al=
1}^{\tilde n^A}, (\wt{\bb T}^A_{l,\beta})_{\beta= 2}^{\tilde
n^A})\nonumber\\[-8pt]\\[-8pt]
&&{}+\sum_{l=1}^{n-1}F_l^B((\tilde i^B_{l,r})_{r
= 1}^{\tilde n^B + 1}, (\wt{\bb U}^B_{l,\al})_{\al= 1}^{\tilde n^B},
(\wt{\bb T}^B_{l,\beta})_{\beta= 2}^{\tilde
n^B}),\nonumber
\end{eqnarray}
where the new arguments, carrying a tilde, satisfy the following relations:
\begin{longlist}
\item
%
%
\begin{equation}\label{749-}
\tilde n^A = n+1 \quad\mbox{and}\quad \tilde n^B = n+2.
\end{equation}
\item
For $1\leq l\leq n$, the family $(\tilde i^A_{l,r})$ is given by
%
%
\begin{equation}\label{749}
(\tilde i^A_{l,1}, \tilde i^A_{l,2}, \tilde i^A_{l,3}, \ldots,
\tilde i^A_{l,n+2}) \deq(i_1, i_2,\ldots,i_l, s,
i_{l+1},\ldots, i_{n+1}).
\end{equation}
For $1\leq l\leq n-1$, the family $(\tilde i^B_{l,r})$ is given by
%
%
\begin{equation}\label{750}\qquad
(\tilde i^B_{l,1}, \tilde i^B_{l,2}, \tilde i^B_{l,3}, \ldots,
\tilde i^B_{l,n+3})
\deq(i_1, i_2,\ldots,i_{l },i_{l+1}, s, i_{l+1}, i_{l+2},\ldots,
i_{n+1}).
\end{equation}
\item All sets $\wt{\bb U}^{A} _{l,\al}$, $\wt{\bb T}^{A} _{l,\beta
}$, $\wt{\bb U}^{B} _{l,\al}$ and $\wt{\bb
T}^{B} _{l,\beta}$ appearing in (\ref{747}) are subsets of
%
%
\begin{equation}\label{751}
\Biggl({\bigcup_{\al= 1}^n \bb U_{\al}}\Biggr) \cup
\Biggl({\bigcup_{\beta= 2}^n \bb T_{\beta}}\Biggr) \cup\{s\}.
\end{equation}
\end{longlist}
\end{lemma}

Now we return to complete the proof for Lemma~\ref{STG}. Using (\ref
{619}), we get for $s\in\bb U$ and $i,j \notin\bb
S$ that
%
%
\begin{equation}\label{ZLL}
G_{ij}^{\bb S, \bb U}
= G_{ij}^{\bb S\setminus\{s\}, \bb U\setminus\{s\}} -
\bigl(G_{ij}^{\bb S\setminus\{s\},
\bb U\setminus\{s\}}\bigr)^{(s)}.
\end{equation}
Using induction on $|\bb U|$ and applying the decomposition (\ref
{657}) to $G_{ij}^{\bb S\setminus\{s\}, \bb
U\setminus\{s\}}$, we get
%
%
\begin{equation} \label{657f}
G_{ij}^{\bb S\setminus\{s\}, \bb U\setminus\{s\}} = \sum_{n =
|\bb U |}^{2 |\bb U |- 2} F_n,\qquad F_n =
\sum_{k = 1}^{K_n'} F_{n,k},
\end{equation}
where\vspace*{1pt} $\sum_{n=|\bb U|}^{2|\bb U|-2} K_n' \leq4^{| \bb U|-1 }(|\bb
U|-1)!$ and each $F_{n,k}$ is of the form
(\ref{6488new}) (with a possible minus sign) with $i_2,\ldots,i_n \in
\bb U\setminus\{s\}$, $i_1 = i$, $i_{n + 1} = j$,
and with some appropriately chosen sets $(\bb U_\al)_{\al= 1}^n, (\bb
T_\beta)_{\beta= 2}^{n}$ satisfying
\[
\bb
S\setminus\bb U \subset\bb U_\al,\qquad \bb T_\beta
\subset\bb S \setminus\{s\},\qquad 1 \leq\al\leq n
, 2 \leq\beta\leq n.
\]
Now from (\ref{ZLL}) we get
%
%
\begin{equation}\label{9754}
G_{ij}^{\bb S, \bb U} = \sum_{n = |\bb U |}^{2 |
\bb U |
- 2} \sum_{k = 1}^{K_n'}
\bigl( F_{n,k}-( F_{n,k})^{(s)}\bigr).
\end{equation}
Moreover, using (\ref{747}), we get
%
%
\begin{equation}
F_{n,k}-( F_{n,k})^{(s)} = \sum_{l=1}^n F^A_{n,k,l}
+\sum_{l=1}^{n-1} F^B_{n,k,l},
\end{equation}
where each $ F^{A}_{n,k,l}$
and $ F^{B}_{n,k,l}$ is of the form (\ref{6488new}) (with a possible
minus sign) with $i_1 = i$, $i_{m + 1} = j$,
where $m=n+2$ for $F^A_{n,k,l}$ and $m=n+3$ for $F^B_{n,k,l}$, and the
other indices belong to $ \bb U $. Here the
sets $(\wt{\bb U}_\al)_{\al= 1}^m$ and $(\wt{\bb T}_\beta)_{\beta
= 2}^{m}$ satisfy
\[
\bb S\setminus\bb U \subset\wt{\bb U}_\al, \qquad\wt{\bb
T}_\beta
\subset\bb S, \qquad 1 \leq\al\leq m, 2
\leq\beta\leq m.
\]
Furthermore, with (\ref{749-}), the number of off-diagonal elements in
the numerators of $ F^{A}_{n,k,l}$ and
$F^{B}_{n,k,l}$ are $n+1$ and $n+2$, respectively. Hence, together with
(\ref{9754}), we obtain
\[
G_{ij}^{\bb S, \bb U} = \sum_{n = |\bb U |}^{2 |
\bb U |
- 2} \sum_{k = 1}^{K_n'}
\Biggl(\sum_{l=1}^n F^A_{n,k,l}
+\sum_{l=1}^{n-1} F^B_{n,k,l}\Biggr).
\]
With the assumption of $\sum_{n=|\bb U|}^{2|\bb U|-2} K'_n\leq4^{|
\bb U|-1 }(|\bb U|-1)!$ for the summation bounds in
(\ref{657f}), we know that $G_{ij}^{\bb S, \bb U}$ can be written in
the form (\ref{657}) with $\sum_{n = |\bb U |+1}^{2 |
\bb U |} K_n\leq4^{| \bb U| } |\bb U| !$. This completes
the proof of Lemma~\ref{STG}.
\end{pf*}
\begin{pf*}{Proof of Lemma~\ref{BQXnew}}
Using (\ref{GijGijk}), it is easy to derive the following two
identities for $s\notin\bb U$:
%
%
\begin{eqnarray}\label{612}
G_{ij}^{(\bb U)} &=& G_{ij}^{(\bb U s)}
+\frac{G_{is}^{(\bb U)}G_{sj}^{(\bb U)}}{G_{ss}^{(\bb U)}}
\qquad\mbox{for } i,j
\notin\bb U\cup\{ s\},
\\
\label{612a}
\frac{1}{G_{kk}^{(\bb U)}} &=& \frac{1}{G_{kk}^{(\bb U s)}}
+\frac{G_{ks}^{(\bb U)}G_{sk}^{(\bb U)}}{G_{kk}^{(\bb U
s)}G_{ss}^{(\bb U)}G_{kk}^{(\bb U )}} \qquad\mbox{for } k \notin
\bb U\cup
\{s\}.
\end{eqnarray}
Now (\ref{612}) implies that Lemma~\ref{BQXnew} holds in the case
$n=1$. We shall first prove it for the case $n = 2$,\vadjust{\goodbreak}
and then give the proof of the general case. If $n=2$, then by
assumption $F$ has the form
%
%
\begin{equation}
F = \frac{G_{ij}^{(\bb U)}G^{(\bb V)}_{jk}}{G^{(\bb T)}_{jj}}
\end{equation}
with some sets $\bb U$, $\bb V$, $\bb T$ and indices $i$, $j$, $k$. For
$s\notin\bb U\cup\bb V\cup\bb T\cup\{ijk\}$
we get from (\ref{612}) that
%
%
\begin{equation}
F = \frac{G_{ij}^{(\bb U s)}G^{(\bb V)}_{jk}}{G^{(\bb
T)}_{jj}}+\frac{G_{is}^{(\bb U)}G_{sj}^{(\bb U)}G^{(\bb
V)}_{jk}}{G_{ss}^{(\bb U)}G^{(\bb T)}_{jj}}.
\end{equation}
Next, using (\ref{612a}) on the first term, we obtain
%
%
\begin{equation}
F = \frac{G_{ij}^{(\bb U s)}G^{(\bb V)}_{jk}}{G^{(\bb T s)}_{jj}}
+\frac{G_{ij}^{(\bb U s)}G_{js}^{(\bb T)}G_{sj}^{(\bb T)}G^{(\bb
V)}_{jk}}{G_{jj}^{(\bb T s)}G_{ss}^{(\bb
T)}G_{jj}^{(\bb T)}} +\frac{G_{is}^{(\bb U)}G_{sj}^{(\bb U)}G^{(\bb
V)}_{jk}}{G_{ss}^{(\bb U)}G^{(\bb T)}_{jj}}.
\end{equation}
Using (\ref{612}) again on the first term, we have
%
%
\begin{equation}\hspace*{28pt}
F = F^{(s)}+\frac{G_{ij}^{(\bb U s)}G^{(\bb V)}_{js}G^{(\bb
V)}_{sk}}{G^{(\bb T s)}_{jj}G^{(\bb V)}_{ss}}
+\frac{G_{ij}^{(\bb U s)}G_{js}^{(\bb T)}G_{sj}^{(\bb T)}G^{(\bb
V)}_{jk}}{G_{jj}^{(\bb T s)}G_{ss}^{(\bb
T)}G_{jj}^{(\bb T)}} +\frac{G_{is}^{(\bb U)}G_{sj}^{(\bb U)}G^{(\bb
V)}_{jk}}{G_{ss}^{(\bb U)}G^{(\bb T)}_{jj}}.
\end{equation}
One can easily check that the last three terms are of the form (\ref
{6488new}), and the indices satisfy
(\ref{749-})--(\ref{751}). This completes the proof for Lemma \ref
{BQXnew} in the case $n=2$.

Now we consider the case of a general $n$. Inserting (\ref{612}) and
(\ref{612a}) into each term in (\ref{6488new}), we have
%
%
\begin{equation}
F((i_r)_{r = 1}^{n+1}, (\bb U_\al)_{\al= 1}^n, (\bb T_\beta
)_{\beta= 2}^{n} ) = \frac PQ,
\end{equation}
where
\[
P = \prod_{\al=1}^n G^{(\bb U_{\al})}_{i_\al,i_{\al+1}}
= \prod_{\al=1}^n\biggl(G_{i_\al,i_{\al+1}}^{(\bb U_{\al} s)}
+\frac{G_{i_\al,s}^{(\bb U_{\al})}G_{s i_{\al+1}}^{(\bb U_{\al
})}}{G_{ss}^{(\bb U_{\al})}}\biggr)
\]
and
\[
Q^{-1} = \prod_{\beta=2}^{n } \bigl(G^{(\bb T_{\beta
})}_{i_\beta,i_\beta} \bigr)^{-1}
= \prod_{\beta=2}^{n } \biggl(
\frac{1}{G_{i_\beta,i_\beta}^{( \bb T_{\beta} s)}}
+\frac{G_{i_\beta s}^{( \bb T_{\beta} )}G_{s i_\beta}^{( \bb
T_{\beta} )}}{G_{i_\beta i_\beta}^{( \bb T_{\beta}
)}G_{ss}^{( \bb T_{\beta} )}G_{i_\beta i_\beta}^{(\bb T_{\beta} s)}}
\biggr).
\]
On the other hand,
%
%
\begin{equation}
(F((i_r)_{r = 1}^{n+1}, (\bb U_\al)_{\al= 1}^n, (\bb
T_\beta)_{\beta= 2}^{n} ) )^{(s)} = \frac
{P^{(s)}}{Q^{(s)}},
\end{equation}
where
\[
P^{(s)} = \prod_{\al=1}^n G^{(\bb U_{\al} s)}_{i_\al,i_{\al+1}}
\quad\mbox{and}\quad \bigl(Q^{(s)}\bigr) ^{-1} = \prod_{\beta=2}^{n }
\bigl(G^{(\bb T_{\beta}
s)}_{i_\beta,i_\beta} \bigr)^{-1}.\vadjust{\goodbreak}
\]

For $m\in\N$, we write, using (\ref{612}) and (\ref{612a}),
%
%
\begin{equation}
G^{(\bb U_{ m})}_{i_{ m},i_{ m+1}} =
A_{2m-1} +B_{2m-1},
\end{equation}
where
%
%
\begin{equation}
A_{2m-1} \deq G^{(\bb U_{ m} s)}_{i_{ m},i_{ m+1}}
\quad\mbox{and}\quad
B_{2m-1} \deq\frac{G^{(\bb U_{ m})}_{i_{m},s}G^{(\bb
U_{m})}_{s,i_{ m+1}}}
{G^{(\bb U_{m})}_{ss}}.
\end{equation}
Similarly, we write
%
%
\begin{equation}
\bigl(G^{(\bb T_{ m+1})}_{i_{ m+1},i_{ m+1}}\bigr)^{-1} =
A_{2m}+B_{2m},
\end{equation}
where
%
%
\begin{equation}
A_{2m} \deq\bigl(G^{(\bb T_{ m+1} s)}_{i_{ m+1},i_{ m+1}}\bigr)^{-1}
\quad\mbox{and}\quad
B_{2m} \deq\frac{G_{i_{ m+1} s}^{( \bb T_{{ m+1}} )}G_{s i_{
m+1}}^{( \bb T_{{ m+1}} )}}
{G_{i_{ m+1} i_{ m+1}}^{( \bb T_{{ m+1}} )}G_{ss}^{( \bb T_{{ m+1}}
)}G_{i_{ m+1} i_{ m+1}}^{(\bb T_{{ m+1}}
s)}}.\hspace*{-35pt}
\end{equation}
Then
\[
F = \prod_{m=1}^{2n-1}(A_m+B_m),\qquad F^{(s)} = \prod
_{m=1}^{2n-1} A_m.
\]

To complete the proof, we use the identity
%
%
\begin{equation}\label{naid}
\prod_{m=1}^{2n-1}(A_m+B_m)-\prod_{m=1}^{2n-1}A_m
= \sum_{m=1}^{2n-1} \Biggl(\prod_{j=1}^{m-1}A_j\Biggr) B_m
\Biggl(\prod_{j=m+1}^{2n-1}(A_j+B_j)\Biggr).\hspace*{-35pt}
\end{equation}
It is easy to check that, for any term of the form
\[
\Biggl(\prod_{j=1}^{m-1}A_j\Biggr) B_m \Biggl(\prod
_{j=m+1}^{2n-1}(A_j+B_j)\Biggr)
\]
in the sum (\ref{naid}), the desired properties (\ref{747})--(\ref
{751}) hold.
\end{pf*}

We may now easily obtain the following bound on $G_{ij}^{\bb S,\bb U}$.
%
%
\begin{lemma}\label{lemma612}
Let $\bb U\subset\bb S\subset\{1,2,\ldots, N\}$ and
%
%
\begin{equation}\label{655e}
|\bb S| \leq\frac{1}{(\Lambda_o+\Lambda_d)\log N}.
\end{equation}
Then
%
%
\begin{equation}\label{616a}
|G^{\bb S, \varnothing}_{ij}-m_{\mathrm{sc}} \delta_{ij}| \leq C
\bigl({\f1 }(i=j)\Lambda_d+\Lambda_o\bigr).
\end{equation}
If in addition $\bb U\neq\varnothing$ and $i,j\notin\bb S$, then
%
%
\begin{equation}\label{616}
|G_{ij}^{\bb S,\bb U}| \leq(C |\bb U| \Lambda
_o)^{|\bb U|+1}
\end{equation}
and
%
%
\begin{equation}\label{newGiiSU}
|(1/G_{ii})^{\bb S,\bb U} | \leq(C |\bb U|
\Lambda_o)^{|\bb U|+1}.
\end{equation}
\end{lemma}
\begin{pf}
The estimate (\ref{616a}) follows (\ref{656}) and (\ref{Gijdeltaij}).
In order to prove (\ref{616}), we apply Lemma
\ref{STG} to each $G_{ij}^{\bb S,\bb U}$, and get
%
%
\begin{equation}
G_{ij}^{\bb S,\bb U} = \sum_{n = |\bb U |+1}^{2 |
\bb U |}
F_n,\qquad F_n = \sum_{k = 1}^{K_n} F_{n,k},
\end{equation}
where $\sum_{n = |\bb U |+1}^{2 |\bb U |}K_n \leq
4^{|\bb
U|} {|\bb U| }!$. Here each $F_{n,k}$ is of the form
(\ref{6488new}) (with a possible minus sign), where $n$ counts the
number of off-diagonal elements in the numerator;
the indices satisfy $i_2,\ldots,i_n \in\bb U$, $i_1 = i$, $i_{n + 1}
= j$. Note that the factors $P$ in
(\ref{6488new}) are the product of off-diagonal terms and the factors
$Q$ the product of diagonal terms. Applying
(\ref{Gijdeltaij}) and (\ref{Giilowerbound}) on the off-diagonal
and diagonal terms in $P$ and $Q$, we get
%
%
\begin{equation}
F_{n,k} \leq\frac{(C \Lambda_o)^{n }}{c^{n-1}} \leq(C
\Lambda_o)^{n }.
\end{equation}
Together with $ \sum_{n = |\bb U |+1}^{2 |\bb U |
}K_n \leq
4^{|\bb U|} {|\bb U| }!$, this implies (\ref{616}).\vspace*{2pt}

In order to prove (\ref{newGiiSU}), we observe that, similarly to
Lemma~\ref{STG}, we have
%
%
\begin{equation}
|(1/G_{ii})^{\bb S,\bb U} | \leq(C|\bb U|)^{|\bb U|+1}
\frac{
(\max_{k,j\notin\bb T, \bb T\subset{\bb S}} |
G_{kj}^{( \bb T)} |)^{|\bb U|+1}
}
{
(\min_{j\notin\bb T, \bb T\subset{\bb S} } |
G_{jj}^{( \bb T)} |)^{|\bb U|+2}}
\end{equation}
provided that
\[
\max_{k,j\notin\bb T, \bb T\subset{\bb S}} \bigl| G_{kj}^{( \bb
T)} \bigr|\leq\min_{j\notin\bb T, \bb T\subset{\bb S} }
\bigl| G_{jj}^{( \bb T)} \bigr|.
\]
Hence, (\ref{newGiiSU}) follows.\vspace*{-2pt}
\end{pf}

\subsection{\texorpdfstring{Proof of Lemma \protect\ref{zlemmasparse}}{Proof of Lemma 4.1}} \label{sectproofof41}
Observe first that (\ref{reszlemmasparse2}) follows immediately from
(\ref{reszlemmasparse}) and Lemma~\ref{lemmaexpandedself-consisteneq}.
It therefore remains to prove (\ref
{reszlemmasparse}).

We define the event $\Xi$ by requiring that on it (\ref{111}) and the
following two events hold:
\begin{longlist}
\item
For every $z\in\wt D $ we have
%
%
\begin{eqnarray}\label{L71}
\Lambda_o(z) &\leq& C\biggl(\frac1q+ (\log N)^{2\xi}\Psi(z)
\biggr) \nonumber\\[-8pt]\\[-8pt]
&\leq& C \Biggl(\frac1q+ (\log
N)^{2\xi}\sqrt{\frac{ \im m_{\mathrm{sc}}(z)+\gamma(z)}{N
\eta}}\Biggr).\nonumber
\end{eqnarray}
\item
For every $z\in\wt D $ we have
%
%
\begin{eqnarray}\label{L72}\qquad
\max_i |G_{ii}(z) - m(z)| & \leq & C\biggl(\frac{(\log
N)^{\xi}}q+(\log N)^{2\xi}\Psi(z) \biggr)
\nonumber\\[-8pt]\\[-8pt]
& \leq & C\Biggl(\frac{(\log N)^{\xi}}q+(\log N)^{2\xi}\sqrt
{\frac{ \im m_{\mathrm{sc}}(z)+\gamma(z)}{N \eta}}\Biggr).\nonumber\vadjust{\goodbreak}
\end{eqnarray}
\end{longlist}
Now Theorem~\ref{WLSCTHM}, Lemmas~\ref{lemmaoff-diagestimate} and
\ref{lemmadiagestimate}, as well as (\ref{111})
and $\wt D \subset D_L$ imply that $\Xi$ holds with $(\xi-1/2,\nu
)$-high probability.
Note that here we reduced the $\xi$ to $\xi- 1/2$ to
account for the intersection of three events of $(\xi,\nu)$-high
probability. It is crucial that $\nu$ remain constant
in this step, as in some applications, such as Theorem~\ref{LSCTHM},
it is iterated.

We write $Z_i$ as
%
%
\begin{equation} Z_i =
\sum_k^{(i)} \biggl({h_{ik}^2 - \frac{1}{N}}\biggr) G^{(i)}_{kk} +
\sum_{k \neq l}^{(i)} h_{ik} G^{(i)}_{kl} h_{li}.
\end{equation}
Lemma~\ref{zlemmasparse} follows from the next two lemmas. As before,
we shall consistently omit the spectral parameter
$z \in\wt D$ from the notation in the following arguments.
%
%
\begin{lemma}\label{sp72}
On $\Xi$ we have with $(\xi,\nu)$-high probability
%
%
\begin{eqnarray}\label{ressp72}
&&\Biggl|{\f1}(\Xi) \frac1N\sum_i\sum_k^{(i)} \biggl({h_{ik}^2 - \frac
{1}{N}}\biggr) G^{(i)}_{kk}\Biggr|
\nonumber\\[-8pt]\\[-8pt]
&&\qquad\leq(\log N)^{4\xi}\biggl(\frac1{q^2}+\frac1{ (N\eta
)^2}+\frac{ \im m_{\mathrm{sc}} +\gamma}{N \eta}\biggr).\nonumber
\end{eqnarray}
\end{lemma}
%
%
\begin{lemma}\label{sp74}
On $\Xi$ we have with $(\xi- 2,\nu)$-high probability
%
%
\begin{eqnarray}\label{ressp74}
&&\Biggl|{\f1}(\Xi) \frac1N\sum_i\sum_{k \neq l}^{(i)} h_{ik} G^{(i)}_{kl}
h_{li}\Biggr|\nonumber\\[-8pt]\\[-8pt]
&&\qquad\leq(\log N)^{14\xi}\biggl(\frac1{q^2}+\frac1{ (N\eta
)^2}+(\log N)^{ 4\xi}\frac{ \im m_{\mathrm{sc}} +\gamma}{N
\eta}\biggr).\nonumber
\end{eqnarray}
\end{lemma}
\begin{pf*}{Proof of Lemma~\ref{sp72}}
We split the sum inside the absolute value on the left-hand side of
(\ref{ressp72}) as
%
%
\begin{eqnarray} \label{715}
&&\frac1N \sum_{i\neq k} \biggl({h_{ik}^2 - \frac{1}{N}}\biggr)
m + \frac1N\sum_{i\neq k}\biggl({h_{ik}^2 - \frac{1}{N}}\biggr)
\bigl(m^{(i)} -m \bigr)\nonumber\\[-8pt]\\[-8pt]
&&\qquad{}+ \frac1N\sum_{i\neq k} \biggl({h_{ik}^2 - \frac{1}{N}}\biggr)
\bigl(G^{(i)}_{kk}-m^{(i)} \bigr).\nonumber
\end{eqnarray}
In order to estimate the first term of (\ref{715}), we use the
estimate (\ref{generalizedLDE}) [with
$(h_{ik}^2-N^{-1})$ playing the role of $a_i$, and setting
$A_i=N^{-1}$, $\alpha=2$, $\beta=-2$ and\vadjust{\goodbreak} $\gamma=1$]
to get, with $(\xi,\nu)$-high probability,
%
%
\begin{equation}
\biggl|\frac1N \sum_{i\neq k} \biggl({h_{ik}^2 - \frac{1}{N}}
\biggr) \biggr| \leq( \log N)^{\xi}\frac{1}{N^{1/2}q}.
\end{equation}
Therefore,
%
%
\begin{eqnarray}\label{716}
\biggl|{\f1}(\Xi)\biggl(\frac1N \sum_{i\neq k} \biggl({h_{ik}^2 - \frac
{1}{N}}\biggr) \biggr) m \biggr|
&\leq&
( \log N)^{\xi}
\biggl| {\f1}(\Xi)
\frac{m}{N^{1/2}q} \biggr|\nonumber\\[-8pt]\\[-8pt]
&\leq& C( \log N)^{\xi} \frac{1}{N^{1/2}q}.\nonumber
\end{eqnarray}

Similarly, in order to estimate the second term of (\ref{715}), we fix
$i$ and sum over~$k$, which yields
with $(\xi,\nu)$-high probability
%
%
\begin{equation}
\Biggl|\max_i \sum_{k}^{(i)}
\biggl(h_{ik}^2 - \frac{1}{N}\biggr)\Biggr| \leq( \log N)^{\xi
}q^{-1},
\end{equation}
where the sum over $k$ was estimated by (\ref{generalizedLDE}). This
yields with $(\xi,\nu)$-high probability
%
%
\begin{eqnarray}
&&\Biggl|{\f1}(\Xi) \frac1N
\sum_i
\sum_{k}^{(i)}
\biggl(h_{ik}^2 - \frac{1}{N}\biggr)
\bigl({m^{(i)}-m}\bigr) \Biggr| \nonumber\\[-8pt]\\[-8pt]
&&\qquad\leq\frac1N\sum_i\bigl| {\f
1}(\Xi) ( \log N)^{\xi}q^{-1}
\bigl({m^{(i)} -m}\bigr) \bigr|.\nonumber
\end{eqnarray}
Using (\ref{iTGT}), we have in $\Xi$
\[
\bigl|m^{(i)}-m \bigr| = \biggl| - \frac{1}{N} \sum_{j} \frac{
G_{ji} G_{ij} }{ G_{ii} } \biggr| \leq
O\biggl(\frac{\im G_{ii}}{ \eta}\biggr).
\]
Thus, we get with $(\xi,\nu)$-high probability
%
%
\begin{equation}\label{718}
\biggl| {\f1}(\Xi) \frac1N\sum_{i\neq k}\biggl({h_{ik}^2 - \frac
{1}{N}}\biggr) \bigl(m^{(i)} -m \bigr) \biggr| \leq C(
\log N)^{\xi}\frac{\im m_{\mathrm{sc}}+\gamma}{qN \eta}.
\end{equation}

Finally, we estimate the third term of (\ref{715}). First, with (\ref
{GijGijk}) and $| m_{\mathrm{sc}} |\geq c$, we note that
if $\Lambda_d\ll1$, then
%
%
\begin{equation}
\bigl|G_{ij}-G_{ij}^{(k)}\bigr| \leq C\Lambda_o^2 \qquad\mbox
{for } i,j\neq
k.
\end{equation}
Together with (\ref{L72}) we get with $(\xi,\nu)$-high probability
%
%
\begin{equation}
\max_{k\neq i}\bigl| \bigl(G^{(i)}_{kk}-m^{(i)} \bigr) \bigr|
\leq C(\log N)^{2\xi}\Biggl(\frac1q+\sqrt{\frac{
\im m_{\mathrm{sc}}+\gamma}{N \eta}}\Biggr).
\end{equation}
Then we use (\ref{aA}) [with $(h_{ik}^2-N^{-1})$ playing the role of
$a_k$ and $G^{(i)}_{kk}-m^{(i)} $ playing the role
of $A_k$, and setting $\alpha=2$, $\beta=-2$ and $\gamma=1$] to get,
with $(\xi,\nu)$-high probability,
%
%
\begin{equation}
\max_i\Biggl|\sum_{k}^{(i)}\!
\biggl(h_{ik}^2 - \frac{1}{N}\biggr) \bigl(G^{(i)}_{kk}-m^{(i)}
\bigr) \Biggr| \leq(\log N)^{4\xi}\Biggl(\frac1q+\sqrt
{\frac{ \im m_{\mathrm{sc}}+\gamma}{N \eta}}\Biggr)q^{ -1}.\hspace*{-40pt}
\end{equation}
Hence, we have with $(\xi,\nu)$-high probability
%
%
\begin{eqnarray}\label{719}
&&\Biggl|{\f1}(\Xi) \frac1N
\sum_i
\Biggl({\sum_{k}^{(i)} \biggl(h_{ik}^2 - \frac{1}{N}\biggr)
\bigl(G^{(i)}_{kk}-m^{(i)} \bigr) }\Biggr)
\Biggr|\nonumber\\[-8pt]\\[-8pt]
&&\qquad \leq(\log N)^{4\xi}\Biggl(\frac1q+\sqrt{\frac{ \im
m_{\mathrm{sc}}+\gamma}{N \eta}}\Biggr)q^{ -1}.\nonumber
\end{eqnarray}
Note that, in applying (\ref{aA}), we used that the family $\{
h_{ik}^2-N^{-1}\}_k$ is independent of the family
$\{G^{(i)}_{kk}-m^{(i)}\}_k$. Combining (\ref{716}), (\ref{718}) and
(\ref{719}), we obtain (\ref{ressp72}).
\end{pf*}
\begin{pf*}{Proof of Lemma~\ref{sp74}}
We shall apply Theorem~\ref{abstractZlemma} to the quantities
%
%
\begin{equation}\label{735new}
\cal Z_i \deq\sum_{k \neq l}^{(i )} h_{ik}G^{(i )}_{kl} h_{li},\qquad
\cal Z_i^{[\bb V]} \deq\f1 (i \notin\bb V) \sum_{k \neq
l}^{(i\bb V)} h_{ik} G^{(i\bb V)}_{kl} h_{li},
\end{equation}
and define $\Xi$ as in the beginning of Section~\ref{sectproofof41},
that is, $\Xi$ is defined by requiring that
(\ref{111}) and (\ref{L71})--(\ref{L72}) hold. Recall that the
collection of random variables $Z_i^{[\bb V]}$ generates
random variables $Z_i^{\bb S,\bb U}$ for any $\bb U\subset\bb S$ by~(\ref{56ASU}).
Let\looseness=-1
%
%
\begin{equation}
p \deq(\log N)^{\xi-3/2}.
\end{equation}\looseness=0
Next, choose
%
%
\begin{eqnarray}\label{736new}
X & \deq & \frac1q+(\log N)^{2\xi}\sqrt{\frac{ \im m_{\mathrm{sc}} +\gamma
}{N \eta}},
\nonumber\\[-8pt]\\[-8pt]
Y & \deq & (\log N)^{2\xi}.\nonumber
\end{eqnarray}
[In other words, $X$ is defined as the right-hand side of (\ref{L71})
up to a constant.] We now derive a bound which
implies both (\ref{511}) and (\ref{5112}), that is, we establish the
assumptions (ii) and (iii) of Theorem
\ref{abstractZlemma}. To this end, we shall prove the stronger
statement that, for $i\notin\bb S$, $r\leq p$ and any
sets $\bb U\subset\bb S$ with $|\bb S|\leq p$, we have
%
%
\begin{equation} \label{737new}
\E({\f1}([ \Xi]_i)
|\cal Z_i^{\bb S, \bb U}|^{r}) \leq( Y (CX u
)^{u})^{r} \qquad\mbox{for } u = |\bb U|+1.
\end{equation}

Using the assumptions of Lemma~\ref{zlemmasparse}, we have in $\wt D$ that
%
%
\begin{equation}
q \geq(\log N)^{5\xi},\qquad
N\eta\geq(\log N)^{14\xi},\qquad \gamma\leq(\log
N)^{-\xi}.\vadjust{\goodbreak}
\end{equation}
It is therefore easy to check that $\cal Z_i$ and $\Xi$ satisfy the
assumptions (i), (iv) and (v) of Theorem
\ref{abstractZlemma}. Thus, the conclusion of Theorem \ref
{abstractZlemma}, (\ref{resZZ}), implies the claim~(\ref{ressp74}).

It remains to prove (\ref{737new}). By the definition of $\cal
Z_i^{\bb S,\bb U}$ in (\ref{56ASU}) and (\ref{735new}),
for $i\notin\bb S$, we have
%
%
\begin{eqnarray}
\cal Z_i^{\bb S,\bb U} & = & (-1)^{|\bb S \setminus\bb U |}
\sum
_{\bb V \st\bb S \setminus\bb U \subset\bb V \subset\bb S}
(-1)^{|\bb V |} \sum_{k
\neq l}^{(i\bb V)} h_{ik} G^{(i\bb V)}_{kl} h_{li}\nonumber\\
& = & (-1)^{|\bb S \setminus\bb U |} \sum_{k \neq
l}^{(i\cup\bb
S \setminus\bb U)}
\sum_{\bb V \st\bb S \setminus\bb U \subset\bb V \subset\bb S
\setminus\{k,l\}} (-1)^{|\bb V |} h_{ik} G^{(i\bb V)}_{kl} h_{li}
\nonumber\\[-8pt]\\[-8pt]
& = & \sum_{k \neq l}^{(i\cup\bb S \setminus\bb U)} h_{ik} h_{li}
\sum_{\bb V \st\bb S \setminus\bb U \subset\bb V \subset\bb S
\setminus\{k,l\}} (-1)^{|\bb S \setminus\bb U |}(-1)^{|
\bb V |}
G^{(i\bb V)}_{kl}\nonumber\\
& = & \sum_{k \neq l}^{(i\cup\bb S \setminus\bb U)} h_{ik} h_{li}
[G_{kl}]^{( \bb Si)\setminus\{k,l\}, \bb U\setminus\{
k,l\}},\nonumber
\end{eqnarray}
where in the last equality we used the definition of $G^{\bb S,\bb U}$,
Definition~\ref{58G}. Thus, we may write
%
%
\begin{equation}
\cal Z_i^{\bb S,\bb U}
= A_1+A_2+A_3+A_4,
\end{equation}
where
\begin{eqnarray*}
A_1 & \deq & \sum_{ k \in\bb U } \sum_{ l\in\bb U\setminus\{ k\}}
h_{ik} h_{li} [G_{kl}]^{( \bb
Si)\setminus\{k,l\}, \bb U\setminus\{ k,l\}},\\
A_2 &\deq&\sum_{k\in\bb U }\sum_l^{( \bb Si k)} h_{ik} h_{li}
[G_{kl}]^{( \bb
Si)\setminus\{k\}, \bb U\setminus\{ k\}},
\\
A_3 & \deq & \sum_{l\in\bb U }\sum_k^{( \bb Sil)} h_{ik} h_{li}
[G_{kl}]^{( \bb Si)\setminus\{l\}, \bb
U\setminus\{ l\}},\\
A_4 &\deq&\sum_{k\neq l }^{(\bb Si)} h_{ik} h_{li}
[G_{kl}]^{( \bb Si), \bb U }.
\end{eqnarray*}
Now we have
%
%
\begin{eqnarray}\label{745new}
\E({\f1}([ \Xi]_i) |\cal Z_i^{\bb
S, \bb U} |^r)
&=& \E\bigl({{\f1}([ \Xi]_i)|
A_1+A_2+A_3+A_4 |^r}\bigr) \nonumber\\[-8pt]\\[-8pt]
&\leq&4^r\sum_{j=1}^4
\E({\f1}([ \Xi]_i)| A_j|^r),\nonumber
\end{eqnarray}
and we are going to bound
$\E({\f1}([ \Xi]_i)| A_j|^r)$
for each $j=1,2,3,4$. Using the assumption (\ref{111}),
that is,
%
%
\begin{equation}
\Lambda\leq\gamma\leq(\log N)^{-\xi},
\end{equation}
(\ref{L71}) and (\ref{L72}), we get $\Lambda_o+\Lambda_d\leq C(\log
N)^{-\xi}$, which implies the assumption
(\ref{655e}) of Lemma~\ref{lemma612}.

Throughout the following we set $u \deq|\bb U |+ 1$.
We begin by estimating the contribution of $A_1$. Observe that if
$i\neq k,l$, $i\in\bb A$ and $i\notin\bb B $, then
$[G_{kl}]^{\bb A, \bb B}$ is independent of the $i$th row
and column of $H$. (The same argument will
be repeatedly used in the rest of the proof below.) Thus, we have
\[
\bigl\|{\f1}([ \Xi]_i) [G_{kl}]^{(
\bb Si)\setminus\{k,l\}, \bb U\setminus\{ k,l\}} \bigr\|
_\infty
=
\bigl\|{\f1}(\Xi) [G_{kl}]^{( \bb Si)\setminus\{
k,l\}, \bb U\setminus\{ k,l\}} \bigr\|_\infty
\leq(C |\bb U | X)^{|\bb U |- 1},
\]
where in the second step we used (\ref{616}) and $\Lambda_o \leq C X$
on $\Xi$. Thus, we find, using $|\bb U |\leq
|\bb S |\leq p = (\log N)^{\xi- 3/2}$ and $q^{-1} \leq X$, that
%
%
\begin{eqnarray}\label{r5100}
\E{\f1}([ \Xi]_i)| A_1|^r
& \leq & (\log N)^{2\xi} \max_{i,k,l} \E| h_{ik} |^r
| h_{li} |^r \bigl({(C |\bb U | X)^{|\bb U
|- 1}}\bigr)^r
\nonumber\\
& \leq & (\log N)^{2 \xi} q^{-2r} \bigl({(C |\bb U |
X)^{|\bb U |- 1}}\bigr)^r
\\
& \leq & ({Y (C X u)^u})^r.\nonumber
\end{eqnarray}

In order to bound the contribution of $A_2$, we estimate, as above,
\[
\bigl\|\f1 ({[\Xi]_i}) [G_{kl}]^{(
\bb Si)\setminus\{k\}, \bb U\setminus\{ k\}} \bigr\|_{\infty}
=
\bigl\|\f1 (\Xi) [G_{kl}]^{( \bb Si)\setminus\{
k\}, \bb U\setminus\{ k\}} \bigr\|_{\infty} \leq(C |
\bb U | X)^{|\bb U |},
\]
where in the last step we used (\ref{616}) and $\Lambda_o \leq C X$
on $\Xi$. Thus, we may apply the moment estimate
(\ref{aaBoapp}) from the \hyperref[app]{Appendix} with
\[
B_{kl} \deq\f1 (k \in\bb U) \f1 (l \notin\bb S \cup\{{i}\})
\f1 ({[\Xi]_i}) [G_{kl}]^{( \bb
Si)\setminus\{k\}, \bb U\setminus\{ k\}}.
\]
This yields
\begin{eqnarray*}
\E{\f1}([ \Xi]_i)| A_2|^r &\leq&
(C r)^{2 r} \biggl({\biggl({\frac{1}{q} + \biggl({\frac{1}{N^2}
(\log N)^\xi N}\biggr)^{1/2}}\biggr) (C |\bb U | X)^{|
\bb U |}}\biggr)^r\\
&\leq&(C r)^{2 r} \bigl({X (C |\bb U | X)^{|\bb U
|}}\bigr)^r,
\end{eqnarray*}
where we used that $|\bb U |\leq(\log N)^\xi$, that the $B_{kl}$
defined above are independent of the randomness
in the $i$th column of $H$, and that
\[
\frac{1}{q} + \frac{(\log N)^{\xi/2}}{\sqrt{N}} \leq\frac
{1}{q} + (\log N)^{2 \xi} \sqrt{\frac{\im m_{\mathrm{sc}}}{N \eta}} \leq
X
\]
as follows from $\im m_{\mathrm{sc}} \geq\sqrt{\eta}$ and $\eta\leq3$.
Thus, we get
%
%
\begin{equation}\label{A2qyxqnew}
\E{\f1}([ \Xi]_i)| A_2|^r \leq
( Y (CX u )^{u})^{r}.
\end{equation}
(Recall that $u = |\bb U |+ 1$.)

Exchanging $k$ and $l$ in the above estimate of $A_2$, we obtain
%
%
\begin{equation}\label{A3qyxqnew}
\E{\f1}([ \Xi]_i)| A_3|^r \leq
( Y (CX u)^{u})^{r}.
\end{equation}

Finally, we estimate the contribution of $A_4$. As above, we estimate
\[
\bigl\|\f1 ({[\Xi]_i}) [G_{kl}]^{( \bb
Si), \bb U } \bigr\|_\infty=
\bigl\|\f1 (\Xi) [G_{kl}]^{( \bb Si), \bb U }
\bigr\|_\infty
\leq(C |\bb U | X)^{|\bb U |+ 1}
\]
by (\ref{616}) and $\Lambda_o \leq C X$ on $\Xi$. We may now apply
the moment estimate (\ref{aaBoapp}) the \hyperref[app]{Appendix} with
\[
B_{kl} \deq\f1 ({k,l \notin\bb S \cup\{i\}}) \f
1 ({[\Xi]_i}) [G_{kl}]^{( \bb Si), \bb U }
.
\]
This yields
\[
\E{\f1}([ \Xi]_i)| A_2|^r \leq
\bigl({r^2 (C |\bb U | X)^{|\bb U |+ 1}}
\bigr)^r,
\]
where we used that the $B_{kl}$ are independent of the randomness in
the $i$th column of $H$. This gives
%
%
\begin{equation}\label{A4qyxqnew}
\E{\f1}([ \Xi]_i)| A_2|^r \leq
( Y (CX u )^{u})^{r}.
\end{equation}

Combining (\ref{r5100}), (\ref{A2qyxqnew}), (\ref{A3qyxqnew}) and
(\ref{A4qyxqnew}), we obtain (\ref{737new}). This
completes the proof.
\end{pf*}

\section{The largest eigenvalue of $A$} \label{sectionlambdamax}
\subsection{Eigenvalue interlacing}
We now concentrate on the spectrum of $A$. We begin by proving the
following interlacing property. Recall that
$\lambda_1 \leq\cdots\leq\lambda_N$ denote the eigenvalues of $H$
and $\mu_1 \leq\cdots\leq\mu_N \eqd\mu_{\max}$
the eigenvalues of $A$. The associated eigenvectors of $H$ are denoted
by $\f u_1,\ldots, \f u_N$, and those of $A$ by
$\f v_1,\ldots, \f v_N \eqd\f v_{\max}$. Also, we set $\nc G(z)
\deq(A - z)^{-1}$.
%
%
\begin{lemma}
The eigenvalues of $H$ and $A$ are interlaced,
%
%
\begin{equation} \label{interlacing}
\lambda_1 \leq\mu_1 \leq\lambda_2 \leq\mu_2
\leq\cdots\leq\mu_{N-1} \leq\lambda_N
\leq\mu_N.
\end{equation}
\end{lemma}
\begin{pf}
We use the identity
%
%
\begin{equation} \label{perturbationidentity}
\langle{\f e}, {\nc G \f e}\rangle^{-1} = f + \langle{\f e}
, {G \f e}\rangle^{-1},
\end{equation}
which follows by taking $\langle{\f e}, { \cdot\f e}\rangle$ in
\[
\nc G(z) (A - z) G(z) = \nc G(z) (H - z) G(z) + f \nc G(z) | \f e
\rangle\langle\f e | G(z).
\]
From (\ref{perturbationidentity}) we get
%
%
\begin{equation} \label{R-Ginv}
\biggl({\sum_\alpha\frac{|\langle{\f v_\alpha}, {\f
e}\rangle|^2}{\mu
_\alpha- z}}\biggr)^{-1}
=
f +
\biggl({\sum_\alpha\frac{|\langle{\f u_\alpha}, {\f
e}\rangle|^2}{\lambda_\alpha- z}}\biggr)^{-1}.
\end{equation}
It is easy to see that the left-hand side of (\ref{R-Ginv}) defines a
function of $z \in\R$ with $N - 1$ singularities
and $N$ zeros, which is smooth and decreasing away from the
singularities. Moreover, its zeros are the eigenvalues of $A$.
The interlacing property now follows from the fact that $z$ is an
eigenvalue of $H$ if and only if the right-hand
side of (\ref{R-Ginv}) is equal to $f$.
\end{pf}

\subsection{\texorpdfstring{The laws of $\mu_{\max}$ and $v_{\max}$}
{The laws of mu max and v max}}

In this section we establish the basic properties of $\mu_{\max}$
and $\f v_{\max}$. We make the assumption that
$f \geq1 + \epsilon_0$ uniformly in $N$ [see (\ref{lowerboundonf})
below], which is necessary to guarantee that $\mu_{N - 1}$
and $\mu_N$ are separated by a gap of order one. Note that in \cite
{PecheFeral} it was proved, in the case where $H$
is a Hermitian Wigner matrix, that if $f \leq1$ no such gap exists.

The following result collects the main properties of $\mu_{\max}$
and $\f v_{\max}$ for the rank-one perturbation $A = H + f | \f e
\rangle\langle\f e |$ of the sparse matrix $H$. The most important technical
result is (\ref{uNe}). It states that, for large $f$, the
eigenvector $\f v_{\max}$
is almost parallel to the perturbation $\f e$. Consequently, $\f e$ is
almost orthogonal to the eigenvectors $\f v_\alpha$ for $\alpha=
1,\ldots, N - 1$ (Corollary~\ref{corollary67}). As it turns out, this
near orthogonality is the key input for establishing the local
semicircle law for $A$ in Section~\ref{sectwtG}. We refer to the
discussion at the beginning of Section~\ref{sectgoodeventforwtG}
for more details on the use of Corollary~\ref{corollary67}.
%
%
\begin{theorem}\label{theoremlargesteigenvalue}
Suppose that $A$ satisfies Definition~\ref{definitionofA} and that
in addition to (\ref{upperboundonf}) we have the
lower bound
%
%
\begin{equation} \label{lowerboundonf}
f \geq1 + \epsilon_0
\end{equation}
for some constant $\epsilon_0 > 0$.

Then we have with $(\xi,\nu)$-high probability
%
%
\begin{equation} \label{positionofmumax}
\mu_{\max} = f + \frac{1}{f} + o(1).
\end{equation}
In particular, there is a constant $c$, depending on $\epsilon_0$,
such that with $(\xi,\nu)$-high probability we have
%
%
\begin{equation} \label{gap}
\mu_{\max} \geq2 + c.
\end{equation}
Also, we have
%
%
\begin{equation} \label{ElambdaN}
\E\mu_{\max} = f + \frac{1}{f} + O \biggl({\frac{1}{f^3} +
\frac{1}{f^2 q} + \frac{1}{f N}}\biggr)
\end{equation}
as well as, with $(\xi,\nu)$-high probability,
%
%
\begin{equation} \label{lambdaN}
\mu_{\max} = f + \frac{1}{f} + O \biggl({\frac{1}{f^3} +
\frac{1}{f^2 q} + \frac{(\log N)^\xi}{\sqrt{N}}}\biggr).
\end{equation}
Note that (\ref{ElambdaN}) and (\ref{lambdaN}) locate $\mu_{\max}$ more
precisely than (\ref{positionofmumax})
in the large-$f$ regime.

Moreover, the phase of $\f v_{\max}$ can be chosen so that we have
with $(\xi,\nu)$-high probability
%
%
\begin{equation} \label{uNe}
\langle{\f v_{\max}}, {\f e}\rangle= 1 - \frac{1}{2f^2} +
O
\biggl({\frac{1}{f^3} + \frac{(\log N)^{2 \xi}}{f \sqrt{N}}}\biggr).
\end{equation}

Finally, there is a constant $C_0$ such that if
%
%
\begin{equation} \label{largerlowerboundonf}
f \geq C_0 (\log N)^{2 \xi} \quad\mbox{and}\quad \xi\geq2,\vadjust{\goodbreak}
\end{equation}
then we have with $(\xi,\nu)$-high probability
%
%
\begin{equation} \label{CLTforlambdamax}
\mu_{\max} = \E\mu_{\max} + \frac{1}{N}\sum_{i,j} h_{ij}
+ O\biggl({\frac{(\log N)^{2 \xi}}{f \sqrt{N}}}\biggr).
\end{equation}
In particular, if (\ref{largerlowerboundonf}) holds, we have (by
the central limit theorem)
%
%
\begin{equation} \label{simpleCLT}
\sqrt{\frac{N}{2}} (\mu_{\max} - \E\mu_{\max})
\longrightarrow\cal N(0,1)
\end{equation}
in distribution, where $\cal N(0,1)$ denotes a standard normal random variable.
\end{theorem}

%
\begin{remark} \label{rem64}
In analogy to Definition~\ref{definitionofminor}, we define $A^{(\bb
T)}$ as the $(N - |\bb T |) \times(N -
|\bb T |)$ minor of $A$ obtained by removing all columns of $A$
indexed by $i \in\bb T$; here $\bb T \subset\{1,\ldots, N\}$. If $A$
satisfies Definition~\ref{definitionofA}, then
so does $({N / (N - |\bb T |)})^{1/2} A^{(\bb
T)}$. Therefore, all results of this section also hold for $A^{(\bb
T)}$ provided $|\bb T |\leq10$. (Here $10$ can
be any fixed number.)
Throughout Sections~\ref{sectionlambdamax} and~\ref{sectwtG} we
abbreviate $\mu^{(\bb T)}_{\max} \deq\mu^{(\bb
T)}_{N - |\bb T |}$ and $\f v^{(\bb T)}_{\max} \deq\f
v^{(\bb
T)}_{N - |\bb T |}$.
\end{remark}
%
%
\begin{remark}
Statistical properties of the $k$ largest eigenvalues of
a random Wigner matrix with a large rank-$k$ perturbation
have been studied in~\cite{PecheFeral,Capitaine,BGRao,SoshPert}.
Theorem~\ref{theoremlargesteigenvalue} collects analogous results
for the more singular case of sparse matrices. We
restrict our attention to the special case where the perturbation is $f
| \f e \rangle\langle\f e |$. (Note that in
\cite{Capitaine,BGRao,SoshPert} the authors allow quite general
finite-rank perturbations of Wigner matrices.)
\end{remark}

The rest of this section is devoted to the proof of Theorem \ref
{theoremlargesteigenvalue}. It is based on the
following standard observation. Let $\mu$ be an eigenvalue of $A$ with
associated normalized eigenvector $\f v$. This means that
\[
(\mu- H) \f v = f \langle{\f e}, {\f v}\rangle\f e.
\]
Suppose now that $\mu$ is not an eigenvalue of $H$. Thus, we can
choose $\f v$ and $K > 0$ such that
%
%
\begin{eqnarray}
\label{uKlambda}
\f v & = & K (\mu- H)^{-1} \f e,
\\
\label{1fe}
1 & = & f \langle{\f e}, {(\mu- H)^{-1} \f e}\rangle.
\end{eqnarray}
Using the spectral decomposition of $H$, we rewrite (\ref{1fe}) as
%
%
\begin{equation} \label{f-1mu}
\frac{1}{f} = \sum_\alpha\frac{|\langle{\f e}, {\f
u_\alpha}\rangle|^2}{\mu- \lambda_\alpha}.
\end{equation}
It is easy to see that (\ref{f-1mu}) has a unique solution, $\mu
_{\max}$, greater than $\lambda_N$. Moreover,
(\ref{f-1mu}) readily yields $\mu_{\max} - \lambda_N \leq f \leq
\mu_{\max} - \lambda_1$, that is,
%
%
\begin{equation} \label{domainofmumax}
\mu_{\max} \in[f + \lambda_1, f + \lambda_N].
\end{equation}
Our proof is based on the series expansions
%
%
\begin{eqnarray} \label{lambdaNsum}
\mu_{\max} & = & f \sum_{k \geq0} \langle{\f e}, {(H /
\mu_{\max})^k \f e}\rangle,
\\[-2pt]
\label{uNsum}
\f v_{\max} & = & K \sum_{k \geq0} (H / \mu_{\max})^k \f e.
\end{eqnarray}
Note that the expansions (\ref{lambdaNsum}) and (\ref{uNsum}) can
be interpreted as
perturbative corrections around the matrix $f | \f e \rangle\langle
\f e |$.

In order to control the expansions (\ref{lambdaNsum}) and (\ref
{uNsum}), we shall need the
following large deviation bound, proved in the \hyperref[app]{Appendix}.
%
%
\begin{lemma} \label{lemmae_he}
Let $1 \leq k \leq\log N$. Then
%
%
\begin{equation} \label{eHe}
|\langle{\f e}, {H^k \f e}\rangle- \E\langle{\f e},
{H^k \f e}\rangle|\leq
C\frac{(\log N)^{k \xi}}{N^{1/2}}
\end{equation}
with $(\xi,\nu)$-high probability provided that $1 \leq q \leq C N^{1/2}$.
\end{lemma}
\begin{pf*}{Proof of (\ref{positionofmumax})}
The key observation is that
%
%
\begin{equation} \label{analysisoffintemoments}
\biggl|\E\langle{\f e}, {H^k \f e}\rangle- \int x^k
\varrho_{\mathrm{sc}}(x)
\,\dd x \biggr|\leq\frac{C(k)}{q}
\end{equation}
for some constant $C(k)$ depending on $k$.
Indeed, a standard application of the moment method (see, e.g., \cite
{guibook}, Section 1.2) shows that $\E\langle{\f e}, {H^{2n} \f
e}\rangle= C_n + O_n(q^{-2})$, where $C_n \deq\frac{1}{n +
1}{2n\choose n} = \int x^{2n} \varrho_{\mathrm{sc}}(x) \,\dd
x$ is the $n$th Catalan number. If $k$ is odd, one finds by a similar
moment estimate that $\E\langle{\f e}, {H^k \f e}\rangle
= O_k(q^{-1})$. We omit the details.

For the following we work on the event of $(\xi, \nu)$-high
probability on which (\ref{boundonHtildeweak}) holds. We
consider solutions $\mu$ of (\ref{1fe}) in the interval $I \deq[2 +
\epsilon_0^2/20, \infty)$. By monotonicity of the
right-hand side of (\ref{f-1mu}) in $I$, we know that (\ref{1fe})
has at most one solution in $I$. For any $k_0 \in
\N$, using (\ref{eHe}) and (\ref{analysisoffintemoments}) we may
expand (\ref{1fe}) in $I$ [see (\ref{lambdaNsum})]
as
\begin{eqnarray*}
\mu& = &f \sum_{k = 0}^{k_0} \int\biggl({\frac{x}{\mu}}
\biggr)^k \varrho_{\mathrm{sc}}(x) \,\dd x\\[-2pt]
&&{} + O \biggl({f \sum_{k > k_0}
\biggl({\frac{\| H \|}{\mu}}\biggr)^k + \frac{C(k_0)}{q} + \frac
{k_0 (\log N)^{\xi k_0}}{\sqrt{N}}}\biggr)
\\[-2pt]
& = & - f \mu m_{\mathrm{sc}}(\mu)\\[-2pt]
&&{} + O \biggl(f \sum_{k > k_0}
\biggl({\frac{\| H \|}{2 + \epsilon_0^2/20}}\biggr)^k + f \sum_{k
> k_0} \biggl({\frac{2}{2 + \epsilon_0^2 / 20}}\biggr)^k\\[-2pt]
&&\hspace*{116pt}{} + \frac
{C(k_0)}{q} + \frac{k_0 (\log N)^{\xi k_0}}{\sqrt{N}}\biggr),\vadjust{\goodbreak}
\end{eqnarray*}
where the first term comes from extending the sum over $k$ to infinity
and using that
\[
\sum_{k = 0}^\infty\int\biggl({\frac{x}{\mu}}\biggr)^k
\varrho_{\mathrm{sc}}(x) \,\dd x = \mu\int\frac{\varrho_{\mathrm{sc}}(x) \,\dd
x}{\mu- x}
\]
for $\mu> 2$.
It is easy to see that the second term is $o(1)$ by an appropriate
choice of $k_0(N)$. Thus, we have proved that, for
$\mu\in I$, the equation (\ref{1fe}) reads
$m_{\mathrm{sc}}(\mu) = -f^{-1} + r(\mu)$, where $r(\mu) \to0$ as $N \to
\infty$ uniformly in $\mu$.

Next, the function $\mu\mapsto m_{\mathrm{sc}}(\mu)$ is continuous and
monotone increasing on $(2, \infty)$, with range
$(-1,0)$. Let $\bar\mu$ be the unique solution of $m_{\mathrm{sc}}(\bar\mu)
= - f^{-1}$. (Note that here we need the assumption
$f > 1$.) Using (\ref{identityofmsc}), we find that $\bar\mu= f +
f^{-1} \geq2 + \epsilon_0^2 / 10$. We
therefore find that, for $N$ large enough, the equation $m_{\mathrm{sc}}(\mu) =
- f^{-1} + r(\mu)$ [which is equivalent to
(\ref{1fe}) on $I$] has a unique solution $\mu\in I$ which satisfies
$\mu= \bar\mu+ o(1)$. Since $\mu$ is the only
solution of (\ref{1fe}) in $I$, we must have $\mu= \mu_{\max}$.
\end{pf*}

Note that (\ref{positionofmumax}) remains valid if $\f e$ in (\ref
{A=H+fee}) is replaced with any
$\ell^2$-normalized vector. It is a simple matter to check that (\ref
{analysisoffintemoments}) is valid for
arbitrary vectors $\f e$. Moreover, Lemma~\ref{lemmae_he} remains
correct for arbitrary $\f e$ provided one replaces
$N^{-1/2}$ on the right-hand side of (\ref{eHe}) with $q^{-1}$. We
omit the details, as we shall not need this result.

From (\ref{gap}) and (\ref{boundonHtildeweak}) we find that with
$(\xi,\nu)$-high probability
%
%
\begin{equation} \label{conditionforgeometricseries}
\frac{\| H \|}{\mu_{\max}} \leq1 - c
\end{equation}
for some constant $c > 0$. In particular, (\ref{lambdaNsum}) and
(\ref{uNsum}) converge
with $(\xi,\nu)$-high probability.
\begin{pf*}{Proof of (\ref{lambdaN})}
From (\ref{lambdaNsum}) and Lemma~\ref{lemmaboundon_h_tildeweak} we find
$\mu_{\max} = f(1 + r(f))$ with $(\xi,\nu)$-high probability, where
$\lim
_{f \to\infty} r(f)
= 0$. Together with the simple identities
%
%
\begin{equation} \label{expeHe}
\E\langle{\f e}, {H \f e}\rangle= 0,\qquad \E\langle{\f
e}, {H^2 \f e}\rangle
= 1,
\end{equation}
(\ref{lambdaNsum}), (\ref{conditionforgeometricseries}) and
Lemma~\ref{lemmae_he} yield with $(\xi,\nu)$-high probability
%
%
\begin{equation}
\mu_{\max} = f + \frac{1}{f} + \frac{\E\langle{\f e}, {H^3
\f e}\rangle}{f^2} + O\biggl({\frac{1}{f^3}}\biggr) +
O\biggl({\frac{(\log N)^\xi}{\sqrt{N}}}\biggr).
\end{equation}
By explicit computation we find that
%
%
\begin{equation} \label{thirdorderterm}
\E\langle{\f e}, {H^3 \f e}\rangle= O(q^{-1}).
\end{equation}
Thus, (\ref{lambdaN}) follows.
\end{pf*}
\begin{pf*}{Proof of (\ref{ElambdaN})}
From (\ref{lambdaNsum}) and (\ref{conditionforgeometricseries})
we get
with $(\xi,\nu)$-high probability
%
%
\begin{equation} \label{self-consistenteqforlambdamax}
\mu_{\max} = f + \frac{f}{\mu_{\max}} \langle{\f e}, {H
\f e}\rangle+ \frac{f}{\mu_{\max}^2}
\langle{\f e}, {H^2 \f e}\rangle+ O\biggl({\frac{1}{f^3} + \frac{1}{f^2
q}}\biggr),
\end{equation}
where we used (\ref{thirdorderterm}). Iterating (\ref
{self-consistenteqforlambdamax}) yields with $(\xi,\nu)$-high probability
%
%
\begin{eqnarray} \label{expansionforElambdaN}\hspace*{32pt}
\mu_{\max} &=& f + \langle{\f e}, {H \f e}\rangle- \langle
{\f e}, {H \f e}\rangle^2 / f + \langle{\f e}, {H \f e}\rangle^3 /
f^2 - \langle{\f e}, {H \f e}\rangle\langle{\f e}, {H^2 \f
e}\rangle/ f^2
\nonumber\\[-8pt]\\[-8pt]
&&{}+ \langle{\f e}, {H^2 \f e}\rangle/ f - 2 \langle{\f e}, {H \f
e}\rangle\langle{\f e}, {H^2 \f e}\rangle/ f^2 + O
\biggl({\frac{1}{f^3} + \frac{1}{f^2 q}}\biggr),\nonumber
\end{eqnarray}
where we used Lemma~\ref{lemmaboundon_h_tildeweak}. In order to
complete the proof of (\ref{ElambdaN}), we use the
rough estimate $\E\mu^2_{\max} \leq\E\tr A^2 \leq C N f^2 + N
\leq N^C$, by
(\ref{upperboundonf}). Recalling (\ref{eHe}), we also get
\[
| \E\langle{\f e}, {H \f e}\rangle^2 | \leq\frac{C}{N},\qquad |
\E\langle{\f e}, {H \f e}\rangle^3 | \leq\frac{C}{Nq},\qquad |\E\langle{\f
e}, {H \f e}\rangle\langle{\f e}, {H^2 \f e}\rangle| \leq\frac{C}{Nq}
\]
by explicit calculation using (\ref{momentconditions}). Now taking
the expectation in (\ref{expansionforElambdaN}),
using (\ref{expeHe}) yields (\ref{ElambdaN}).
\end{pf*}
\begin{pf*}{Proof of (\ref{uNe})}
We compute the normalization constant $K$ in (\ref{uNsum}) from
%
%
\begin{eqnarray} \label{computationofK}
K^{-2} &=& \sum_{k,k' \geq0} \mu_{\max}^{-k - k'} \langle{\f
e}, {H^{k+k'} \f e}\rangle\nonumber\\
&=& 1 + 2 \mu_{\max}^{-1} \langle{\f e}, {H \f
e}\rangle+ 3 \mu_{\max}^{-2}
\langle{\f e}, {H^2 \f e}\rangle+ O(\mu_{\max}^{-3})
\\
&=& 1 + \frac{3}{f^2} + O\biggl({\frac{1}{f^3} + \frac{(\log N)^{2
\xi}}{f \sqrt{N}}}\biggr)\nonumber
\end{eqnarray}
with $(\xi,\nu)$-high probability, where we used Lemmas~\ref{lemmae_he}
and \ref
{lemmaboundon_h_tildeweak}, as well as (\ref{lambdaN})
and (\ref{conditionforgeometricseries}). Now (\ref{uNe}) is an
easy consequence of (\ref{uNsum}),
(\ref{lambdaN}) and Lemmas~\ref{lemmae_he} and~\ref{lemmaboundon_h_tildeweak}.
\end{pf*}

What remains is to prove (\ref{CLTforlambdamax}).
\begin{pf*}{Proof of (\ref{CLTforlambdamax})}
We assume (\ref{largerlowerboundonf}), and, in particular, $\mu
_{\max} \geq\break\frac{C_0}{2} (\log N)^{ 2 \xi}$ by
(\ref{positionofmumax}). Thus,
from Lemma~\ref{lemmaboundon_h_tildeweak} we get with $(\xi,\nu
)$-high probability
\[
\frac{\| H \|}{\mu_{\max}} \leq\frac{6}{C_0 (\log N)^{2
\xi}}.
\]
From (\ref{lambdaNsum}) and (\ref{lambdaN}) we therefore get with
$(\xi,\nu)$-high probability
\[
\mu_{\max} = f \sum_{k = 0}^{c_0 \log N} \frac{\langle{\f
e}, {H^k \f e}\rangle}{\mu_{\max}^k} +
O ({\me^{- c_0 \log N \log\log N}}),\vadjust{\goodbreak}
\]
where $c_0 \leq1$ is a positive constant to be chosen later.
Thus, we find with $(\xi,\nu)$-high probability
\begin{eqnarray*}
\mu_{\max} &=& f \sum_{k = 0}^{c_0 \log N} \frac{\E\langle
{\f e}, {H^k \f e}\rangle}{\mu_{\max}^k} +
\frac{f}{\mu_{\max}} \langle{\f e}, {H \f e}\rangle+ f \sum
_{k = 2}^{c_0
\log N} \frac{\langle{\f e}, {H^k \f e}\rangle- \E
\langle{\f e}, {H^k \f e}\rangle}{\mu_{\max}^k}\\
&&{} + O ({\me
^{- c_0
\log N \log\log N}}).
\end{eqnarray*}
Therefore, we get, for any $0 < c_0 \leq1$ and using (\ref{lambdaN})
and Lemma~\ref{lemmae_he}, that with $(\xi,\nu)$-high probability we
have
\[
\mu_{\max} = f \sum_{k = 0}^{c_0 \log N} \frac{\E\langle
{\f e}, {H^k \f e}\rangle}{\mu_{\max}^k} +
\frac{f}{\mu_{\max}} \langle{\f e}, {H \f e}\rangle+ O
\biggl({\frac
{(\log N)^{2 \xi}}{f \sqrt{N}}}\biggr).
\]
Here the constant in $O(\cdot)$ depends on $c_0$.

Next, Lemma~\ref{lemmaboundon_h_tildeweak} yields
%
%
\begin{equation} \label{boundonEeHe}
|\E\langle{\f e}, {H^k \f e}\rangle|\leq
(5/2)^k
+ N^{Ck} \me^{-\nu(\log N)^\xi} \leq3^k
\end{equation}
for $k \leq(\nu/C) (\log N)^{\xi- 1}$.
[Here we used Schwarz's inequality and the trivial estimate
$\E\langle{\f e}, {H \f e}\rangle\leq N^C$ to estimate the
contribution of
the low-probability event on which (\ref{boundonHtildeweak}) does
not hold.]
By the assumption (\ref{largerlowerboundonf}) on $\xi$, (\ref
{boundonEeHe}) holds for $k \leq c_0 \log N$ for
$c_0$ small enough. It is therefore easy to see that the
equation
\[
\bar\mu= f \sum_{k = 0}^{c_0 \log N} \frac{\E\langle{\f e},
{H^k \f e}\rangle}{\bar\mu^k}
\]
has a unique solution $\bar\mu> 0$, which satisfies $\bar\mu= f +
O(f^{-1})$.
Writing $\mu_{\max} = \bar\mu+ \zeta$, we get with $(\xi,\nu
)$-high probability
%
%
\begin{eqnarray} \label{zetaNsum}
\zeta&=& \frac{f}{\mu_{\max}} \langle{\f e}, {H \f e}\rangle
+ f \sum
_{k = 0}^{c_0 \log N} \frac{\E\langle{\f e}, {H^k \f e}\rangle
}{\bar\mu^k} \biggl[{\biggl({1 + \frac{\zeta}{\bar\mu}}
\biggr)^{-k} - 1}\biggr]\nonumber\\[-8pt]\\[-8pt]
&&{} + O \biggl({\frac{(\log N)^{2 \xi}}{f \sqrt
{N}}}\biggr).\nonumber
\end{eqnarray}

Next, by (\ref{lambdaN}) we find $\zeta= O(f^{-1})$ with $(\xi,\nu
)$-high probability. Moreover, (\ref{expeHe}) and Lemma~\ref{lemmae_he} imply
that $\langle{\f e}, {H \f e}\rangle= O((\log N)^\xi N^{-1/2})$
with $(\xi
,\nu)$-high probability, and that the sum in (\ref{zetaNsum})
starts at $k =
2$. This yields the expression with $(\xi,\nu)$-high probability
\[
\zeta= \langle{\f e}, {H \f e}\rangle+ \frac{1}{f} \sum_{l
\geq1} a_l
\zeta^l + O \biggl({\frac{(\log N)^{2 \xi}}{f \sqrt{N}}}\biggr)
\]
for some coefficients $a_l = O(1)$, by (\ref{boundonEeHe}). We
conclude that with $(\xi,\nu)$-high probability we have
\[
\zeta= \langle{\f e}, {H \f e}\rangle\bigl({1 +
O(f^{-1})}\bigr) + O
\biggl({\frac{(\log N)^{2 \xi}}{f \sqrt{N}}}\biggr) = \langle
{\f e}, {H \f e}\rangle+ O \biggl({\frac{(\log N)^{2 \xi}}{f
\sqrt{N}}}\biggr),
\]
where we used that $|\langle{\f e}, {H \f e}\rangle|\leq
(\log N)^\xi
N^{-1/2}$ with $(\xi,\nu)$-high probability.

Summarizing, we have proved that with $(\xi,\nu)$-high probability we have
%
%
\begin{equation} \label{lambdamaxollambda}
\mu_{\max} = \bar\mu+ \frac{1}{N} \sum_{i,j} h_{ij} + R,
\end{equation}
where $| R |\leq O ({\frac{(\log N)^{2 \xi}}{f \sqrt
{N}}})$.
Using $\E\mu_{\max}^2 \leq N^C$, we therefore get
\[
\E| R |\leq O \biggl({\frac{(\log N)^{2 \xi}}{f \sqrt
{N}}}\biggr),
\]
and (\ref{CLTforlambdamax}) follows by taking the expectation in
(\ref{lambdamaxollambda}).
\end{pf*}

This concludes the proof of Theorem~\ref{theoremlargesteigenvalue}.

For future reference,\vspace*{2pt} we record two simple corollaries which we shall
use in Section~\ref{sectwtG} to control the
matrix elements of $\wt G$.
%
%
\begin{corollary}
Suppose that $A$ satisfies Definition~\ref{definitionofA}. Then we
have with $(\xi,\nu)$-high probability
%
%
\begin{equation} \label{boundonbulkeigenvalues}
|\mu_\alpha|\leq\max_\beta|\lambda_\beta
|=
\| H \|\leq2 + (\log N)^\xi q^{-1/2} \qquad\mbox{for
}
\alpha= 1,\ldots, N-1.\hspace*{-32pt}
\end{equation}
\end{corollary}
\begin{pf}
Use (\ref{interlacing}) and Lemma~\ref{lemmaboundon_h_tildeweak}.
\end{pf}
%
%
\begin{corollary} \label{corollary67}
Suppose that $A$ satisfies Definition~\ref{definitionofA} and that,
in addition, $f \leq C_0 N^{1/2}$.
Then we have with $(\xi,\nu)$-high probability
%
%
\begin{equation} \label{ualphae}
\sum_{\alpha\neq N} |\langle{\f v_\alpha}, {\f e}\rangle
|^2 =
O(f^{-2}).
\end{equation}
\end{corollary}
\begin{pf}
The statement is trivial unless $f \geq1 + \epsilon_0$, in which case
we use (\ref{uNe}) and $\| {\f e}\|=1$.
\end{pf}

\section{\texorpdfstring{Control of $\nc G$: Proofs of Theorems \protect\ref{LSCTHMA}
and \protect\ref{theoremdelocalization}}
{Control of G: Proofs of Theorems 2.9 and 2.16}} \label{sectwtG}

In this section we adopt the convention that if $F = F(H)$ is any
function of $H$, then $F(A)$ is denoted by $\wt F$, that is, we use
the tilde
$\wt{(\cdot)}$ to indicate quantities defined in terms of $A = H + f
| \f e \rangle\langle\f e |$. Thus, for example, we have
\begin{eqnarray*}
A &=& \nc H,\qquad \mu_\alpha= \nc\lambda_\alpha,\qquad
\f v_\alpha= \nc{\f u}_\alpha,\\
\nc G(z) :\!&=&
(A - z)^{-1},\qquad \nc m(z) \deq\frac{1}{N} \sum\nc
G_{ii}(z)
\end{eqnarray*}
and
%
%
\begin{eqnarray}
\nc\Lambda_o &\deq& \max_{i \neq j} |\nc G_{ij} |,\qquad
\nc\Lambda_d \deq\max_i |\nc G_{ii} - m_{\mathrm{sc}}
|,\nonumber\\[-8pt]\\[-8pt]
\nc\Lambda&\deq&|\nc m - m_{\mathrm{sc}} |,
\qquad
\nc v_i \deq \nc G_{ii} - m_{\mathrm{sc}}.\nonumber
\end{eqnarray}
Note that $\nc G$ and $\nc m$ were already introduced in (\ref{tildeg}).

We begin by using the interlacing property (\ref{interlacing}) to
derive a bound on~$\Lambda$. Recall the convention
that if $F = F(H)$ is any function of $H$, then $\nc F$ is defined as
$F(A)$, where, we recall, $A = H + f | \f e \rangle\langle\f e |$.
%
%
\begin{lemma} \label{lemmaLambdaLambdaA}
Let $A$ satisfy Definition~\ref{definitionofA}.
Then for any $z \in D_L$ we have
\[
|\nc\Lambda(z) - \Lambda(z) |\leq\frac{\pi}{N
\eta}.
\]
\end{lemma}
\begin{pf}
Define the empirical density $\nc\varrho(x) \deq\frac{1}{N} \sum
_\alpha\delta(x - \mu_\alpha)$. Thus, the integrated
empirical density defined in (\ref{empiricalintegrateddensity}) can
be written as $\nc{\fra n}(E) =
\int_{-\infty}^E \nc\varrho(x) \,\dd x$. Similarly, define the
quantities $\varrho$ and $\fra n$ in terms of the
eigenvalues $\lambda_1,\ldots, \lambda_N$ of~$H$.
Using integration by parts, we find
\[
\nc\Lambda(z) - \Lambda(z) = \int\frac{ \nc\varrho(x) -
\varrho(x)}{x - z} \,\dd x = - \int\frac{\nc
{\fra n}(x) - \fra n(x)}{(x - z)^2} \,\dd x.
\]
By (\ref{interlacing}) we have $|\nc{\fra n}(x) - \fra n(x)
|
\leq N^{-1}$ for all $x$. Thus, we find
\[
|\nc{\Lambda}(z) - \Lambda(z) |\leq\frac{1}{N}
\int
\frac{1}{| x - z |^2} \,\dd x = \frac{\pi}{N
\eta}.
\]
\upqed
\end{pf}

We note that the claim (\ref{scmA}) of Theorem~\ref{LSCTHMA} is now
an immediate consequence of Lemma~\ref{lemmaLambdaLambdaA} and the
strong local semicircle law (\ref{scm}) for $H$.

The rest of this section is devoted to the proof of the estimate (\ref
{GijestimateA}) for the matrix elements of $\nc
G$. From now on we consistently assume the upper bound (\ref
{fupperbound}) on $f$.

\subsection{Basic estimates on the good events} \label{sectgoodeventforwtG}
In this section we control the individual matrix elements $\nc G_{ij}$
in terms of $\nc\Lambda$,
which in turn will be estimated using Lemma~\ref{lemmaLambdaLambdaA}.
Our basic strategy is similar to that of Section~\ref{sectionWLSC},
but, owing
to the nonvanishing expectation of $a_{ij}$, the self-consistent equation
for $\nc G_{ii}$ has several additional error terms as compared to Lemma
\ref{lemmaself-consistentequation}; see Lemma~\ref{self-consistenteqwt}
and Proposition~\ref{propositionoff-diagestimate} below. The most dangerous
of these error terms is estimated in Lemma~\ref{lemmaestimateoffGh} below.
We will use the spectral decomposition of $\nc H$, combined with bounds on
$\langle{\f e}, {\f v_\alpha}\rangle$ and $\|\f v_\alpha\|
_\infty$. The
former quantities are estimated
using Corollary~\ref{corollary67}, while the latter are estimated by
bootstrapping.
The spectral decomposition requires simultaneous control of all
eigenvectors,\vadjust{\goodbreak} whose associated eigenvalues are distributed throughout
the spectrum. Since bounds on $\|\f v_\alpha\|_\infty$
(delocalization bounds)
may be derived from a priori bounds on $\nc\Lambda_d(z)$ for $\re z$
being near the
corresponding eigenvalue, we will therefore need bounds on $\nc\Lambda_d(z)$
that are uniform for all $z \in D_L$ with a fixed imaginary part.
Hence, the bootstrapping now occurs simultaneously for all $E \in
[-\Sigma, \Sigma]$ (see Definition~\ref{defDeta} below).

We use the following self-consistent equation for $\nc G$, whose proof
is an elementary calculation using
(\ref{GijGijk}) and (\ref{Gijformula}) applied to $\nc G$; see also
Lemma~\ref{lemmaself-consistentequation}.

%
\begin{lemma} \label{self-consistenteqwt}
We have the identity
%
%
\begin{equation} \label{self-consistentA}
\nc G_{ii} = \frac{1}{-z - m_{\mathrm{sc}} - ([\nc v] - \nc\Upsilon_i)},
\end{equation}
where
\[
\nc\Upsilon_i \deq h_{ii} - \nc Z_i + \nc{\cal A}_i
\]
and
%
%
\begin{eqnarray}\label{AdefA}
\nc Z_i &\deq&\IE_i \sum_{k,l}^{(i)} a_{ik} \nc G^{(i)}_{kl}
a_{li},\nonumber\\[-8pt]\\[-8pt]
\nc{\cal A}_i &\deq&\frac{f}{N} - \frac{f^2}{N} \frac{N - 1}{N}
\bigl\langle{\f e}, {\nc G^{(i)} \f e}\bigr\rangle+ \frac{1}{N}
\sum_j \frac{\nc G_{ij} \nc G_{ji}}{\nc G_{ii}}.\nonumber
\end{eqnarray}
\end{lemma}

Recall that in expressions such as (\ref{AdefA}) the vector $\f e$
stands for $\f e_{N - 1}$; see~(\ref{defofbe}).

%
\begin{definition} \label{defDeta}
For $N^{-1} (\log N)^L \leq\eta\leq3$ introduce the set $D(\eta)
\deq\{z \in D_L \st\im z = \eta\}$. We define the
event
%
%
\begin{equation}\label{Omgdef}
\nc\Omega(\eta) \deq\Bigl\{{\sup_{z\in D(\eta)}
\bigl({\nc\Lambda_d(z) + \nc\Lambda_o(z)}\bigr) \leq(\log
N)^{-\xi}}\Bigr\}.
\end{equation}
\end{definition}

Recall the definition of $A^{(\bb T)}$ from Remark~\ref{rem64}.
Similarly to Lemma~\ref{lemmaiTGT}, we have the following result for
the matrix $A$.
%
%
\begin{lemma} \label{lemmaiTGTA}
Fixing $z=E+i\eta\in D_L$, we have for any $i$ and $\bb T \subset\{
1,\ldots, N\}$ satisfying
$i \notin\bb T$ and $|\bb T |\leq10$ that
%
%
\begin{equation} \label{iTGTA}
\nc m^{(i \bb T)}(z) = \nc m^{(\bb T)}(z) + O \biggl({\frac{1}{N
\eta}}\biggr)
\end{equation}
holds in $\nc\Omega(\eta)$.
\end{lemma}

The following lemma is crucial in dealing with error terms arising from
the nonvanishing expectation of $a_{ij}$.
Recall that, when indexing the eigenvalues and eigenvectors of $A^{(\bb
T)}$, we defined $\alpha_{\max} \deq N -
|\bb T |$.\vadjust{\goodbreak}
%
%
\begin{lemma} \label{lemmaestimateoffGh}
Fixing $z=E+\ii\eta\in D_L$, we have for any $\bb T \subset\{1,\ldots,
N\}$ satisfying
$|\bb T |\leq10$ and for any $i\in\bb T$ that
%
%
\begin{equation} \label{estimateoffGh}
\Biggl|\sum_{k,l}^{(\bb T)} \frac{f}{N} \nc G^{(\bb T)}_{kl}
h_{li} \Biggr|
\leq
C (\log N)^\xi\Biggl({\frac{1}{q} + \sqrt{\frac{\im\nc m}{N \eta
}} + \frac{1}{N \eta}}\Biggr)
\end{equation}
on $\nc\Omega(\eta)$ with $(\xi,\nu)$-high probability.
\end{lemma}
\begin{pf}
For technical reasons, it is convenient to avoid the situation where
$\mu_{\max}$ is close to $\Sigma$. In order to
ensure this, we may if necessary increase $\Sigma$ slightly and hence
assume that $f \leq\Sigma- 3$ or $f \geq
\Sigma+ 3$.
We start by proving the following delocalization bound. Define
%
%
\begin{equation} \label{delocbounddefinition}\quad
R \deq\max_{|\bb T |\leq10} \max_{\alpha\neq
\alpha
_{\max}} \max_j \bigl| v_\alpha^{(\bb T)}(j) \bigr|,\qquad
R_{\max} \deq\max_{|\bb T |\leq10} \max_j \bigl|
v_{\max}^{(\bb T)}(j) \bigr|.
\end{equation}
First we claim that on $\nc\Omega( \eta)$ we have with $(\xi,\nu
)$-high probability
%
%
\begin{equation} \label{delocbound}
R \leq C \sqrt{\eta}
\end{equation}
and, assuming $f \leq\Sigma- 3$, we have with $(\xi,\nu)$-high probability
%
%
\begin{equation} \label{delocboundmax}
R_{\max} \leq C \sqrt{\eta}.
\end{equation}

In order to prove (\ref{delocbound}) and (\ref{delocboundmax}), we
note that on $\nc\Omega(\eta)$ we have, in
analogy to (\ref{GiionOmega}),
%
%
\begin{equation} \label{GiionOmegaA}
c \leq\bigl|\nc G_{jj}^{(\bb T)}( z) \bigr|\leq C
\end{equation}
for all $ z \in D_L$ such that $\im z=\eta$ and $N$ large enough.
From (\ref{boundonbulkeigenvalues}) we find that
$ z \deq\mu_\alpha^{(\bb T)} + \ii\eta\in D_L$ with $(\xi,\nu
)$-high probability
for $\alpha\neq\alpha_{\max}$; see Remark~\ref{rem64}.
Thus, we get with $(\xi,\nu)$-high probability
\[
C \geq\im\nc G^{(\bb T)}_{jj}\bigl(\mu^{(\bb T)}_\alpha+ \ii\eta
\bigr) = \sum_\beta\frac{\eta| v^{(\bb T)}_\beta(j) |
^2}{(\mu^{(\bb T)}_\beta- \mu^{(\bb T)}_\alpha)^2 +
\eta^2} \geq\frac{| v^{(\bb T)}_\alpha(j) |^2}{\eta
}.
\]
This concludes the proof of (\ref{delocbound}). Next, if $f \leq
\Sigma- 3$, then by (\ref{positionofmumax}) and
Lem\-ma~\ref{lemmaboundon_h_tildeweak} we have $\mu_{\max}^{(\bb T)}
\in[- \Sigma, \Sigma]$ with $(\xi,\nu)$-high probability. Thus, we get
(\ref{delocboundmax}) just like above.

Having established (\ref{delocbound}) and (\ref{delocboundmax}), we may
now estimate the left-hand side of~(\ref{estimateoffGh}), using the
spectral decomposition of $ \nc G^{(\bb T)}$, by
%
%
\begin{eqnarray} \label{max+restsplit}
&&\frac{f}{\sqrt{N - |\bb T |}} \Biggl|\frac{\langle
{\f e}, {\f v^{(\bb T)}_{\max}}\rangle}{\mu_{\max}^{(\bb T)}
- z} \sum_l^{(\bb T)} v^{(\bb T)}_{\max}(l) h_{li}
\Biggr|\nonumber\\[-8pt]\\[-8pt]
&&\qquad{}+
\frac{f}{\sqrt{N - |\bb T |}} \Biggl|\sum_{\alpha
\neq\alpha_{\max}} \frac{\langle{\f e}, {\f v^{(\bb
T)}_\alpha}\rangle}{\mu_\alpha^{(\bb T)} - z} \sum_l^{(\bb T)}
v^{(\bb T)}_\alpha(l) h_{li} \Biggr|.\nonumber
\end{eqnarray}
By the delocalization bound (\ref{delocbound}) and the large
deviation estimate (\ref{aA}), we find for $\alpha\neq
\alpha_{\max}$ on $\nc\Omega(\eta)$ with $(\xi,\nu)$-high probability
\[
\Biggl|\sum_l^{(\bb T)} v_\alpha^{(\bb T)}(l) h_{li}
\Biggr|\leq(\log N)^\xi\biggl({\frac{R}{q} + \frac{1}{\sqrt
{N}}}\biggr).
\]
Similarly, we have
\[
\Biggl|\sum_l^{(\bb T)} v_{\max}^{(\bb T)}(l) h_{li}
\Biggr|\leq(\log N)^\xi\biggl({\frac{R_{\max}}{q} + \frac
{1}{\sqrt{N}}}\biggr).
\]
Next, we estimate the first term of (\ref{max+restsplit}). If $f
\leq\Sigma- 3$, then $f \leq C$, and the first term
of (\ref{max+restsplit}) is bounded, with $(\xi,\nu)$-high
probability, by
\[
\frac{C}{\sqrt{N} \eta} (\log N)^\xi\biggl({\frac{\sqrt{\eta
}}{q} + \frac{1}{\sqrt{N}}}\biggr) \leq C (\log N)^\xi
\biggl({\frac{1}{q} + \frac{1}{N \eta}}\biggr).
\]
If $f \geq\Sigma+ 3$, then by (\ref{positionofmumax}) and (\ref
{lambdaN}) we get $|\mu^{(\bb T)}_{\max} - z
|
\geq c f$ with $(\xi,\nu)$-high probability. Thus, the first term of
(\ref{max+restsplit}) is bounded with $(\xi,\nu)$-high probability by
\[
\frac{C f}{\sqrt{N}} \frac{(\log N)^\xi}{f} \biggl({\frac{1}{q} +
\frac{1}{\sqrt{N}}}\biggr) \leq\frac{C (\log
N)^\xi}{\sqrt{N} q},
\]
where we used the trivial bound $R_{\max} \leq1$. We therefore get
that the left-hand side of (\ref{estimateoffGh}) is bounded with
$(\xi,\nu)$-high probability by
\begin{eqnarray*}
&&
C (\log N)^\xi\biggl({\frac{1}{q} + \frac{1}{N \eta}}\biggr)+ C
(\log N)^\xi\frac{f}{\sqrt{N}} \biggl({\frac{R}{q} + \frac
{1}{\sqrt{N}}}\biggr) \sum_{\alpha\neq\alpha_{\max}} \frac
{|\langle{\f e}, {\f v^ {(\bb T)}_\alpha}\rangle
|
}{|\mu_\alpha^{(\bb T)} - z |}
\\
&&\qquad\leq
C (\log N)^\xi\biggl({\frac{1}{q} + \frac{1}{N \eta}}\biggr)\\
&&\qquad\quad{} +
C (\log N)^\xi\frac{f}{\sqrt{N}} \biggl({\frac{R}{q} + \frac
{1}{\sqrt{N}}}\biggr) \biggl({\sum_{\alpha\neq\alpha_{\max}}
\bigl|\bigl\langle{\f e}, {\f v^{(\bb T)}_\alpha}\bigr\rangle
\bigr|^2}\biggr)^{1/2} \\
&&\hspace*{11pt}\qquad\quad{}\times\biggl({\sum_{\alpha} \frac{1}{|\mu
_\alpha^{(\bb T)} - z |^2}}\biggr)^{1/2}.
\end{eqnarray*}
By (\ref{ualphae}) this becomes
%
%
\begin{eqnarray} \label{term3ofZij}
&&C (\log N)^\xi\biggl({\frac{1}{q} + \frac{1}{N \eta}}\biggr)
\nonumber\\[-8pt]\\[-8pt]
&&\qquad{}+
C (\log N)^\xi\frac{1}{\sqrt{N}} \biggl({\frac{R}{q} + \frac
{1}{\sqrt{N}}}\biggr) \biggl({\sum_{\alpha} \frac{1}{|\mu
_\alpha^{(\bb T)} - z |^2}}\biggr)^{1/2}.\nonumber
\end{eqnarray}
Using (\ref{iTGTA}), we get
\[
\sum_\alpha\frac{1}{|\mu_\alpha^{(\bb T)} - z |^2} =
\frac
{1}{\eta} \im\sum_\alpha\frac{1}{\mu^{(\bb
T)}_\alpha- z} = \frac{1}{\eta} \im\tr\nc G^{(\bb T)} =
\frac{N}{\eta} \im\nc m +
O\biggl({\frac{1}{\eta^2}}\biggr),
\]
which therefore yields the bound
\[
C (\log N)^\xi\biggl({\frac{1}{q} + \frac{1}{N \eta}}\biggr) +
C (\log N)^\xi\biggl({\frac{\sqrt{\eta}}{\sqrt{N} q} + \frac
{1}{N}}\biggr) \Biggl({\sqrt{\frac{N}{\eta} \im\nc m} + \frac
{1}{\eta}}\Biggr)
\]
on $\nc\Omega(\eta)$ with $(\xi,\nu)$-high probability.
Here we used (\ref{delocbound}). The claim follows.~%
\end{pf}

For the following statements it is convenient to abbreviate
%
%
\begin{equation} \label{defPhi}
\Phi(z) \deq\frac{(\log N)^{\xi}}{q} + (\log N)^{2 \xi}
\Biggl({\sqrt{\frac{\im\nc m(z)}{N \eta}} + \frac{1}{N \eta
}}\Biggr).
\end{equation}

%
\begin{proposition} \label{propositionoff-diagestimate}
Assume (\ref{fupperbound}). Then for $z=E+i\eta\in D_L$ we have
%
%
\begin{eqnarray}
\label{GijaprioriA}
\nc\Lambda_o(z) & \leq & C \Phi(z),
\\
\label{ZiaprioriA}
\max_i |\nc Z_i(z) |
& \leq &
C \Phi(z),
\\
\label{AiaprioriA}
\max_i |\nc{\cal A}_i(z) |& \leq&\frac{C}{N \eta}
\end{eqnarray}
in $\nc\Omega(\eta)$ with $(\xi,\nu)$-high probability.
\end{proposition}
\begin{pf}
We start with (\ref{GijaprioriA}).
Let $i \neq j$. Using (\ref{Gijformula}) for $A$ instead of~$H$, and
writing $a_{ij} = f/N + h_{ij}$, we get
with $(\xi,\nu)$-high probability
%
%
\begin{eqnarray} \label{estimateofGijA}
C^{-1} |\nc G_{ij} |&\leq&\frac{1}{q} + \Biggl|
\sum_{k,l}^{(ij)} h_{ik} \nc G^{(ij)}_{kl} h_{lj} \Biggr|
+ \Biggl|\sum_{k,l}^{(ij)} \frac{f}{N} \nc G^{(ij)}_{kl} h_{lj}
\Biggr|\nonumber\\[-8pt]\\[-8pt]
&&{}
+ \Biggl|\sum_{k,l}^{(ij)} h_{ik} \nc G^{(ij)}_{kl} \frac{f}{N}
\Biggr|+ \Biggl|\sum_{k,l}^{(ij)} \frac{f}{N} \nc G^{(ij)}_{kl} \frac
{f}{N} \Biggr|\nonumber
\end{eqnarray}
by Lemma~\ref{lemmaOmegah} and (\ref{GiionOmegaA}).

The second term of (\ref{estimateofGijA}) is bounded exactly as in
(\ref{estimateofGij}) and (\ref{estimateofGij2}); using (\ref
{abB}) and (\ref{GiionOmegaA}), we estimate it by
\[
(\log N)^{\xi} \frac{C}{q} +C (\log N)^{2 \xi} \Biggl({\sqrt
{\frac{\im\nc m}{N \eta}} + \frac{1}{N \eta}}\Biggr)
\]
on $\nc\Omega(\eta)$ with $(\xi,\nu)$-high probability.\vadjust{\goodbreak}

The last term of (\ref{estimateofGijA}) is bounded with $(\xi,\nu
)$-high probability by
%
%
\begin{eqnarray}\label{f2N2bound}
\frac{f^2}{N^2} (N - 2) \bigl|\bigl\langle{\f e}, {\nc G^{(ij)} \f
e}\bigr\rangle
\bigr|&\leq& C \frac{f^2}{N} \frac{|\langle{\f
e}, {\f v_{\max}^{(ij)}}\rangle|^2}{|\mu
_{\max}^{(ij)} - z |}
+ \frac{f^2}{N} \sum_{\alpha\neq\alpha_{\max}} \frac{
|\langle{\f e}, {\f v_\alpha^{(ij)}}\rangle|
^2}{|\mu_\alpha^{(ij)} - z |}
\nonumber\\
&\leq&
\frac{C f}{N} + \frac{C}{N \eta} + \frac{f^2}{N \eta} \sum
_{\alpha\neq\alpha_{\max}} \bigl|\bigl\langle{\f e}, {\f
v_\alpha^{(ij)}}\bigr\rangle\bigr|^2\\
&\leq&\frac{C f}{N} + \frac{C}{N \eta},\nonumber
\end{eqnarray}
where in the first step we used (\ref{lambdaN}), and in the last step
(\ref{ualphae}). Here we estimated the term
arising from $\mu_{\max}^{(ij)}$ by $C (N \eta)^{-1}$ if $f \leq2
\Sigma$, and by $C f / N$ if $f \geq2 \Sigma$.

Using Lemma~\ref{lemmaestimateoffGh}, the third and fourth terms
in (\ref{estimateofGijA}) are bounded on $\nc
\Omega(\eta)$ with $(\xi,\nu)$-high probability by the right-hand
side of (\ref
{estimateoffGh}). This concludes the proof of (\ref{GijaprioriA}).

Next, we prove (\ref{ZiaprioriA}). By definition,
%
%
\begin{equation} \label{Zsplit}
\nc Z_i = \sum_{k,l}^{(i)} h_{ik} \nc G^{(i)}_{kl} \frac{f}{N} +
\sum_{k,l}^{(i)} \frac{f}{N} \nc G^{(i)}_{kl}
h_{li} + \IE_i \sum_{k,l}^{(i)} h_{ik} \nc G^{(i)}_{kl} h_{li}.
\end{equation}
The first two terms are bounded using Lemma~\ref{lemmaestimateoffGh},
and the last one exactly as (\ref{Ziapriori}).

Finally, we prove (\ref{AiaprioriA}). Using (\ref{lambdaN}),
(\ref{uNe}), (\ref{ualphae}) and (\ref{GiionOmegaA}), we find
on $\nc\Omega(\eta)$ with $(\xi,\nu)$-high probability
%
%
\begin{eqnarray}\label{Aestimate}
\nc{\cal A}_i & = & \frac{f}{N} - \frac{f^2}{N} \frac{N - 1}{N}
\frac{|\langle{\f e}, {\f v_{\max}^{(i)}}\rangle|
^2}{\mu_{\max}^{(i)} - z} \nonumber\\
&&{}- \frac{f^2}{N} \frac{N - 1}{N}
\sum_{\alpha\neq\alpha_{\max}}\frac{|\langle{\f e}, {\f
v_\alpha^{(i)} }\rangle|^2}{\mu_\alpha^{(i)} - z}+ \frac{1}{N}
\sum_j \frac{ \nc G_{ij} \nc G_{ji}}{ \nc G_{ii}}
\nonumber\\
& = &\frac{f}{N} - \frac{f^2}{N} \frac{N - 1}{N} \frac{1}{f}
\biggl[{1 + O \biggl({\frac{1}{f}}\biggr)}\biggr]\\
&&{} + O
\biggl({\frac{1}{N \eta}}\biggr) +
O \biggl({\frac{f^2}{N \eta} \sum_{\alpha\neq\alpha_{\max}}
\bigl|\bigl\langle{\f e}, {\f v_\alpha^{(i)}}\bigr\rangle\bigr|^2}\biggr) + O
\biggl({\frac{1}{N} \sum_j |\nc G_{ij} |^2}\biggr)
\nonumber\\
& = & O \biggl({\frac{1}{N \eta}}\biggr),\nonumber
\end{eqnarray}
where in the second step we distinguished the two cases $f \leq2
\Sigma$ and $f \geq2 \Sigma$, as in
(\ref{f2N2bound}).
\end{pf}

We may now estimate $\nc\Lambda_d$ in terms of $\nc\Lambda$.
%
%
\begin{lemma}
Assume (\ref{fupperbound}).
For $z =E+i\eta\in D_L$ we have
%
%
\begin{equation}\label{lambdadleqlambdaA}
\max_{i} | \nc G_{ii}(z) - \nc m(z) | \leq C \Phi(z)
\end{equation}
on $\nc\Omega(\eta)$ with $(\xi,\nu)$-high probability. In
particular, on $\nc\Omega
(\eta)$ we have with $(\xi,\nu)$-high probability
%
%
\begin{equation}\label{436A}
\nc\Lambda_d(z) \leq\nc\Lambda(z) + C \Phi(z).
\end{equation}
\end{lemma}
\begin{pf}
Using (\ref{ZiaprioriA}), (\ref{AiaprioriA}) and Lemma \ref
{lemmaOmegah}, we find
%
%
\begin{equation}
\max_i |\nc\Upsilon_i |\leq(\log N)^\xi\frac
{C}{q} + C
(\log N)^{2 \xi} \Biggl({\sqrt{\frac{\im\nc m}{N \eta}} + \frac
{1}{N \eta}}\Biggr)
\end{equation}
on $\nc\Omega( \eta)$ with $(\xi,\nu)$-high probability. From
(\ref{self-consistentA}) we therefore get
%
%
\begin{eqnarray}
| \nc G_{ii}- \nc G_{jj}| &=& |\nc G_{ii} |
|\nc G_{jj} ||\nc\Upsilon_i-\nc\Upsilon_j |\nonumber\\[-8pt]\\[-8pt]
&\leq&
(\log N)^\xi\frac{C}{q} + C (\log N)^{2 \xi} \Biggl({\sqrt{\frac
{\im\nc m}{N \eta}} + \frac{1}{N \eta}}\Biggr)\nonumber
\end{eqnarray}
on $\nc\Omega(\eta)$ with $(\xi,\nu)$-high probability. Since $\nc
m = \frac
{1}{N}\sum_{j} \nc G_{jj}$, the proof is complete.
\end{pf}

\subsection{\texorpdfstring{Establishing $\nc\Omega(\eta)$ with high probability}
{Establishing Omega(eta) with high probability}}
What remains to complete the proof of Theorem~\ref{LSCTHMA} is to
prove that the events $\nc\Omega(\eta)$ hold with
$(\xi, \nu)$-high probability. We do this using a simplified version
of the continuity argument of Sections
\ref{sectinitialestimates} and~\ref{subseccon}.
%
%
\begin{lemma} \label{lemmalargeetaA}
If $\eta\geq2$, then $\nc\Omega(\eta)$ holds with $(\xi, \nu
)$-high probability.
\end{lemma}
\begin{pf}
The proof is similar to that of Lemma~\ref{5992}; we merely sketch the
modifications.

Let\vspace*{1pt} $z = E + \ii\eta\in D_L$ for $\eta\geq2$. We estimate $\nc
\Lambda_o(z)$ following closely the proof of
(\ref{GijaprioriA}), using (\ref{estimateoffGh}) and setting $R
= 1$ in (\ref{term3ofZij}). Using the rough bound
$|\nc G_{ij} |+ |\nc m |\leq1$ as in (\ref
{trivialresolventestimates}), we find
\[
\nc\Lambda_o \leq(\log N)^\xi\frac{C}{q} + (\log N)^{2 \xi}
\frac{C}{\sqrt{N}} + \frac{Cf}{N} \leq C (\log
N)^{- 2 \xi}
\]
with $(\xi,\nu)$-high probability.
Similarly, we find
\[
|\nc Z_i |\leq(\log N)^\xi\frac{C}{q} + (\log N)^{2
\xi}
\frac{C}{\sqrt{N}} \leq(\log N)^{- 2 \xi}
\]
with $(\xi,\nu)$-high probability.
In order to estimate $\nc{\cal A}_i$, we proceed similarly to (\ref
{Aestimateforlargeeta}) and find
\[
|\nc{\cal A}_i |\leq\frac{f}{N} + \frac{f^2}{N}
\bigl|\bigl\langle{\f e}, {\nc G^{(i)} \f e}\bigr\rangle\bigr|+ \frac
{1}{N} \sum_j
\bigl|\nc G_{jj}^{(i)} \bigr||\nc G_{ji} |({|
a_{ij} |+ |\nc Z_{ij} |}).
\]
The term $\nc Z_{ij}$ is estimated exactly as $\nc\Lambda_o$ above;
using the calculation of~(\ref{f2N2bound}), we
therefore get
\[
|\nc{\cal A}_i |\leq
(\log N)^\xi\frac{C}{q} + (\log N)^{2 \xi} \frac{C}{\sqrt{N}} +
\frac{Cf}{N} \leq C (\log N)^{- 2 \xi}
\]
with $(\xi,\nu)$-high probability.

Now we may\vspace*{2pt} follow the proof of Lemma~\ref{5992} to the letter,
starting from (\ref{self-consistentequationlargeeta})
to get $\nc\Lambda_d \leq C (\log N)^{-2\xi}$ with $(\xi,\nu
)$-high probability.

Thus, we have proved that $\nc\Lambda_d(z) + \nc\Lambda_o(z) \leq C
(\log N)^{-2\xi}$ with $(\xi,\nu)$-high probability. A simple lattice
argument along the lines of Corollary~\ref{cor414} then concludes the proof.
\end{pf}

The following simple continuity argument establishes $\nc\Omega(\eta
)$ with $(\xi,\nu)$-high probability for smaller $\eta$. Let $\eta_k$
be a sequence as in Section~\ref{subseccon}.

Note that, unlike in Section~\ref{subseccon}, each step $k \to k+1$
of the continuity argument has to establish a
statement for all $z \in D(\eta_{k+1})$.
%
%
\begin{lemma}
We have
\[
\P({\nc\Omega(\eta_k)^c}) \leq k \me^{-\nu
(\log N)^\xi}.
\]
\end{lemma}
\begin{pf}
We proceed by induction on $k$. The case $k = 1$ was proved in
Lem\-ma~\ref{lemmalargeetaA}. We write
\[
\P({\nc\Omega(\eta_{k+1})^c}) \leq\P\bigl({\nc
\Omega(\eta_k) \cap\nc\Omega(\eta_{k + 1})^c}\bigr) + \P
({\nc\Omega(\eta_k)^c}).
\]
Now for any $w \in D(\eta_k)$ and on $\nc\Omega(\eta_k)$ we have,
using (\ref{GijaprioriA}), (\ref{436A}) and
(\ref{scmA}),
\[
\nc\Lambda_d(w) + \nc\Lambda_o(w) \leq C (\log N)^{-2\xi}
\]
with $(\xi,\nu)$-high probability. Using the estimate (\ref{w-z}),
we find, for any $z
\in D(\eta_{k+1})$,
\[
\nc\Lambda_d(z) + \nc\Lambda_o(z) \leq C (\log N)^{-2\xi}
\]
with $(\xi,\nu)$-high probability. Using a lattice argument similar
to Corollary~\ref
{cor414}, we therefore find
\[
\P\bigl({\nc\Omega(\eta_k) \cap\nc\Omega(\eta_{k + 1})^c}
\bigr) \leq\me^{-\nu(\log N)^\xi}.
\]
The claim follows.
\end{pf}

This estimate (\ref{GijestimateA}) now follows from (\ref
{GijaprioriA}), (\ref{436A}), (\ref{scmA}) and the
lattice argument of Corollary~\ref{cor414}. This concludes the proof
of Theorem~\ref{LSCTHMA}.

\subsection{\texorpdfstring{Eigenvector delocalization: Proof of Theorem \protect\ref{theoremdelocalization}}
{Eigenvector delocalization: Proof of Theorem 2.16}} \label{sectionproofofdeloc}

We may now prove Theorem~\ref{theoremdelocalization}. Delocalization
for the eigenvectors $\f v_1,\ldots, \f v_{N -
1}$ is an immediate consequence of the weak local semicircle law. From
(\ref{interlacing}) and Lemma~\ref{lemmaboundon_h_tildeweak} we find that
$\mu_1,\ldots, \mu_{N - 1} \in
[-\Sigma, \Sigma]$ with $(\xi,\nu)$-high probability. Let $L = 8
\xi$ and set
$\eta\deq(\log N)^L$. Using (\ref{weakSCL}), Lemma \ref
{lemmaLambdaLambdaA} and (\ref{436A}), we therefore find
with $(\xi,\nu)$-high probability
%
%
\begin{equation} \label{pfofdeloc}\quad
C \geq\im\nc G_{jj}(\mu_\alpha+ \ii\eta) = \sum_\beta
\frac{\eta| v_\beta(j) |^2}{(\mu_\beta-
\mu_\alpha)^2 + \eta^2} \geq\frac{| v_\alpha(j) |
^2}{\eta
} \qquad(\alpha< N).\hspace*{-28pt}
\end{equation}
This concludes the proof of (\ref{delocforbulk}). Moreover, the same
argument [with $\alpha= N$ in (\ref{pfofdeloc})] proves (\ref
{delocforconstantf}) if $f \leq\Sigma- 3$,
since in that case $\mu_{\max} \in D$
with $(\xi,\nu)$-high probability by (\ref{positionofmumax}) and
Lemma~\ref{lemmaboundon_h_tildeweak}.

Next, we note that (\ref{l2estimate}) is an immediate consequence of
(\ref{uNe}).

In order to prove (\ref{delocalizationofumax}), we use the
following large deviation estimate which is proved in
the \hyperref[app]{Appendix}.
%
%
\begin{lemma} \label{lemmapinned_lde}
For $k \leq\log N$ and fixed $i$ we have with $(\xi,\nu)$-high probability
%
%
\begin{equation} \label{pinnedLDEestimate}
|(H^k \f e)(i) |= \biggl|\sum_{i_1,\ldots, i_k} h_{i i_1} h_{i_1
i_2} \cdots h_{i_{k - 1} i_k}
\biggr|\leq(\log
N)^{k \xi}.
\end{equation}
\end{lemma}

Now from the expansion (\ref{uNsum}) we get with $(\xi,\nu)$-high
probability
\[
K^{-1} v_{\max}(i) = \frac{1}{\sqrt{N}} + O \biggl({\frac
{(\log N)^\xi}{\sqrt{N} f}}\biggr),
\]
and (\ref{delocalizationofumax}) follows since $K = 1 + O(f^{-2})$
[see (\ref{computationofK})].
In this argument we used that $f\sim\mu_{\max} \geq C_0 (\log
N)^\xi$ for some large enough $C_0$ to overcome the
logarithmic factors
in (\ref{uNsum}) that arise from (\ref{pinnedLDEestimate}). This
concludes the proof of Theorem~\ref{theoremdelocalization}.

Finally, we outline the proof of (\ref{generaldeloc}) for $1 \ll f
\leq C (\log N)^\xi$. The idea is to use the same
proof as for (\ref{delocforconstantf}), relying\vspace*{1pt} on the estimate
(\ref{pfofdeloc}). In order to do this, we need the
pointwise bound $C \geq\nc G_{ii}(\mu_N + \ii\eta)$ which we get by
extending the proof of Theorem~\ref{LSCTHMA} to a
larger set $D_L$.
Here $D_L$ has to contain $\mu_N$, so that we have to choose
$\Sigma=C(\log N)^\xi$ in the definition (\ref{definitionDL}) of $D_L$
with some large constant C.

This extension requires some modifications in our proof of the local
semicircle law. Now\vspace*{1pt} instead of the bounds
(\ref{boundsonmsc}), we have $c (\log N)^{-\xi} \leq|
m_{\mathrm{sc}}(z) |\leq1$ for $z \in D_L$. We modify the definitions
(\ref{defOmegaz}) of $\Omega(z)$ and (\ref{Omgdef}) of
$\nc\Omega(\eta)$ by replacing $(\log N)^{-\xi}$ with $(\log N)^{- 2
\xi}$. Then, on these events, we get the lower bound $| G_{ii}(z) |\geq
c (\log N)^{-\xi}$ instead of (\ref{lowerboundonGii}). One can then
check that all estimates of Sections~\ref{sectionWLSC}--\ref{sectwtG}
remain valid with some deterioration in the form of larger powers of
$(\log N)^\xi$, provided that $L \geq C \xi$ for some large enough $C$;
we omit the details.

\subsection{\texorpdfstring{Control of the average of $\sum_{k \neq i} \nc G_{1k}^{(i)} h_{ki}$}
{Control of the average of sum (k not equal i) G 1k (i) h ki}} \label{sectaverageforedge}

In this final section we estimate the averaged quantity
%
%
\begin{equation} \label{resexperror}
\Pi\deq\frac{1}{N} \sum_{i \neq1} \sum_{k \neq i} \nc
G_{1k}^{(i)} h_{ki}.
\end{equation}
The estimate of $\Pi$ is not needed for the local semicircle law, but
we give its proof here, as it is a natural
application of the abstract decoupling lemma, Theorem \ref
{abstractZlemma}. The expression (\ref{resexperror}) arises
as an error term when controlling resolvent expansions of the
noncentered matrix $A$. Such expansions are used in the
companion paper~\cite{EKYY2} to establish the universality of the
extreme eigenvalues; see Section 6.3 in~\cite{EKYY2}.

Note that a naive application of the large deviation bound (\ref{aA})
yields $|\Pi|\leq(\log N)^{C \xi} q^{-1}$
with $(\xi,\nu)$-high probability. In order to establish universality
of the extreme
eigenvalues in~\cite{EKYY2}, it is crucial that the
factor $q^{-1}$ be improved to $q^{-2}$. This is the content of the
following proposition.
%
%
\begin{proposition} \label{propaverageforedge}
Suppose that the assumptions of Theorem~\ref{LSCTHMA} hold. Then for
any $z \in D_{120 (\xi+ 2)}$ we have
with $(\xi,\nu)$-high probability
%
%
\begin{equation} \label{boundonPi}
|\Pi(z) |\leq(\log N)^{C \xi} \biggl({\frac
{1}{q^2} + \frac{\im m_{\mathrm{sc}}(z)}{N \eta} + \frac{1}{(N \eta
)^2}}\biggr).
\end{equation}
\end{proposition}
\begin{pf}
We shall apply Theorem~\ref{abstractZlemma} to the quantities
%
%
\begin{eqnarray} \label{defZionesided}
\cal Z_i &\deq&\f1 (i \neq1) \sum_{k}^{(i)} \wt G_{1k}^{(i)}
h_{ki},\nonumber\\[-8pt]\\[-8pt]
\cal Z_i^{[\bb U]} &\deq&\f1 (i \neq1) \f1
(i,1 \notin\bb U) \sum_{k}^{(\bb U i)} \wt
G_{1k}^{(\bb U i)} h_{ki}.\nonumber
\end{eqnarray}
Thus, $\Pi= [\cal Z]$ and $\cal Z_i^{[\varnothing]} = \cal Z_i$.

We define the deterministic control parameters
\[
X(z) \deq(\log N)^{40 (\xi+ 2)} \Biggl({\frac{1}{q} + \sqrt
{\frac{\im m_{\mathrm{sc}}(z)}{N \eta}} + \frac{1}{N
\eta}}\Biggr),\qquad
Y(z) \deq(\log N)^\xi
\]
and the event
\[
\Xi\deq\bigcap_{z \in D_{120 (\xi+ 2)}} \Bigl\{{\max_{1
\leq i,j \leq N} |\nc G_{ij}(z) - \delta_{ij} m_{\mathrm{sc}}(z)
|\leq X(z)}\Bigr\}.
\]
Recall that the collection of random variables $({\cal Z_i^{[\bb
U]}})_{\bb U}$ generates random variables $\cal
Z_i^{\bb S, \bb U}$ through (\ref{56ASU}). We choose $p \deq(\log
N)^\xi$ in Theorem~\ref{abstractZlemma}. It is
immediate that the assumptions (i) and (iv) of Theorem \ref
{abstractZlemma} are satisfied. By Theorem~\ref{LSCTHMA},
the assumption (v) of Theorem~\ref{abstractZlemma} holds as well.\vadjust{\goodbreak}

We shall prove that, for any $\bb U \subset\bb S$ with $1,i \notin\bb
S$, $|\bb S |\leq p$, and $r \leq p$, we have
%
%
\begin{equation} \label{mainestimatefor711}
\E({\f1 ({[\Xi]_i}) |\cal Z_i^{\bb S,
\bb U} |^r}) \leq({Y (C u X)^u})^r,
\end{equation}
where $u \deq|\bb U |+ 1$. Supposing (\ref{mainestimatefor711})
is proved, both assumptions (ii) and (iii) of
Theorem~\ref{abstractZlemma} are satisfied. Then the claim of Theorem
\ref{abstractZlemma}, (\ref{resZZ}) and Markov's
inequality yield (\ref{boundonPi}).

It remains to prove (\ref{mainestimatefor711}). Throughout the
following we abbreviate $u \deq|\bb U |+ 1$. By
the definition of $\cal Z_i^{\bb S, \bb U}$ in (\ref{56ASU}) and
(\ref{defZionesided}) we find, for $1,i \notin\bb
S$,
\begin{eqnarray*}
\cal Z_i^{\bb S, \bb U} & = & \f1 (i \neq1) (-1)^{|\bb S
\setminus\bb U |} \sum_{\bb V \col\bb S \setminus\bb U
\subset\bb V \subset\bb S} (-1)^{|\bb V |} \sum_k^{(\bb V i)}
\nc G_{1k}^{(\bb V i)} h_{ki}
\\
& = &\f1 (i \neq1) \sum_{k}^{((i \bb S) \setminus\bb U)} h_{ki}
(-1)^{|\bb S \setminus\bb U |}
\sum_{\bb V \col\bb S \setminus\bb U \subset\bb V \subset\bb S
\setminus\{k\}} (-1)^{|\bb V |} \nc G_{1 k}^{(\bb
V i)}
\\
& = &\f1 (i \neq1) \sum_{k}^{((i \bb S) \setminus\bb U)} h_{ki}
\nc G_{1k}^{(\bb S i) \setminus\{k\}, \bb U
\setminus\{k\}},
\end{eqnarray*}
where in the last step we again used (\ref{56ASU}), as well as
Definition~\ref{58G} and the fact that $((\bb S i)
\setminus\{k\}) \setminus(\bb U \setminus\{k\}) = (\bb S i)
\setminus\bb U$.
We split
\[
\cal Z_i^{\bb S, \bb U} = D_1 + D_2,
\]
where
\[
D_1 \deq\f1 (i \neq1) \sum_{k \in\bb U} h_{ki} \nc
G_{1k}^{(\bb S i) \setminus\{k\}, \bb U \setminus\{k\}},\qquad
D_2 \deq\f1 (i \neq1) \sum_{k}^{(\bb S i)} h_{ki} \nc
G_{1k}^{(\bb S i), \bb U}.
\]
Thus, we may estimate
\[
\E({\f1 ({[\Xi]_i}) |\cal Z_i^{\bb S,
\bb U} |^r}) \leq2^r \sum_{j = 1}^2 \E
({\f1 ({[\Xi]_i}) | D_j |^r}).
\]
To that end, we shall make use of (\ref{616}). Note that Lemma \ref
{lemma612} is entirely deterministic. In
particular, it applies if all quantities are defined in terms of $A$
rather than $H$ (and hence bear a tilde in our
convention). We shall apply it to the Green function $\nc G$.

We start by estimating $D_1$. Since $\nc G_{1k}^{(\bb S i) \setminus\{
k\}, \bb U \setminus\{k\}}$ in $D_1$ is by
definition independent of the $i$th column of $H$, for $|\bb U
|
\leq|\bb S |\leq p = (\log N)^\xi$ we get from
(\ref{616}) that
\[
\bigl\|\f1 ({[\Xi]_i}) \nc G_{1k}^{(\bb S i)
\setminus\{k\}, \bb U \setminus\{k\}} \bigr\|_{\infty}
=
\bigl\|\f1 (\Xi) \nc G_{1k}^{(\bb S i) \setminus\{k\}, \bb U
\setminus\{k\}} \bigr\|_{\infty} \leq(C |\bb U |
X)^{|\bb U |}.
\]
Here we used that $\nc\Lambda_o \leq X$ on $\Xi$.
Now we may estimate, using $|\bb U |\leq(\log N)^\xi$,
\begin{eqnarray*}
\E({\f1 ([\Xi]_i) | D_1 |^r}) & \leq &(\log
N)^{r \xi} \max
_{k,i} \E| h_{ki} |^r ({C |\bb U | X})^{r
|\bb U |}
\\
& \leq &\bigl({C (\log N)^\xi q^{-1} (C |\bb U |
X)^{|\bb U |}}\bigr)^r
\\
& \leq &({Y (C u X)^u})^r.
\end{eqnarray*}

Next, we estimate $D_2$. As above, since $\nc G_{1k}^{(\bb S i), \bb
U}$ in $D_2$ is by definition independent of the
$i$th column of $H$, for $|\bb U |\leq(\log N)^\xi$ we get from
(\ref{616}) that
\[
\bigl\|\f1 ({[\Xi]_i}) \nc G_{1k}^{(\bb S i),
\bb U} \bigr\|_{\infty} =
\bigl\|\f1 (\Xi) \nc G_{1k}^{(\bb S i), \bb U} \bigr\|
_{\infty} \leq
(C |\bb U | X)^{|\bb U |+ 1}.
\]
Now we use the moment estimate (\ref{generalizedLDEapp}) with
$\alpha= 1$, $\beta= -2$, $\gamma= 1$ and
\[
A_k \deq\f1 (k \notin\bb S \cup\{i\}) \f1 ({[\Xi
]_i}) \nc G_{1k}^{(\bb S i), \bb U}.
\]
This yields
\[
\E({\f1 ([\Xi]_i) | D_2 |^r}) \leq
\bigl({r (C |\bb U | X)^{|\bb U |+ 1}}\bigr)^r.
\]
Here we used that $A_k$ defined above is independent of the randomness
in the $i$th column of $H$. Thus, we conclude
that
\[
\E({\f1 ([\Xi]_i) | D_2 |^r}) \leq(Y (C u
X)^u)^r.
\]
This completes the proof of (\ref{mainestimatefor711}), and hence
of (\ref{boundonPi}).
\end{pf}

\section{Density of states and eigenvalue locations}

\subsection{Local density of states}
The following estimate is the key tool for controlling the local
density of states---and hence proving Theorems
\ref{thmlocaldensityofstates} and~\ref{thmgeneralintegrateddensity}.
%
%
\begin{lemma} \label{lemmacountingfunctionestimate}
Recall the definition (\ref{defkappa}) of the distance $\kappa_E$
from $E$ to the spectral edge.
Suppose that the event
%
%
\begin{equation} \label{lschypothesis}\qquad
\bigcap_{z \in D_L} \biggl\{{|\nc m(z) - m_{\mathrm{sc}}(z)
|\leq(\log N)^{C \xi} \biggl(\min\biggl\{ \frac{1}{q^2 \sqrt
{\kappa_E + \eta}}, \frac{1}{q} \biggr\} +\frac{1}{N \eta}
\biggr) }\biggr\}
\end{equation}
holds with $(\xi,\nu)$-high probability for $L \deq C_0 \xi$, where
$C_0$ is a positive
constant.
For given $E_1 < E_2$ in $[-\Sigma, \Sigma]$ we abbreviate
%
%
\begin{equation} \label{defofcalE}\quad
\kappa\deq\min\{{\kappa_{E_1}, \kappa_{E_2}}\},\qquad \cal
E \deq\max\{{E_2 - E_1, (\log N)^L N^{-1}}\}.
\end{equation}
Then, for any $-\Sigma\leq E_1 < E_2 \leq\Sigma$, we have
%
%
\begin{equation} \label{mainestimateonn-nsc}
\bigl|\bigl({\nc{\fra n}(E_2) - \nc{\fra n}(E_1)}\bigr) -
\bigl({n_{\mathrm{sc}}(E_2) - n_{\mathrm{sc}}(E_1)}\bigr) \bigr|\leq
(\log N)^{C \xi} \biggl[{\frac{1}{N} + \frac{\cal E}{q^2 \sqrt
{\kappa+ \cal E}}}\biggr]\hspace*{-35pt}
\end{equation}
with $(\xi,\nu)$-high probability.
\end{lemma}
\begin{pf}
Recall the definitions (\ref{defrhosc}) and (\ref{defmsc}).
Similarly, we have
\[
\nc\varrho(x) = \frac{1}{N} \sum_{\alpha= 1}^N \delta(x - \mu
_\alpha),\qquad \nc{\fra n}(E) =
\int_{-\infty}^E \nc\varrho(x) \,\dd x = \frac{1}{N} |
\{{\alpha\col\mu_\alpha\leq E}\} |.
\]
Thus, we may write
\[
\nc m(z) = \frac{1}{N} \tr G(z) = \int\frac{\nc\varrho(x)
\,\dd x}{x - z}.
\]
We introduce the differences
\[
\varrho^\Delta\deq\nc\varrho- \varrho_{\mathrm{sc}},\qquad
m^\Delta\deq\nc m - m_{\mathrm{sc}}.
\]

Following~\cite{ERSY}, we use the Helffer--Sj\"ostrand functional
calculus. Set $\wt\eta\deq(\log N)^L N^{-1}$.
(Recall that $L = C_0 \xi$.) Let $\chi$ be a smooth cutoff function
equal to $1$ on $[-\cal E, \cal E]$ and vanishing
on $[-2 \cal E, 2 \cal E]^c$, such that $|\chi'(y) |\leq C
\cal
E^{-1}$. Set $\eta\deq N^{-1}$ and let $f$ be a
characteristic function of the interval $[E_1, E_2]$ smoothed on the
scale $\eta\dvtx f(x) = 1$ on $[E_1, E_2]$, $f(x) =
0$ on $[E_1 - \eta, E_2 + \eta]^c$, $| f'(x) |\leq C \eta^{-1}$,
and $| f''(x) |\leq C \eta^{-2}$. Note that the
supports of $f'$ and $f''$ have measure $O(\eta)$.

Then we have the estimate (see equation (B.13) in~\cite{ERSY})
%
%
\begin{eqnarray} \label{HSsplit}\quad
\biggl|\int f(\lambda) \varrho^\Delta(\lambda) \,\dd
\lambda\biggr|&\leq& C \biggl|\int\dd x \int
_0^\infty\dd y\, \bigl(f(x) + y f'(x)\bigr) \chi'(y) m^\Delta(x + \ii
y) \biggr|
\nonumber\\
&&{}+
C \biggl|\int\dd x \int_0^\eta\dd y\, f''(x) \chi(y) y
\im m^\Delta(x + \ii y) \biggr|\\
&&{}+
C \biggl|\int\dd x \int_\eta^\infty\dd y\, f''(x) \chi(y)
y \im m^\Delta(x + \ii y) \biggr|.\nonumber
\end{eqnarray}
Since $\chi'$ vanishes away from $[\cal E, 2 \cal E]$, the first term
on the right-hand side of (\ref{HSsplit}) is
bounded with $(\xi,\nu)$-high probability by
\begin{eqnarray*}
&&\frac{C}{\cal E}\int\dd x \int_{\cal E}^{2 \cal E} \dd y\, |
f(x) + y f'(x) |(\log N)^{C \xi} \biggl({\frac{1}{q^2 \sqrt
{\kappa
+ \cal E}} + \frac{1}{N \cal E}}\biggr)\\
&&\qquad\leq
(\log N)^{C \xi} \biggl({\frac{\cal E}{q^2 \sqrt{\kappa+ \cal E}}
+ \frac{1}{N}}\biggr),
\end{eqnarray*}
where we abbreviated $\kappa\deq\min\{{\kappa_{E_1}, \kappa
_{E_2}}\}$.
In order to estimate the two remaining terms of (\ref{HSsplit}), we
estimate $\im m^\Delta(x + \ii y)$. If $y \geq\wt
\eta$, we may use~(\ref{lschypothesis}). Consider therefore the case
$0 < y \leq\wt\eta$. From Lemma~\ref{lemmamsc} we find
%
%
\begin{equation} \label{immscbound}
|{\im m_{\mathrm{sc}}(x + \ii y) }|\leq C \sqrt{\kappa_x + y}.
\end{equation}
By spectral decomposition of $A$, it is easy to see that the function
$y \mapsto y \im\nc m(x + \ii y)$ is monotone\vadjust{\goodbreak}
increasing. Thus, we get, using (\ref{immscbound}), $x + \ii\wt
\eta\in D_L$ and (\ref{lschypothesis}), that
%
%
\begin{eqnarray} \label{immforsmally0}\qquad
y \im\nc m(x + \ii y) &\leq&\wt\eta\im\nc m(x + \ii\wt\eta)
\leq(\log N)^{C\xi} \wt\eta
\biggl({\sqrt{\kappa_x + \wt\eta} + \frac{1}{q} + \frac{1}{N \wt
\eta}}\biggr) \nonumber\\[-8pt]\\[-8pt]
&\leq&\frac{(\log N)^{C \xi}}{N}\qquad (y \leq
\wt\eta)\nonumber
\end{eqnarray}
with $(\xi,\nu)$-high probability. From (\ref{immscbound}),
$m^\Delta= \nc m -
m_{\mathrm{sc}}$, and the definition of $\wt\eta$, we find
%
%
\begin{equation} \label{immforsmally}
| y \im m^\Delta(x + \ii y) |\leq\frac{(\log N)^{C
\xi
}}{N} \qquad(y \leq\wt\eta)
\end{equation}
with $(\xi,\nu)$-high probability.
Since $\eta\leq\wt\eta$, we therefore find that the second term of
(\ref{HSsplit}) is bounded with $(\xi,\nu)$-high probability by
\[
\frac{C(\log N)^{C \xi}}{N} \int\dd x\, | f''(x) |\int
_0^\eta
\dd y\, \chi(y) \leq\frac{(\log N)^{C
\xi}}{N}.
\]

In order to estimate the third term on the right-hand side of (\ref
{HSsplit}), we integrate by parts, first in $x$ and
then in $y$, to obtain the bound
%
%
\begin{eqnarray} \label{HSafterIBP}
&&
C \biggl|\int\dd x\, f'(x) \eta\re m^\Delta(x + \ii\eta)
\biggr|\nonumber\\
&&\qquad{}+
C \biggl|\int\dd x \int_\eta^\infty\dd y\, f'(x) \chi'(y) y
\re m^\Delta(x + \ii y) \biggr|
\\
&&\qquad{}+
C \biggl|\int\dd x \int_\eta^\infty\dd y\, f'(x) \chi(y)
\re m^\Delta(x + \ii y) \biggr|.\nonumber
\end{eqnarray}
The second term of (\ref{HSafterIBP}) is similar to the first term
on the right-hand side of (\ref{HSsplit}), and is
easily seen to be bounded by
\[
(\log N)^{C \xi} \biggl({\frac{\cal E}{q^2 \sqrt{\kappa+ \cal E}}
+ \frac{1}{N}}\biggr).
\]

In order to bound the first and third terms of (\ref{HSafterIBP}),
we estimate, for any $y \leq\wt\eta$,
\[
| m^\Delta(x + \ii y) |\leq|
m^\Delta(x + \ii\wt\eta) |+ \int_y^{\wt\eta} \dd u\,
\bigl({|\partial_u \wt m(x + \ii u) |+
|\partial_u m_{\mathrm{sc}}(x + \ii u) |}\bigr).
\]
Moreover, using (\ref{immforsmally0}) and (\ref{ward}), we find
for any $u \leq\wt\eta$ that
\begin{eqnarray*}
|\partial_u \wt m(x + \ii u) |&=&
\biggl|\frac{1}{N} \tr\wt G^2(x + \ii u) \biggr|\leq\frac
{1}{N} \sum_{i,j} | G_{ij}(x + \ii u) |^2 \\
&=&
\frac{1}{u} \im\wt m(x + \ii u) \leq\frac{1}{u^2} \wt\eta
\im\wt m(x + \ii\wt\eta)
\end{eqnarray*}
with $(\xi,\nu)$-high probability.
Similarly, we find from (\ref{defmsc}) that
\[
|\partial_u m_{\mathrm{sc}}(x + \ii u) |\leq\frac
{1}{u^2} \wt\eta\im m_{\mathrm{sc}}(x + \ii\wt\eta).
\]
Thus, (\ref{lschypothesis}) yields
%
%
\begin{equation} \label{boundforfullmdelta}
| m^\Delta(x + \ii y) |\leq(\log N)^{C \xi
} \biggl({1 + \int_y^{\wt\eta} \dd u\, \frac{\wt\eta
}{u^2}}\biggr) \leq(\log N)^{C \xi} \frac{\wt\eta}{y}\qquad
(y \leq\wt\eta)\hspace*{-26pt}
\end{equation}
with $(\xi,\nu)$-high probability.

Using (\ref{boundforfullmdelta}), we may now bound the first term
of (\ref{HSafterIBP}) by $(\log N)^{C \xi} N^{-1}$.

What remains is the third term of (\ref{HSafterIBP}). We first split
the $y$-integration domain $[\eta, \infty)]$ into
the pieces $[\eta, \wt\eta]$ and $[\wt\eta, \infty)$. Using
(\ref{boundforfullmdelta}), we estimate the integral over
the first piece, with $(\xi,\nu)$-high probability, by
\[
\int\dd x\, | f'(x) |\int_\eta^{\wt\eta} \dd y\,
|
m^\Delta(x + \ii y) |\leq\frac{(\log N)^{C\xi
}}{N}.
\]
Using (\ref{lschypothesis}), we may therefore estimate the third term
of (\ref{HSafterIBP}), with $(\xi,\nu)$-high probability, by
\begin{eqnarray*}
&&\frac{(\log N)^{C \xi}}{N} +
(\log N)^{C \xi} \int\dd x \int_{\wt\eta}^{2 \cal E} \dd y\,
| f'(x) |\biggl({\frac{1}{N y} + \frac{1}{\sqrt{\kappa_x
+ y}}
\frac{1}{q^2}}\biggr)
\\
&&\qquad \leq\frac{(\log N)^{C \xi}}{N} +
(\log N)^{C \xi} \int\dd x\, | f'(x) |\biggl[{\int_{\wt
\eta}^{2 \cal E} \dd y\, \frac{1}{N y} + \frac{1}{q^2} \int_{\wt
\eta}^{2 \cal E} \dd y\, \frac{1}{\sqrt{\kappa+ y}}}\biggr]
\\
&&\qquad \leq
(\log N)^{C \xi} \biggl({\frac{1}{N} + \frac{\cal E}{q^2 \sqrt
{\kappa+ \cal E}}}\biggr).
\end{eqnarray*}

Summarizing, we have proved that
%
%
\begin{equation} \label{smoothedcountingfunction}
\biggl|\int f(\lambda) \varrho^\Delta(\lambda) \,\dd
\lambda\biggr|\leq(\log N)^{C \xi} \biggl[{\frac{1}{N}
+ \frac{\cal E}{q^2 \sqrt{\kappa+ \cal E}}}\biggr]
\end{equation}
with $(\xi,\nu)$-high probability.

In order to estimate $|\nc{\fra n}(E) - n_{\mathrm{sc}}(E) |$, we observe
that (\ref{immforsmally0}) implies
\[
|\nc{\fra n}(x + \eta) - \nc{\fra n}(x - \eta) |\leq
C
\eta\im\nc m(x + \ii\eta) \leq\frac{(\log N)^{C
\xi}}{N}
\]
with $(\xi,\nu)$-high probability. Thus, we get
\begin{eqnarray*}
\biggl|\nc{\fra n}(E_1) - \nc{\fra n}(E_2) - \int f(\lambda)
\varrho(\lambda) \,\dd\lambda\biggr|&\leq& C \sum
_{i =
1,2} \bigl({\nc{\fra n}(E_i + \eta) - \nc{\fra n}(E_i - \eta
)}\bigr) \\
&\leq&\frac{C(\log N)^{C \xi}}{N}
\end{eqnarray*}
with $(\xi,\nu)$-high probability.
Similarly, since $\varrho_{\mathrm{sc}}$ has a bounded density, we~find
\[
\biggl| n_{\mathrm{sc}}(E_1) - n_{\mathrm{sc}}(E_2) - \int f(\lambda)
\varrho _{\mathrm{sc}}(\lambda) \,\dd\lambda\biggr|\leq C \eta= \frac{C}{N}.
\]
Together with (\ref{smoothedcountingfunction}), we therefore get
(\ref{mainestimateonn-nsc}).\vspace*{-2pt}
\end{pf}

We draw two simple consequences from Lemma~\ref{lemmacountingfunctionestimate}.
%
%
\begin{proposition}[(Uniform local density of states)]
\label{thmgenerallocaldensityofstates}
Suppose that $A$ satisfies Definition~\ref{definitionofA} and that
$\xi$ and $q$ satisfy (\ref{assumptionsforSLSC}).
Then, for any $E_1$ and $E_2$ satisfying $E_2 \geq E_1 + (\log N)^{C
\xi} N^{-1}$ we have
%
%
\begin{eqnarray} \label{generaldensityofstatesestimate}\qquad
&&\nc{\cal N}(E_1, E_2) \nonumber\\[-8pt]\\[-8pt]
&&\qquad= \cal N_{\mathrm{sc}}(E_1, E_2) \biggl[{1 + O
\biggl({\frac{(\log N)^{C \xi}}{\cal N_{\mathrm{sc}}(E_1, E_2)} \biggl({1 +
\frac{N}{q^2} \frac{E_2 - E_1}{\sqrt{\kappa+ E_2 - E_1}}}
\biggr)}\biggr)}\biggr]\nonumber
\end{eqnarray}
with $(\xi,\nu)$-high probability, where we abbreviated $\kappa\deq
\min\{{\kappa_{E_1}, \kappa_{E_2}}\}$.\vspace*{-2pt}
\end{proposition}
\begin{pf}
By (\ref{scmA}), the estimate (\ref{lschypothesis}) holds. Assuming
$-\Sigma\leq E_1 \leq E_2 \leq\Sigma$, we get
from (\ref{mainestimateonn-nsc}), with $(\xi,\nu)$-high probability,
\[
|\nc{\cal N}(E_1, E_2) - \cal N_{\mathrm{sc}}(E_1, E_2) |
\leq(\log N)^{C \xi} \biggl({1 + \frac{N}{q^2} \frac{E_2 -
E_1}{\sqrt{\kappa+ E_2 - E_1}}}\biggr),
\]
from which the claim follows. If $E_1 < -\Sigma$ and $\E_2 \leq
\Sigma$, the claim follows by replacing $E_1$ with
$-\Sigma$ and using Lemma~\ref{lemmaboundonHtilde}. The other
cases where $-\Sigma\leq E_1 \leq E_2 \leq\Sigma$
does not hold are treated similarly using
Lemma~\ref{lemmaboundonHtilde}.\vspace*{-2pt}~%
\end{pf}

The proof of Theorem~\ref{thmlocaldensityofstates} is completed
by observing that both its statements, (\ref{localdensityinbulk})
and (\ref{localdensityatedge}), are special cases
of (\ref{generaldensityofstatesestimate}).
[Recall that in the bulk we have $\cal N_{\mathrm{sc}}(E_1, E_2) \sim N (E_2 -
E_1)$; at the spectral edge we have $\cal
N_{\mathrm{sc}}(E_1, E_2) \geq N (E_2 - E_1)^{3/2}$, which is sharp for $E_1 = -2$.]\vspace*{-2pt}
\begin{pf*}{Proof of Theorem~\ref{thmgeneralintegrateddensity}}
Let us assume that $E \leq0$; the case $E > 0$ is treated similarly. Setting
\[
E_1 \deq-2 - (\log N)^{C_1 \xi} (q ^{-2} + N^{-2/3})
\]
for some constant $C_1 > 0$, we find that $n_{\mathrm{sc}}(E_1) = 0$ and $\nc
{\fra n}(E_1) = 0$ with $(\xi,\nu)$-high probability for $C_1$ large
enough, by Lemma~\ref{lemmaboundonHtilde}. We may assume that $E
\geq E_1$.

By (\ref{scmA}), the estimate (\ref{lschypothesis}) holds.
Therefore, setting $E_2 = E$ in Lem\-ma~\ref
{lemmacountingfunctionestimate} yields
\begin{eqnarray*}
|\nc{\fra n}(E) - n_{\mathrm{sc}}(E) |&\leq&(\log
N)^{C \xi} \biggl({\frac{1}{N} + \frac{1}{q^2} \sqrt{E - E_1 +
(\log N)^{C \xi} N^{-1}}}\biggr)\\
&\leq&(\log N)^{C \xi} \biggl({\frac{1}{N} + \frac{1}{q^3} +
\frac{\sqrt{\kappa_E}}{q^2}}\biggr)\vadjust{\goodbreak}
\end{eqnarray*}
with $(\xi,\nu)$-high probability.
This holds for any fixed $E$.
The claim (\ref{n-nsc}), which is uniform in $E$, now follows by a
lattice argument similar to Corollary~\ref{cor414}, whereby we choose a lattice of points $E_i \in[-\Sigma
,\Sigma]$ with $| E_{i + 1} - E_i |\leq N^{-1}$;
we omit the details.
\end{pf*}

\subsection{Eigenvalue locations}
The following result contains the main estimate on the locations of the
eigenvalues $\mu_\alpha$ of $A$. Recall the
definition (\ref{defofgamma}) of the classical location $\gamma
_\alpha$ of the $\alpha$th eigenvalue.
%
%
\begin{proposition} \label{propeigenvaluelocations}
Suppose that $A$ satisfies Definition~\ref{definitionofA} and that
$\xi$ satisfies~(\ref{assumptionsforSLSC}).
Let $\phi$ be an exponent satisfying $0 < \phi\leq1/2$ and assume
that $q = N^{\phi}$. Then the following statements
hold with $(\xi,\nu)$-high probability for all $\alpha= 1,\ldots, N
- 1$, for some
sufficiently large constant $K$:
\begin{longlist}
\item
If $\max\{{\kappa_{\mu_\alpha}, \kappa_{\gamma_\alpha}}\} \leq
(\log N)^{K\xi} (N^{-2/3} + N^{-2 \phi})$, then
%
%
\begin{equation} \label{evalueresult1}
|\mu_\alpha- \gamma_\alpha|\leq(\log N)^{C \xi}
({N^{-2/3} + N^{-2\phi}}).
\end{equation}
\item
If $\max\{{\kappa_{\mu_\alpha}, \kappa_{\gamma_\alpha}}\} \geq
(\log N)^{K \xi} (N^{-2/3} + N^{-2 \phi})$, then
%
%
\begin{equation} \label{evalueresult2}\quad
|\mu_\alpha- \gamma_\alpha|\leq(\log N)^{C \xi}
({N^{-2/3} \wh\alpha^{-1/3} + N^{2/3 - 4 \phi} \wh\alpha
^{-2/3} + N^{-2 \phi}}),
\end{equation}
where we abbreviated $\wh\alpha\deq\min\{{\alpha, N - \alpha}\}$.
\end{longlist}
\end{proposition}
\begin{pf}
To simplify the presentation, we concentrate only on the eigenvalues
$\mu_1,\ldots, \mu_{N/2}$.
The remaining eigenvalues $\mu_{N/2 + 1},\ldots, \mu_{N - 1}$ are
dealt with similarly, using Lemma~\ref{lemmaboundonHtilde} and the
estimate $\mu_N \geq2 + c$ which holds with $(\xi
,\nu)$-high probability.

We define the event $\wt\Omega$ as the intersection of the events on
which (\ref{boundonHtilde}) holds and on which
%
%
\begin{equation} \label{densitylocinproof}
|\nc{\fra n}(E) - n_{\mathrm{sc}}(E) |\leq(\log
N)^{C_0 \xi} \biggl({\frac{1}{N} + \frac{1}{q^3} + \frac{\sqrt
{\kappa_E}}{q^2}}\biggr)
\end{equation}
holds for all $E \in[-\Sigma, \Sigma]$ and some positive constant
$C_0$. Recalling (\ref{n-nsc}) and (\ref{boundonHtilde}), we
find that $\wt\Omega$ holds with $(\xi,\nu)$-high probability for large
enough~$C_0$. Note that on $\wt\Omega$ we have
$\mu_{N/2} \leq1$. Indeed, the condition $\mu_{N/2} \leq1$ is
equivalent to $\nc{\fra n}(1) \geq1/2$, which follows
from (\ref{densitylocinproof}) and the fact that $n_{\mathrm{sc}}(1) >
1/2$.\looseness=1

Abbreviate $\zeta\deq\min\{{2 \phi, 2/3}\}$ and let $C_1 > C_0$. We
use the dyadic decomposition
\[
\{1,\ldots, N/2\} = \bigcup_{k = 0}^{2 \log N} U_k,
\]
where we defined
\begin{eqnarray}
U_0 & \deq & \bigl\{{\alpha\leq N/2 \col2 + \max\{{\gamma_\alpha
, \mu_\alpha}\} \leq2 (\log N)^{C_1 \xi} N^{-\zeta}}\bigr\},
\nonumber\\
U_k & \deq & \bigl\{{\alpha\leq N/2 \col2^k (\log N)^{C_1 \xi}
N^{-\zeta} < 2 + \max\{{\gamma_\alpha, \mu_\alpha}\} \leq2^{k+1}
(\log N)^{C_1 \xi} N^{-\zeta}}\bigr\}\nonumber\\
&&\eqntext{\mbox{for } k \geq
1.}
\end{eqnarray}
By definition of $U_0$ and Lemma~\ref{lemmaboundonHtilde}, on
$\wt\Omega$ we have
\[
|\mu_\alpha- \gamma_\alpha|\leq(\log N)^{C \xi}
N^{-\zeta} \qquad(\alpha\in U_0).
\]
This proves (\ref{evalueresult1}).

Next, let $k \geq1$. From (\ref{densitylocinproof}) we find that
on $\wt\Omega$ we have
%
%
\begin{equation} \label{compnnscforevaluelocations}\quad
\frac{\alpha}{N} = n_{\mathrm{sc}}(\gamma_\alpha) = \nc{\fra n}
(\mu_\alpha) = n_{\mathrm{sc}}(\mu_\alpha) + (\log N)^{C_0
\xi} O\biggl({\frac{1}{N} + \frac{1}{q^3} + \frac{\sqrt{\kappa
_{\mu_\alpha}}}{q^2}}\biggr).\hspace*{-32pt}
\end{equation}
On $\wt\Omega$ and for $\alpha\in U_k$, the second term on the
right-hand side of (\ref{compnnscforevaluelocations}) may be
estimated as
\begin{eqnarray*}
&&(\log N)^{C_0 \xi} O\biggl({\frac{1}{N} + \frac{1}{q^3} + \frac
{\sqrt{\kappa_{\mu_\alpha}}}{q^2}}\biggr)\\
&&\qquad\leq
(\log N)^{C_0 \xi} (N^{-1} + N^{- 3 \phi}) + C 2^{(k + 1)/2} (\log
N)^{(C_0 + C_1 /2)\xi} N^{- \zeta/2 - 2 \phi},
\end{eqnarray*}
since $\kappa_{\mu_\alpha} \leq2 + \mu_\alpha$.
Moreover, on $\wt\Omega$ and for $\alpha\in U_k$ we have
\[
n_{\mathrm{sc}}(\gamma_\alpha) + n_{\mathrm{sc}}(\mu_\alpha) \geq c 2^{3k/2}
(\log N)^{(3/2) C_1 \xi} N^{-3\zeta/2},
\]
where we used the simple estimate $n_{\mathrm{sc}}(-2 + x) \sim x^{3/2}$ for $0
\leq x \leq3$. Thus, we have, on $\wt\Omega$ and
for $\alpha\in U_k$,
\[
(\log N)^{C_0 \xi} O\biggl({\frac{1}{N} + \frac{1}{q^3} + \frac
{\sqrt{\kappa_{\mu_\alpha}}}{q^2}}\biggr) \ll
n_{\mathrm{sc}}(\gamma_\alpha) + n_{\mathrm{sc}}(\mu_\alpha),
\]
from which we deduce using (\ref{compnnscforevaluelocations}) that
\[
n_{\mathrm{sc}}(\mu_\alpha) = n_{\mathrm{sc}}(\gamma_\alpha) \bigl({1 + O
\bigl[{(\log N)^{-(C_1 - C_0) \xi}}\bigr]}\bigr).
\]
Thus, we find, on $\wt\Omega$ and for $\alpha\in U_k$, that $2 +
\gamma_\alpha\sim2 + \mu_\alpha$ and, hence, $n_{\mathrm{sc}}'(x) \sim
n_{\mathrm{sc}}'(\gamma_\alpha)$ for any $x$ between $\gamma _\alpha$
and $\mu_\alpha$. Here\vspace*{1pt} we used that $n_{\mathrm{sc}}'(x)
\sim(n_{\mathrm{sc}}(x))^{1/3} \sim\sqrt{2 + x}$ for $-2 \leq x \leq1$.
Thus, the mean value theorem and (\ref{compnnscforevaluelocations})
imply, on $\wt\Omega$ and for $\alpha\in U_k$,
\begin{eqnarray*}
&&
|\mu_\alpha- \gamma_\alpha|\\
&&\qquad \leq \frac{C
|
n_{\mathrm{sc}}(\mu_\alpha) - n_{\mathrm{sc}}(\gamma_\alpha) |
}{n_{\mathrm{sc}}'(\gamma_\alpha)}
\\
&&\qquad \leq \frac{C (\log N)^{C_0 \xi}}{(\alpha/N)^{1/3}}
\bigl({N^{-1} + N^{- 3 \phi} + N^{-2 \phi} \sqrt{\kappa_{\mu_\alpha
}}}\bigr)
\\
&&\qquad \leq \frac{C (\log N)^{C_0 \xi}}{\alpha^{1/3}}
\bigl({N^{-2/3} + N^{1/3 - 3 \phi} + N^{-2 \phi}\alpha^{1/3} + N^{1/3 -2
\phi} \sqrt{|\mu_\alpha- \gamma_\alpha|}}\bigr),
\end{eqnarray*}
where we used that $\kappa_{\mu_\alpha} \leq\kappa_{\gamma_\alpha
} + |\mu_\alpha- \gamma_\alpha|$ and
$\kappa_{\gamma_\alpha} \sim(\alpha/ N)^{2/3}$. Thus, we find, on
$\wt\Omega$ and for $\alpha\in
U_k$,
\[
|\mu_\alpha- \gamma_\alpha|\leq(\log N)^{C \xi}
({N^{-2/3} \alpha^{-1/3} + N^{2/3 - 4 \phi}\alpha^{-2/3} + N^{-2
\phi}}).
\]
This proves (\ref{evalueresult2}).
\end{pf}
\begin{pf*}{Proof of Theorem~\ref{thmeigenvaluelocations}}
We apply Proposition~\ref{propeigenvaluelocations}. As before, we
only deal with the eigenvalues $\alpha\leq N/2$;
the proof for the eigenvalues $N/2 < \alpha\leq N - 1$ is the same.
Suppose that $\alpha$ satisfies Case (i) of
Proposition~\ref{propeigenvaluelocations}. Using $\alpha/ N =
n_{\mathrm{sc}}(\gamma_\alpha) \sim(2 + \gamma_\alpha)^{3/2}$,
we find that
%
%
\begin{equation} \label{boundonalphaatedge}
\alpha\leq(\log N)^{K \xi} (1 + N^{1 - 3 \phi}).
\end{equation}
Therefore, we get, squaring (\ref{evalueresult1}) and (\ref
{evalueresult2}) and summing over $\alpha$,
\begin{eqnarray*}
\sum_{\alpha= 1}^{N - 1} |\mu_\alpha- \gamma_\alpha|^2
&\leq&
(\log N)^{C \xi} ({1 + N^{1 - 3 \phi}})({N^{-4/3} +
N^{-4 \phi}}) \\
&&{}+ (\log N)^{C \xi} ({N^{-1} + N^{4/3 - 8
\phi} + N^{1 - 4 \phi}})
\end{eqnarray*}
with $(\xi,\nu)$-high probability. This concludes the proof of (\ref
{mainestimateonQ}).

Finally, we note that (\ref{detailedestimateforlargephi}) follows
from (\ref{evalueresult1}) and (\ref{evalueresult2}) as well as
the above observation that Case (i) in Proposition
\ref{propeigenvaluelocations} implies
(\ref{boundonalphaatedge}).
\end{pf*}

%
\begin{appendix}\label{app}

\section*{\texorpdfstring{Appendix: Moment estimates: Proofs of Lemmas
\lowercase{\protect\ref{lemma_lde},
\protect\ref{lemmaboundon_h_tildeweak},
\protect\ref{lemmae_he} and
\protect\ref{lemmapinned_lde}}}
{Appendix: Moment estimates: Proofs of Lemmas 3.8, 4.3, 6.5 and 7.10}}

In order to prove Lemma~\ref{lemma_lde}, we prove the following high
moment bounds, which are also independently
useful.
%
%
\setcounter{theorem}{0}
\begin{lemma} \label{lemmamomentbounds}
\textup{(i)} Let $(a_i)$ be a family of centered and independent random variables satisfying
%
%
\setcounter{equation}{0}
\begin{equation} \label{generalizedmomentconditionapp}
\E| a_i |^p \leq\frac{C^p}{N^\gamma q^{\alpha p +
\beta}}
\end{equation}
for all $2 \leq p \leq(\log N)^{A_0 \log\log N}$, where $\alpha\geq
0$ and $\beta, \gamma\in\R$. Then for all even
$p$ satisfying $2 \leq p \leq(\log N)^{A_0 \log\log N}$ we have
%
%
\begin{equation} \label{generalizedLDEapp}
\E\biggl|\sum_i A_i a_i \biggr|^p \leq(Cp)^{p}
\biggl[{\frac{\sup_i | A_i |}{q^\alpha} + \biggl({\frac
{1}{N^\gamma q^{\beta+ 2 \alpha}} \sum_i | A_i |^2}
\biggr)^{1/2}}\biggr]^p
\end{equation}
for some constant $C > 0$ depending only on the constant in (\ref
{generalizedmomentconditionapp}).

\textup{(ii)}
Let $a_1,\ldots, a_N$ be centered and independent random variables satisfying
%
%
\begin{equation} \label{boundonmomentsofsparseentriesapp}
\E| a_i |^p \leq\frac{C^p}{N q^{p - 2}}
\end{equation}
for all $2 \leq p \leq(\log N)^{A_0 \log\log N}$. Then for all even
$p$ satisfying $2 \leq p \leq(\log N)^{A_0 \log\log
N}$ and all $B_{ij} \in\C$ we have
%
%
\begin{equation} \label{aaBoapp}\quad
\E\biggl|\sum_{i \neq j} \ol{a}_i B_{ij} a_j
\biggr|^p \leq(Cp)^{2p}
\biggl[{\frac{\max_{i \neq j} | B_{ij} |}{q} +
\biggl({\frac{1}{N^2} \sum_{i \neq j} | B_{ij} |^2}
\biggr)^{1/2}}\biggr]^p
\end{equation}
for some $C$ depending only on the constant in (\ref
{boundonmomentsofsparseentriesapp}).
\end{lemma}
\begin{pf}
We begin with (i). To prove (\ref{generalizedLDEapp}), we set $p =
2r$ and compute
%
%
\begin{equation} \label{expectationofp-thmomentexpanded}\quad
\E\biggl|\sum_i A_i a_i \biggr|^{2 r} = \sum_{i_1,\ldots, i_{2r}} \ol
{A}_{i_1} \cdots\ol{A}_{i_{r}}
A_{i_{r+1}} \cdots
A_{i_{2r}} \E\ol{a}_{i_1} \cdots\ol{a}_{i_r} a_{i_{r+1}}
\cdots a_{i_{2r}}.\hspace*{-32pt}
\end{equation}
Each configuration of labels $(i_1,\ldots, i_{2r})$ defines an
equivalence relation (or partition) $\Gamma$ on the index
set $\{1,\ldots, 2r\}$ by requiring that the indices $j$ and $k$ are
in the same equivalence class if and only if their
labels satisfy $i_j = i_k$. We
organize the summation over the labels $i_1,\ldots, i_{2r}$ by (i)
prescribing a partition $\Gamma$ of the set of
indices, (ii) summing over all label configurations yielding the
partition $\Gamma$, and (iii) summing over all
partitions $\Gamma$. Thus, let a partition $\Gamma$ be given. Let $l$
denote the number of equivalence classes of
$\Gamma$, and order the equivalence classes in some arbitrary fashion.
Let $r_s$ be the size of equivalence class $s$;
clearly, we have $r_1 + \cdots+ r_l = 2r$. Moreover, since the random
variables $a_i$ are centered, we find that each
equivalence class has size at least $2$; in particular, $r_s \geq2$
for each $s$ and, hence, $l \leq r$. Using the
independence of the $a_i$'s, we thus find that the contribution of the
partition $\Gamma$ to (\ref{expectationofp-thmomentexpanded}) is
bounded in absolute value by
%
%
\begin{equation} \label{expectationofpartition}
\sum_{i_1,\ldots, i_l} \prod_{s = 1}^l | A_{i_s} |^{r_s} \E
| a_{i_s} |^{r_s} \leq\prod_{s = 1}^l \biggl({\sum_i
| A_i |^{r_s} \frac{C^{r_s}}{N^\gamma q^{\alpha r_s + \beta
}}}\biggr).
\end{equation}
Abbreviating $A \deq\max_i | A_i |$, we find that (\ref
{expectationofpartition}) is bounded by
\begin{eqnarray*}
&&\prod_{s = 1}^l \biggl({(CA q^{-\alpha})^{r_s} A^{- 2} N^{-\gamma}
q^{-\beta} \sum_i | A_i |^2 }\biggr)\\
&&\qquad = (C A
q^{-\alpha})^{2r} \biggl({\frac{1}{A^2 N^\gamma q^\beta} \sum_i
| A_i |^2}\biggr)^l
\\
&&\qquad \leq (C A q^{-\alpha})^{2r} \max\biggl\{{1, \biggl({\frac
{1}{A^2 N^\gamma q^\beta} \sum_i | A_i |^2}
\biggr)^{r}}\biggr\}
\\
&&\qquad \leq C^r \biggl[{\frac{A}{q^\alpha} + \biggl({\frac
{1}{N^\gamma q^{\beta+ 2 \alpha}} \sum_i | A_i |^2}
\biggr)^{1/2}}\biggr]^{2r}.
\end{eqnarray*}

Next,\vspace*{1pt} it is easy to see that the total number of partitions of $2r$
elements is bounded by $(C r)^{2r}$, so that we get
\[
\E\biggl|\sum_i A_i a_i \biggr|^{2r} \leq(Cr)^{2r}
\biggl[{\frac{A}{q^\alpha} + \biggl({\frac{1}{N^\gamma q^{\beta+
2 \alpha}} \sum_i | A_i |^2}\biggr)^{1/2}}\biggr]^{2r}.
\]
This concludes the proof of (\ref{generalizedLDEapp}).

The proof of (\ref{aaBoapp}) needs more effort. Without loss of
generality, we set $B_{ii} = 0$ for all $i$. As above,
we set $p = 2r$. We find
%
%
\begin{eqnarray} \label{momentofaaBo}\qquad
\E\biggl|\sum_{i \neq j} \ol{a}_i B_{ij} a_j
\biggr|^{2r}
&=& \sum_{i_1,\ldots, i_{4r}} \ol{B}_{i_1 i_2} \cdots
\ol{B}_{i_{2r - 1} i_{2r}} B_{i_{2r+1} i_{2r+2}} \cdots B_{i_{4r
- 1} i_{4r}}
\nonumber\\[-8pt]\\[-8pt]
&&\hspace*{26.6pt}{}\times\E a_{i_1} \ol{a}_{i_2} \cdots a_{i_{2r - 1}}
\ol{a}_{i_{2r}}
\ol{a}_{i_{2r+1}} a_{i_{2r+2}} \cdots\ol{a}_{i_{4r - 1}}
a_{i_{4r}}.\nonumber
\end{eqnarray}
As above, we associate a partition $\Gamma(\f i) \equiv\Gamma= \{
\gamma\}$ of the index set $\{1,\ldots, 4r\}$ with
every label configuration $\f i = (i_1,\ldots, i_{4r})$ by requiring
that $k$ and $l$ are in the same equivalence class
of $\Gamma$ if and only if $i_k = i_l$. We rewrite (\ref{momentofaaBo})
by first specifying a partition $\Gamma$ and
summing over all label configurations $\f i$ satisfying $\Gamma(\f i)
= \Gamma$, and subsequently summing over all
partitions $\Gamma$. Note that a partition $\Gamma$ yields a nonzero
contribution to the right-hand side of
(\ref{momentofaaBo}) only if (i) each equivalence class contains at
least two indices, and (ii) $[2k - 1] \neq[2k]$
for all $k = 1,\ldots, 2r$; here $[n]$ denotes the equivalence class
$\gamma\ni n$ of $n$ in $\Gamma$.

Next, we encode $\Gamma$ using a multigraph (i.e., a graph which may
have multiple edges) $G \equiv G(\Gamma)$ defined
as follows. The vertex set of $G$ is the set of equivalence classes $\{
\gamma\}$ of $\Gamma$. Each factor $\ol{B}_{i_{2
k - 1} i_{2k}}$ or $B_{i_{2k - 1} i_{2k}}$ gives rise to an edge of $G$
connecting the vertices $[2k - 1]$ and $[2k]$.
Note that, by property (ii) of $\Gamma$, no edge of $G$ connects a
vertex to itself. Moreover, $G$ has $2r$ edges.

Let $G$ be a multigraph with $v$ vertices. We define the \textit{value}
of $G$ through
%
%
\begin{equation} \label{definitionofvalueofgraph}
\scr V(G) \deq\sum_{i_1,\ldots, i_v} \biggl({\prod_{\{\gamma
, \gamma'\} \in E(G)} | B_{i_\gamma i_{\gamma'}} |}\biggr)
\prod_{\gamma= 1}^v \frac{1}{N q^{ [\delta_\gamma- 2]_+}},
\end{equation}
where $\delta_\gamma$ is the degree of $\gamma$ in $G$.

Fix a partition $\Gamma$. We claim that the contribution to the
right-hand side of (\ref{momentofaaBo}) of all label
configurations $\f i$ satisfying $\Gamma(\f i) = \Gamma$ is bounded
in absolute value by $C^r \scr V(G(\Gamma))$.
This is an easy consequence of the definition of $G(\Gamma)$: each
vertex $\gamma$ carries a label $i_\gamma$, and the
contribution of vertex $\gamma$ is bounded by $\E| a_{i_\gamma}
|^{\delta_\gamma} \leq C^{\delta_\gamma} (N
q^{\delta_\gamma- 2})^{-1}$. [Note that, by the property (i) of
$\Gamma$, we have $\delta_\gamma\geq2$. Here we also
used that $\sum_\gamma\delta_\gamma= 4r$.]

Next, we estimate $\scr V(G(\Gamma))$. Let $G_0 \deq G(\Gamma)$. The
idea is to construct a sequence of multigraphs
$G_0, G_1,\ldots, G_s$ by successively removing edges incident to
vertices of degree greater than two, until all vertices
have degree at most two.

If all vertices of $G_0$ have degree at most two, set $s = 0$.
Otherwise, pick a vertex $\wt\gamma$ of $G_0$ with
degree greater than two, and let $\wt\gamma'$ be adjacent to~$\wt
\gamma$. Define $R(G_0)$ as the multigraph obtained
from $G_0$ by removing an edge connecting $\wt\gamma$ and $\wt\gamma
'$. We claim that
%
%
\begin{equation} \label{removinganedge}
\scr V(G_0) \leq\frac{B_o}{d} \scr V(R(G_0))
\end{equation}
(regardless of the choice of the removed edge). Here we abbreviated
$B_o \deq\max_{i \neq j} | B_{ij} |$. The
estimate (\ref{removinganedge}) is obtained by estimating
$| B_{i_{\wt\gamma} i_{\wt\gamma'}} |\leq B_o$ in~(\ref
{definitionofvalueofgraph}), and by noting that
$[\delta_\gamma- 2]_+$ in $G_0$ is strictly greater than in $R(G_0)$.
Now set $G_1 \deq R(G_0)$.

We continue inductively in this manner, generating a sequence
$G_0,\ldots, G_s$ of multigraphs with the properties that
$G_{k+1} = R(G_k)$ (for an immaterial choice of~$R$), $G_s$ has $2r -
s$ edges, and all vertices of $G_s$ have degree at
most two. By~(\ref{removinganedge}), we have
%
%
\begin{equation} \label{removingexcessedges}
\scr V(G_0) \leq\biggl({\frac{B_o}{q}}\biggr)^s \scr V(G_s).
\end{equation}

Next, it is immediate from its definition that $G_s$ is a disjoint
union of simple closed and open paths. Here a simple
open path of length $l \geq0$ is the graph with vertices $1,\ldots, l
+ 1$ and edges $\{1,2\},\ldots, \{l, l + 1\}$;
similarly, a simple closed path of length $l \geq2$ is the graph with
vertices $1,\ldots, l$ and edges $\{1,2\},\ldots,
\{l - 1, l\}, \{l,1\}$.

From the definition (\ref{definitionofvalueofgraph}) we
immediately find
%
%
\begin{equation} \label{factorizationofgraph}
\scr V(G \cup G') = \scr V(G) \scr V(G'),
\end{equation}
where $\cup$ denotes disjoint union. We shall now prove that, if $G$
is a simple (open or closed) path of length $l$, we
have
%
%
\begin{equation} \label{estimateofsimplepath}
\scr V (G) \leq\biggl({\frac{1}{N^2} \sum_{i,j} | B_{ij}
|^2}\biggr)^{l/2}.
\end{equation}
Using (\ref{removingexcessedges}), (\ref{factorizationofgraph})
and (\ref{estimateofsimplepath}), we find that
%
%
\begin{equation} \label{boundoncompletemultigraph}
\scr V(G(\Gamma)) \leq\biggl({\frac{B_o}{q}}\biggr)^s
\biggl({\frac{1}{N^2} \sum_{i,j} | B_{ij} |^2}\biggr)^{(2r -
s)/2}.
\end{equation}

Let us now prove (\ref{estimateofsimplepath}). We start with a
simple closed path of length $l$, whose value
(\ref{definitionofvalueofgraph}) is given by
\[
\cal C_l \deq\frac{1}{N^l} \sum_{i_1,\ldots, i_l} | B_{i_1
i_2} |\cdots| B_{i_{l -1} i_l} || B_{i_l i_1}
|.
\]
Assume first that $l = 2k$ is even. Then
\begin{eqnarray*}
\cal C_{2k} & \leq & \frac{1}{N^{2k}} \biggl({\sum_{i_1,\ldots,
i_{2k}} | B_{i_1 i_2} |^2 | B_{i_3 i_4} |^2 \cdots
| B_{i_{2k - 1} i_{2k}} |^2 }\biggr)^{1/2}\\
&&{}\times \biggl({\sum
_{i_1,\ldots, i_{2k}} | B_{i_2 i_3} |^2 | B_{i_4 i_5}
|^2 \cdots| B_{i_{2k} i_1} |^2}\biggr)^{1/2}
\\
& \leq & \biggl({\frac{1}{N^2} \sum_{i,j} | B_{ij} |
^2}
\biggr)^{l/2}.
\end{eqnarray*}
If $l = 2k + 1$ is odd, we find
\begin{eqnarray*}
\cal C_{2k + 1} & = & \frac{1}{N^{2k + 1}} \sum_{i_1, i_2} |
B_{i_1 i_2} |\biggl({\sum_{i_3,\ldots, i_{2k+1}} | B_{i_2
i_3} |\cdots| B_{i_{2k + 1} i_1} |}\biggr)
\\
& \leq &\frac{1}{N^{2k + 1}} \biggl({\sum_{i_1, i_2} | B_{i_1
i_2} |^2}\biggr)^{1/2}\\
&&{}\times \biggl({\sum_{i_1,\ldots, i_{2k+1}}
\sum_{i_3',\ldots, i_{2k+1}'} | B_{i_2 i_3} |\cdots|
B_{i_{2k + 1} i_1} || B_{i_2 i_3'} |\cdots|
B_{i_{2k + 1}' i_1} |}\biggr)^{1/2}
\\
& \leq &\biggl({\frac{1}{N^2} \sum_{i,j} | B_{ij} |
^2}\biggr)^{1/2} \cal
C_{4k}^{1/2}
\leq\biggl({\frac{1}{N^2} \sum_{i,j} | B_{ij} |
^2}
\biggr)^{l/2}.
\end{eqnarray*}
This proves (\ref{estimateofsimplepath}) for closed simple paths.
Consider now an open simple path of length $l$,
whose value (\ref{definitionofvalueofgraph}) is
\[
\cal O_l \deq\frac{1}{N^{l+1}} \sum_{i_1,\ldots, i_{l+1}}
| B_{i_1 i_2} |\cdots| B_{i_{l} i_{l+1}} |.
\]
If $l = 2k$ is even, we get
\begin{eqnarray*}
\cal O_l & \leq &\frac{1}{N^{2k + 1}} \biggl({\sum_{i_1,\ldots,
i_{2k+1}} | B_{i_1 i_2} || B_{i_3 i_4} |\cdots
| B_{i_{2k - 1} i_{2k}} |}\biggr)^{1/2}\\
&&{}\times \biggl({\sum_{i_1,\ldots, i_{2k+1}}
| B_{i_2 i_3} || B_{i_4 i_5} |
\cdots| B_{i_{2k} i_{2k + 1}} |}\biggr)^{1/2}
\\
& \leq &\biggl({\frac{1}{N^2} \sum_{i,j} | B_{ij} |
^2}
\biggr)^{l/2}.
\end{eqnarray*}
Finally, if $l = 2k + 1$ is odd, we find
\begin{eqnarray*}
\cal O_l & \leq &\frac{1}{N^{2k + 2}} \sum_{i_1, i_2} | B_{i_1
i_2} |\biggl({\sum_{i_3,\ldots, i_{2k + 2}} | B_{i_2 i_3}
|\cdots| B_{i_{2k + 1} i_{2k + 2}} |}\biggr)
\\
& \leq &\frac{1}{N^{2k + 2}} \biggl({\sum_{i_1, i_2} | B_{i_1
i_2} |^2}\biggr)^{1/2}\\
&&{}\times \biggl({\sum_{i_1,\ldots, i_{2k + 2}}
\sum_{i_3',\ldots, i_{2k + 2}'} | B_{i_2 i_3} |^2 \cdots
| B_{i_{2k + 1} i_{2k + 2}} |^2 | B_{i_2 i_3'} |^2
\cdots| B_{i_{2k + 1}' i_{2k + 2}'} |^2 }\biggr)^{1/2}
\\
& \leq &\biggl({\frac{1}{N^2} \sum_{i,j} | B_{ij} |
^2}
\biggr)^{1/2} \cal O_{4k}^{1/2} \leq
\biggl({\frac{1}{N^2} \sum_{i,j} | B_{ij} |^2}
\biggr)^{l/2}.
\end{eqnarray*}
This concludes the proof of (\ref{estimateofsimplepath}).

Thus, we get from (\ref{boundoncompletemultigraph}) that the
contribution to the right-hand side of (\ref{momentofaaBo}) of all
label configurations $\f i$ satisfying $\Gamma(\f i) =
\Gamma$ is bounded in absolute value by
\[
C^r \biggl({\frac{B_o}{q} + \biggl({\frac{1}{N^2} \sum_{i \neq j}
| B_{ij} |^2}\biggr)^{1/2}}\biggr)^{2r}.
\]
In order to conclude the proof of (\ref{aaBoapp}), we need a
combinatorial bound on the number of multigraphs of the
above type containing $2r$ edges, as well as on the number of
partitions $\Gamma$ associated with any given multigraph
$G$. Their product is easily seen to be bounded by $(Cr)^{4r}$. This
completes the proof of (\ref{aaBoapp}).
\end{pf}
\begin{pf*}{Proof of Lemma~\ref{lemma_lde}}
The proof is a simple application of Lemma~\ref{lemmamomentbounds}
and Markov's inequality.

In order to prove (i), we choose $p = \nu(\log N)^\xi$ in (\ref
{generalizedLDEapp}) and apply a high moment Markov
inequality.

Next, we prove (ii).
The bound (\ref{aA}) follows immediately from (i) by setting $\alpha=
1$, $\beta= -2$ and $\gamma= 1$. Similarly,
the bound (\ref{aaBd}) follows easily by applying (i) to the random
variables $| a_i |^2 - \sigma^2_i$ and setting $A_i
= B_{ii}$; here $\alpha= 2$, $\beta= -2$ and $\gamma= 1$, as can be
easily seen using (\ref{boundonmomentsofsparseentries}).
Moreover, the claim (\ref{aaBo}) follows by setting $p = \nu(\log
N)^\xi$ in (\ref{aaBoapp}) and applying a high moment
Markov inequality.

Finally, we prove (iii). Write
\[
\biggl|\sum_{i,j} a_i B_{ij} b_j \biggr|\leq
\biggl|\sum_i a_i B_{ii} b_i \biggr|+ \biggl|\sum_{i \neq
j} a_i B_{ij} b_j \biggr|.
\]
The first term is dealt with by noting that the random variables $a_1
b_1,\ldots,\break a_N b_N$ are independent and satisfy
(\ref{generalizedmomentcondition}) for\vadjust{\goodbreak} $\alpha= 2$, $\beta= -4$
and $\gamma= 2$. Therefore, (\ref{generalizedLDE}) yields with
$(\xi,\nu)$-high probability
\begin{eqnarray*}
\biggl|\sum_i a_i B_{ii} b_i \biggr|&\leq&(\log
N)^\xi\biggl[{\frac{B_d}{q^2} + \biggl({\frac{1}{N^2} \sum_i
| B_{ii} |^2}\biggr)^{1/2}}\biggr] \\
&\leq&2 (\log N)^\xi
\frac{B_d}{q^2}.
\end{eqnarray*}

In order to bound the off-diagonal terms, we set $A_i \deq\sum_{j
\neq i} B_{ij} b_j$. Then we may again apply
(\ref{generalizedLDE}) to get with $(\xi,\nu)$-high probability
\[
| A_i |\leq(\log N)^\xi\biggl[{\frac{B_o}{q} +
\biggl({\frac{1}{N} \sum_{j \neq i} | B_{ij} |^2}
\biggr)^{1/2}}\biggr].
\]
Since $A_i$ is independent of $a_j$, we therefore get from (\ref
{generalizedLDE})
\begin{eqnarray*}
\biggl|\sum_{i \neq j} a_i B_{ij} b_j \biggr|& =&
\biggl|\sum_i A_i a_i \biggr|
\\
& \leq &(\log N)^\xi\biggl[{\frac{\max_i | A_i |}{q} +
\biggl({\frac{1}{N} \sum_i | A_i |^2}\biggr)^{1/2}}\biggr]
\\
& \leq & C (\log N)^{2 \xi} \biggl[{\frac{B_o}{q} +
\biggl({\frac{1}{N^2} \sum_{i \neq j} | B_{ij} |^2}
\biggr)^{1/2}}\biggr]
\end{eqnarray*}
with $(\xi,\nu)$-high probability.
We remark finally that the constant $C$ may be absorbed into the small
constant $\nu$ when applying the high moment
Markov inequality used to prove (\ref{generalizedLDE}).
\end{pf*}
\begin{pf*}{Proof of Lemma~\ref{lemmae_he}}
To prove (\ref{eHe}) for $k = 1$, we estimate with $(\xi,\nu)$-high
probability
\[
|\langle{\f e}, {H \f e}\rangle|=
\biggl|
\frac{1}{N} \sum_{i,j} h_{ij} \biggr|= O({(\log
N)^\xi N^{-1/2}}),
\]
where we invoked (\ref{momentconditions}) and applied (\ref
{generalizedLDE}) to the $O(N^2)$ variables $\{h_{ij}
\col i < j\}$ (and similarly for $i \geq j$) with $\alpha= 1$, $\beta
= -2$ and $\gamma= 1$.

If $k \geq2$, we use a high moment expansion. The following notation
will prove helpful. We abbreviate $\alpha= (i,j)$
and write $h_\alpha\deq h_{ij}$. Defining
\[
B_{(i,j) (k,l)} \deq\delta_{jk},
\]
we may thus write
%
%
\begin{equation} \label{Bhhstep}\quad
\IE\langle{\f e}, {H^k \f e}\rangle= \frac{1}{N} \sum
_{\alpha_1,\ldots, \alpha_k} B_{\alpha_1 \alpha_2} B_{\alpha_2
\alpha_3} \cdots B_{\alpha_{k - 1} \alpha_k} \IE
({h_{\alpha_1} \cdots h_{\alpha_k}}),
\end{equation}
where $\IE(\cdot) \deq(\cdot) - \E(\cdot)$.
In order to make all matrix entries independent of each other, we split
$H = H' + H''$ into two triangular
matrices, where
\[
h'_{ij} \deq h_{ij} \f1 (i \leq j),\qquad h_{ij}'' \deq
h_{ij} \f1 (i > j).
\]
This results in a splitting of (\ref{Bhhstep}) into $2^k$ terms, of
which we only consider
\[
X_k \deq\frac{1}{N} \sum_{\alpha_1,\ldots, \alpha_k}
B_{\alpha_1 \alpha_2} B_{\alpha_2 \alpha_3} \cdots
B_{\alpha_{k - 1} \alpha_k} \IE({h'_{\alpha_1} \cdots
h'_{\alpha_k}})
\]
(the other terms are dealt with in exactly the same manner and the
resulting factor $2^k$ is immaterial).

We abbreviate $\ff{\alpha}= (\alpha_1,\ldots, \alpha_k)$ and write
\[
X_k = \sum_{\ff{\alpha}} ({\zeta_{\ff{\alpha}} - \E
\zeta
_{\ff{\alpha}}}),
\]
where we defined
\[
\zeta_{\ff{\alpha}} \deq\frac{1}{N} B_{\alpha_1 \alpha_2}
B_{\alpha_2 \alpha_3} \cdots B_{\alpha_{k - 1} \alpha_k}
h'_{\alpha_1} \cdots h'_{\alpha_k}.
\]
For even $p \in\N$ we get therefore
%
%
\begin{equation} \label{Xpdef}
\E X_k^p = \sum_{\ff{\alpha}^1,\ldots, \ff{\alpha}^p} \E
[{(\zeta_{\ff{\alpha}^1} - \E\zeta_{\ff{\alpha}^1})
\cdots(\zeta_{\ff{\alpha}^p} - \E\zeta_{\ff{\alpha}^p})}
].
\end{equation}
By independence of the family $\{h'_\alpha\}$, we find that a summand
in (\ref{Xpdef}) indexed by $\ff{\alpha}$
vanishes if there is an $r$ such that $[\ff{\alpha}^r] \cap[\ff
{\alpha}
^{r'}] = \varnothing$ for all $r' \neq r$. Here
$[(\alpha_1,\ldots, \alpha_k)] \deq\{\alpha_1,\ldots, \alpha_k\}
$. Thus, we find
%
%
\begin{equation} \label{Xpdef2}\quad
\E X_k^p = \sum_{\ff{\alpha}^1,\ldots, \ff{\alpha}^p} \E
[{(\zeta_{\ff{\alpha}^1} - \E\zeta_{\ff{\alpha}^1})
\cdots(\zeta_{\ff{\alpha}^p} - \E\zeta_{\ff{\alpha}^p})}
] \chi(\ff
{\alpha}^1,\ldots, \ff{\alpha}^p),
\end{equation}
where
\[
\chi(\ff{\alpha}^1,\ldots, \ff{\alpha}^p) \deq\prod_{r = 1}^p
\f1 ({\exists r' \col[\ff{\alpha}^r] \cap[\ff{\alpha
}^{r'}] \neq\varnothing}).
\]

For each given label configuration $\ff{\alpha}= (\ff{\alpha}^r) =
(\alpha^r_l)$, we define a partition $\Gamma(\ff{\alpha})$
of the index set $\{(r,l) \col r = 1,\ldots, p, l = 1,\ldots, k\}$
by imposing that $(r,l)$ and $(r',l')$ are in the
same equivalence class of $\Gamma(\ff{\alpha})$ if and only if
$\alpha
_{rl} = \alpha_{r'l'}$. We now perform the sum over
$\ff{\alpha}$ in (\ref{Xpdef2}) by first specifying a partition
$\Gamma
$ and summing over all $\ff{\alpha}$ satisfying
$\Gamma= \Gamma(\ff{\alpha})$, and then summing over all partitions
$\Gamma$. Note that any partition $\Gamma$ yielding a
nonzero contribution to (\ref{Xpdef2}) satisfies the two following conditions:
\begin{longlist}
\item
Each equivalence class of $\Gamma$ contains at least two elements.
\item
For each $r = 1,\ldots, p$ there are $r' = 1,\ldots, p$ and $l,l' =
1,\ldots, k$ such that $(r,l)$ and $(r',l')$ are in
the same equivalence class of $\Gamma$.
\end{longlist}
Condition (i) follows from the fact that $h'_{\alpha}$ is centered,
and condition (ii) from the definition of $\chi$.\vadjust{\goodbreak}

Let us fix a partition $\Gamma$ satisfying (i) and (ii). Its
contribution to (\ref{Xpdef2}) is
%
%
\begin{eqnarray} \label{Xpdef3}
&&\biggl|\sum_{\ff{\alpha}\col\Gamma(\ff{\alpha}) = \Gamma}
\E[{(\zeta_{\ff{\alpha}^1} - \E\zeta_{\ff{\alpha}^1})
\cdots(\zeta_{\ff{\alpha}_p} - \E\zeta_{\ff{\alpha}^p})}
] \chi(\ff{\alpha}) \biggr|\nonumber\\[-8pt]\\[-8pt]
&&\qquad\leq
\sum_{\ff{\alpha}\col\Gamma(\ff{\alpha}) = \Gamma} \E
[{(|\zeta_{\ff{\alpha}^1} |+ \E|\zeta_{\ff{\alpha
}^1} |) \cdots(|\zeta_{\ff{\alpha}^p} |+ \E|
\zeta_{\ff{\alpha}^p} |)}]
\chi
(\ff{\alpha}).\nonumber
\end{eqnarray}

Next, we note that $\Gamma$ gives rise to a multigraph $G \equiv
G(\Gamma)$ defined as follows. The vertex set $V(G)$ is
given by the equivalence classes of $\Gamma$. Each pair $\{{(r,l),
(r,l+1)}\}$, $l = 1,\ldots, k - 1$, gives rise to an
edge that connects the vertices $\gamma\ni(r,l)$ and $\gamma' \ni
(r,l+1)$. Thus, the set of edges $E(G)$ of $G$
contains $p(k - 1)$ edges. The interpretation of the edges is that each
factor $B_{\alpha_{r,l} \alpha_{r,l+1}}$ on the
right-hand side of (\ref{Xpdef3}) is represented with an edge.

The expectation on the right-hand side of the identity
\begin{eqnarray*}
&&\E[{(|\zeta_{\ff{\alpha}^1} |+ \E|\zeta
_{\ff{\alpha} ^1} |) \cdots(|\zeta_{\ff{\alpha}^p}
|+ \E|\zeta_{\ff{\alpha}^p} |)}] \\
&&\qquad= \frac
{1}{N^p} \Biggl[{\prod_{l = 1}^{k - 1} B_{\alpha_l^r \alpha
_{l+1}^r}}\Biggr] \E\Biggl[{\prod_{r = 1}^p \Biggl({\prod_{l =
1}^p | h'_{\alpha^r_l} |+ \E\prod_{l = 1}^p |
h'_{\alpha^r_l} |}\Biggr)}\Biggr]
\end{eqnarray*}
is bounded by
\[
2^p \prod_{\gamma\in V(G)} \frac{C^{|\gamma|}}{N
q^{|\gamma|- 2}},
\]
where $|\gamma|$ denotes the size of the equivalence class
$\gamma$; this is a simple consequence of (\ref{momentconditions})
and the constraint $q \leq C N^{1/2}$.
By construction of $G$, each vertex $\gamma$ of $G$ carries a label
$\alpha_\gamma$. Thus, we may bound (\ref{Xpdef3}) by
\[
\frac{2^p}{N^p} \sum_{\alpha_1,\ldots, \alpha_v} \biggl[{\prod
_{\{ \gamma,\gamma'\} \in E(G)} B_{\alpha_{\gamma} \alpha_{\gamma
'}}}\biggr] \prod_{\gamma\in V(G)} \frac{C^{|\gamma|
}}{N q^{|\gamma|- 2}},
\]
where $v = | V(G) |$ denotes the number of vertices of $G$.
Here we
dropped the factor $\chi$, and the restriction
that $\alpha_1,\ldots, \alpha_v$ be distinct, to obtain an upper bound.
By property (i) above, we have $|\gamma|- 2 \geq0$ and we get
the bound
%
%
\begin{equation} \label{Xpdef5}
\frac{2^p C^{pk}}{N^{p + v}} \sum_{\alpha_1,\ldots, \alpha_v}
\prod_{\{\gamma,\gamma'\} \in E(G)} B_{\alpha_{\gamma}
\alpha_{\gamma'}}.
\end{equation}

Next, we split $G = G_1 \cup\cdots\cup G_l$ into its connected
components; here $l$ denotes the number of connected
components. An immediate consequence of the property (ii) of $\Gamma$
is the bound
%
%
\begin{equation} \label{lleqp2}
l \leq p/2.
\end{equation}
Thus, (\ref{Xpdef5}) becomes
%
%
\begin{equation} \label{connestimate}
\frac{C^{pk}}{N^{p+v}} \prod_{j = 1}^l \biggl[{\sum_{\alpha_1,\ldots, \alpha
_{v_j}} \prod_{\{\gamma,\gamma'\} \in E(G_j)}
B_{\alpha_{\gamma} \alpha_{\gamma'}}}\biggr],
\end{equation}
where $v_j = | V(G_j) |$ denotes the number of vertices in $G_j$.

In order to estimate the contribution of the $j$th connected component,
we pick a root $r_j \in V(G_j)$ and a spanning
tree $T_j$ of $G_j$. First, we use the trivial bound $B_{\alpha_\gamma
\alpha_{\gamma'}} \leq1$ for edges that do not
belong to $T_j$. Second, we sum over all of the $v_j - 1$ nonroot
labels $\alpha_\gamma$, starting from the leaves of
$T_j$, and using the identity
\[
\sum_{\alpha_{\gamma'}} B_{\alpha_\gamma\alpha_{\gamma'}} = N
\]
at each step. Third, we sum over the root label $\gamma_{r_j}$, which
yields a factor bounded by $N^2$. Putting
everything together yields
\[
\sum_{\alpha_1,\ldots, \alpha_{v_j}} \prod_{\{\gamma,\gamma'\}
\in E(G_j)} B_{\alpha_{\gamma} \alpha_{\gamma'}} \leq
N^{v_j + 1}.
\]
Returning to (\ref{connestimate}), we thus find that the right-hand
side of (\ref{Xpdef3}) is bounded by
\[
\frac{C^{pk}}{N^{p+v}} N^{v + l} \leq\frac{C^{pk}}{N^{p/2}},
\]
where we used (\ref{lleqp2}).

Since the number of partitions $\Gamma$ is bounded by $(kp)^{kp}$, we
get the bound
\[
\E X_k^p \leq\biggl({\frac{(C kp)^k}{N^{1/2}}}\biggr)^p.
\]
Choosing $p = \frac{1}{2Ck} (\log N)^\xi$ and applying a high moment
Markov inequality completes the proof.
\end{pf*}
\begin{pf*}{Proof of Lemma~\ref{lemmapinned_lde}}
The proof is similar to (in fact, considerably simpler than) the proof
of Lemma~\ref{lemmae_he}. We only sketch the
argument, using the notation of the proof of Lemma~\ref{lemmae_he}
without further comment. Write
\[
X_k \deq\sum_{i_1,\ldots, i_k} h_{i i_1} h_{i_1 i_2} \cdots
h_{i_{k - 1} i_k} = \sum_{\alpha_1,\ldots, \alpha_k} B_{\alpha_0 \alpha
_1} B_{\alpha_1 \alpha_2}
\cdots B_{\alpha_{k - 1} \alpha_k} h_{\alpha_1} \cdots
h_{\alpha_k},
\]
where $\alpha_0 \deq(1,i)$. Then, as in the proof of Lemma \ref
{lemmae_he}, we write $\E X_k^p$ as a sum over
partitions $\Gamma$ which give rise to multigraphs $G \equiv G(\Gamma
)$ whose edges are given by the factors $B$ and
whose vertices are given by equivalence classes $\gamma$ of the set $\{
1,\ldots, k\} \times\{1,\ldots, p\}$, to which
has been adjoined a distinguished vertex $\gamma_0$. The vertex
$\gamma_0$ corresponds to the fixed label $\alpha_0$,
and it has degree~$p$. Each multigraph $G$ has $pk$ edges, and is
connected. In this fashion we find that the\vadjust{\goodbreak}
contribution of the multigraph $G$ to $\E X_k^p$ is bounded by
%
%
\begin{equation} \label{exprforconnectedtree}
\sum_{(\alpha_\gamma)_{\gamma\neq\gamma_0}} \biggl[{\prod_{\{
\gamma, \gamma'\} \in E(G)} B_{\alpha_\gamma\alpha_{\gamma
'}}}\biggr] \prod_{\gamma\in V(G)\setminus\{\gamma_0\}} \E
| h_{\alpha_\gamma} |^{|\gamma|},
\end{equation}
where the first sum ranges over families $(\alpha_\gamma)_{\gamma\in
V(G) \setminus\{\gamma_0\}}$ of labels; every
vertex $\gamma\neq\gamma_0$ carries a label $\alpha_\gamma$ which
is summed over. The vertex $\gamma_0$ carries the
label $\alpha_0$ which is fixed.

Since $\E h_\alpha= 0$, it is easy to see that $|\gamma|
\geq2$
for all $\gamma$. Choosing a spanning tree of
the connected graph $G$, one therefore finds that (\ref
{exprforconnectedtree}) is bounded~by
\[
N^{| V(G) |- 1} \prod_{\gamma\in V(G)\setminus\{\gamma_0\}}
\Bigl({\max_\alpha\E| h_{\alpha} |^2}\Bigr) =
1.
\]
Since the number of partitions $\Gamma$ is bounded by $(kp)^{kp}$, we
find $\E X_k^p \leq(C kp)^{kp}$ for $p \leq(\log
N)^\xi$. Choosing $p = \frac{1}{2C k} (\log N)^\xi$ and applying
Markov's inequality completes the proof.
\end{pf*}
\begin{pf*}{Proof of Lemma~\ref{lemmaboundon_h_tildeweak}}
Our proof is a standard application of the moment method, along the
lines of~\cite{EYY}, Lemma 7.2.

In a first step, we truncate the entries $h_{ij}$. Let $C_1 = C$ be a
constant for which Lemma~\ref{lemmaOmegah}
holds. Define
\[
\mu_{ij} \deq\E h_{ij} \f1 (| h_{ij} |\leq C_1
q^{-1}).
\]
Choose an independent family $(X_{ij})$ of random variables,
independent of $H$, such that
\[
\P(X_{ij} = q^{-1}) = \mu_{ij} q,\qquad \P(X_{ij} = 0) =
1 - \mu_{ij} q.
\]
Now set
\[
\wh h_{ij} \deq h_{ij} \f1 (| h_{ij} |\leq C_1 q^{-1}) -
X_{ij}.
\]
It is easy to see that $|\mu_{ij} |\leq\me^{-\nu(\log
N)^\xi
}$ and, therefore,
%
%
\begin{equation} \label{htildehhat}
\P(h_{ij} \neq\wh h_{ij}) \leq\me^{-\nu(\log N)^\xi}.
\end{equation}
Moreover, we have
%
%
\begin{equation} \label{propertiesofhhat}\quad
\E\wh h_{ij} = 0,\qquad |\wh h_{ij} |\leq\frac{C_1
+ 1}{q},\qquad \E|\wh h_{ij} |^2 \leq
\frac{1}{N}
\bigl(1 + \me^{-\nu(\log N)^\xi}\bigr).
\end{equation}

By (\ref{htildehhat}), it suffices to prove that $\|\wh H
\|
\leq2 + (\log N)^\xi q^{-1/2}$ with $(\xi,\nu)$-high probability.
We shall
prove that, for even $k \leq c \sqrt{q}$, we have
%
%
\begin{equation} \label{Nk2k}
|\E\tr\wh H^k |\leq3 N k
2^k.
\end{equation}
In order to prove (\ref{Nk2k}), we write
%
%
\begin{equation} \label{EtrHk}
\E\tr\wh H^k = \sum_{i_1,\ldots, i_k} \E\wh h_{i_1 i_2}
\cdots\wh h_{i_{k - 1} i_k} \wh h_{i_k i_1}\vadjust{\goodbreak}
\end{equation}
and apply a
graphical expansion to the right-hand side. Before giving its precise
definition, we outline how it arises from
(\ref{EtrHk}). Let the label configuration $i_1,\ldots, i_k$ be
fixed. We represent each index $j = 1,\ldots, k$ by
a vertex $[j]$, whereby two indices $j$ and $j'$ correspond to the same
vertex if their labels agree, $i_j = i_{j'}$.
Let $p$ be the number of vertices. We then construct a closed walk
through the sequence of edges $([1],[2]), ([2],[3]),\ldots,
([k],[1])$. The walk has $k$ steps. We name the $p$ vertices
$1,\ldots, p$, whereby vertex $v$ is reached after all vertices
$1,\ldots, v - 1$. Since $\E\wh h_{ij} = 0$, it is easy
to see from (\ref{EtrHk}) that each edge of the walk must appear at
least twice.

We may now give a precise definition of such walks. Let $\f w =
(w_1,\ldots, w_k)$ be a sequence with $w_v \in\{1,\ldots, p\}$. With
$\f w$ we associate a multigraph $G(\f w)$ as
follows. The vertex set of $G(\f w)$ is $\{1,\ldots,
p\}$; the edge set of $G(\f w)$ is given by the undirected edges $\{
w_1, w_2\},\ldots, \{w_{k - 1}, w_k\}, \{w_k,
w_1\}$. [Note that $G(\f w)$ may contain multiple edges as well as
loops.] We say that $\f w$ is an \textit{ordered closed
walk} of length $k$ on $p$
vertices if:
\begin{longlist}
\item
A vertex that is visited for the first time at time $j$ is greater than
all vertices visited before time $j\dvtx
\max_{j' \leq j} w_{j'} \leq\max_{j' < j} w_{j'} + 1$.
\item
All vertices are visited: $\{w_1,\ldots, w_k\} = \{1,\ldots, p\}$.
\item
Every edge of $G(\f w)$ appears at least twice.
\end{longlist}

Let $\cal W(k,p)$ denote the set of ordered closed walks of length $k$
on $p$ vertices. The key combinatorial estimate
of our proof is the bound
\[
|\cal W(k,p) |\leq\pmatrix{k\cr2p - 2} p^{2(k - 2p + 2)}
2^{2p - 2},
\]
proved in~\cite{vu}. Using the notion of ordered closed walks, it is
not hard to see that (\ref{EtrHk}) may be rewritten
as
%
%
\begin{equation} \label{trHintermsofwalks}
\E\tr\wh H^k = \sum_{p = 1}^{k/2 + 1} \sum_{\f w \in\cal
W(k,p)} \sum_{\ff{\ell}\in\cal L(p)} \E\wh h_{\ell(w_1)
\ell(w_2)} \cdots\wh h_{\ell(w_{k - 1}) \ell(w_k)} \wh h_{\ell
(w_k) \ell(w_1)},\hspace*{-40pt}
\end{equation}
where $\cal L(p)$ is the set of all $p$-tuples $\ff{\ell}= (\ell
(1),\ldots, \ell(p)) \in\{1,\ldots, N\}^p$ whose
components are disjoint. See~\cite{EYY}, Section 7.1, for a detailed proof.

Next, associate with the multigraph $G(\f w)$ its \textit{skeleton}
$S(\f w)$, obtained from $G(\f w)$ by discarding the
multiplicity of every edge (i.e., by successively
removing edges until it has no multiple edges). For $e \in E(S(\f w))$,
we denote by $\nu(e)$ the multiplicity of the
edge $e$ in $G(\f w)$. We have the obvious relation $\sum_{e \in
E(S(\f w))} \nu(e) = k$. If $e = \{v,v'\}$, we write
$\ell(e) \deq(\ell(v), \ell(v'))$ [the chosen order of the pair
$\ell(e)$ is immaterial]. Then it is easy to see that
\begin{eqnarray*}
&&\bigl|\E\wh h_{\ell(w_1) \ell(w_2)} \cdots\wh h_{\ell(w_{k -
1}) \ell(w_k)} \wh h_{\ell(w_k) \ell(w_1)} \bigr|\\
&&\qquad \leq
\prod_{e \in E(S(\f w))} \E
\bigl|\wh h_{\ell(e)} \bigr|^{\nu(e)}
\\
&&\qquad \leq
\prod_{e \in E(S(\f w))} \frac{1}{N}
\bigl(1 + \me^{-\nu(\log N)^\xi}\bigr) \frac{C^{\nu(e) - 2}}{q^{\nu(e) - 2}}
\\
&&\qquad \leq \biggl[{\frac{q^2}{C^2 N} \bigl(1 + \me^{-\nu(\log N)^\xi
}\bigr)}\biggr]^{E_S} \biggl({\frac{C}{q}}\biggr)^k,
\end{eqnarray*}
where we used (\ref{propertiesofhhat}) and introduced the shorthand
$E_S \deq| E(S(\f w)) |$. Therefore, summing
over $\ell\in\cal L(p)$ in (\ref{trHintermsofwalks}) yields
\[
|\E\tr\wh H^k |\leq\sum_{p = 1}^{k/2 +
1} \sum_{\f w \in\cal W(k,p)} N^p \biggl[{\frac{q^2}{C^2 N} \bigl(1 +
\me^{-c (\log N)^\xi}\bigr)}\biggr]^{E_S} \biggl({\frac{C}{q}}
\biggr)^k.
\]

Next, it is immediate that we have the relations $p - 1 \leq E_S \leq
k/2$; these inequalities follow from the
above properties (ii) and (iii), respectively. Since $d^2/N \ll1$, we
therefore get
\[
|\E\tr\wh H^k |\leq N \sum_{p = 1}^{k/2
+ 1} |\cal W(k,p) |\bigl(1 + \me^{-c
(\log N)^\xi}\bigr)^p \biggl({\frac{C}{q}}\biggr)^{k - 2p + 2}.
\]
For $k \leq N$ this yields
\[
|\E\tr\wh H^k |\leq3 N \sum_{p =
1}^{k/2 + 1} S(k,p),
\]
where
\[
S(k,p) \deq\pmatrix{k\cr2p - 2} p^{2(k - 2p + 2)} 2^{2p - 2}
\biggl({\frac{C}{q}}\biggr)^{k - 2p + 2}.
\]
It is elementary to check that $S(k, k/2 + 1) = 2^k$ and
\[
S(k,p) \leq\frac{k^4}{8} \biggl({\frac{C}{q}}\biggr)^2
S(k,p+1).
\]
Therefore, choosing $k \leq c \sqrt{q}$ implies $S(k,p) \leq2^k$.
This concludes the proof of~(\ref{Nk2k}).

The claim now follows by setting $k = c \sqrt{q}$ with a sufficiently
small constant~$c$, applying a high moment Markov
inequality and recalling that
$\sqrt{q} \gg(\log N)^\xi$ by~(\ref{lowerboundond}).
\end{pf*}
\end{appendix}
%



\printaddresses

\end{document}